\documentclass[11pt, a4paper]{amsart}
\usepackage{amsmath}
\usepackage{amssymb}
\usepackage{amsfonts}
\usepackage{amscd}
\usepackage{amsthm}
\usepackage{mathrsfs}
\usepackage{enumerate}
\usepackage[numbers]{natbib}

\usepackage[refpage]{nomencl}
\makenomenclature

\usepackage[english,french]{babel}

\calclayout

\newcommand{\CC}{{\rm\bf C}}
\newcommand{\RR}{{\rm\bf R}}
\newcommand{\QQ}{{\rm\bf Q}}
\newcommand{\ZZ}{{\rm\bf Z}}

\newcommand{\Adeles}{{\rm\bf A}}

\newcommand{\OO}{\mathcal {O}}

\DeclareMathOperator{\absNorm}{\mathfrak{N}}

\DeclareMathOperator{\Spec}{\mathrm{Spec}}
\DeclareMathOperator{\diag}{\mathrm{diag}}

\DeclareMathOperator{\GL}{\mathrm{GL}}
\DeclareMathOperator{\SL}{\mathrm {SL}}

\DeclareMathOperator{\End}{\mathrm {End}}
\DeclareMathOperator{\Hom}{\mathrm {Hom}}
\DeclareMathOperator{\Sym}{\mathrm{Sym}}
\DeclareMathOperator{\kernel}{\mathrm {ker}}

\DeclareMathOperator{\rank}{\mathrm {rk}}

\DeclareMathOperator{\Gal}{\mathrm {Gal}}

\DeclareMathOperator{\Ind}{\mathrm{Ind}}
\DeclareMathOperator{\res}{\mathrm{res}}

\DeclareMathOperator{\sgn}{\mathrm{sgn}}

\DeclareMathOperator{\coker}{\mathrm{coker}}

\newcommand{\lieb}{{\mathfrak {b}}}

\newcommand{\lieg}{{\mathfrak {g}}}
\newcommand{\liegk}{{\mathfrak {gk}}}
\newcommand{\lieh}{{\mathfrak {h}}}
\newcommand{\liek}{{\mathfrak {k}}}
\newcommand{\liel}{{\mathfrak {l}}}

\newcommand{\lieu}{{\mathfrak {u}}}

\newcommand{\liez}{{\mathfrak {z}}}

\DeclareMathOperator{\vol}{\mathrm {vol}}

\newcommand{\absnorm}[1]{\ensuremath{\left|{#1}\right|}}

\newcommand{\adots}{\ensuremath{
\put(-3,-6){\normalsize$\,$}
\put(-2.2,-3.5){\normalsize$\cdot$}
\cdot
\put(-1.0,3.5){\normalsize$\cdot$}
\put(0.0,6){\normalsize$\,$}
}}

\newcommand{\nocontentsline}[3]{}
\newcommand{\tocless}[2]{\bgroup\let\addcontentsline=\nocontentsline#1{#2}\egroup}

\theoremstyle{Theorem}
\newtheorem{introconjecture}{Conjecture}

\newtheorem{introtheorem}[introconjecture]{Theorem}

\theoremstyle{plain}
\newtheorem{theorem}{Theorem}[section]
\newtheorem{lemma}[theorem]{Lemma}
\newtheorem{corollary}[theorem]{Corollary}
\newtheorem{proposition}[theorem]{Proposition}
\newtheorem{conjecture}[theorem]{Conjecture}

\theoremstyle{remark}
\newtheorem{remark}[theorem]{Remark}
\newtheorem{definition}[theorem]{Definition}

\begin{document}

\title{Non-abelian $p$-adic Rankin-Selberg $L$-functions\\and\\non-vanishing of central $L$-values}
\author{Fabian Januszewski}

\address{Institut f\"ur Mathematik, Fakult\"at EIM, Universit\"at Paderborn, Germany}
\email{fabian.januszewski@upb.de}

\selectlanguage{english}
\maketitle

{\centering\footnotesize Dedicated to the memory of Claus-G\"unther Schmidt.\par}

  \begin{abstract}
\selectlanguage{english}
We construct $p$-adic $L$-functions for torsion classes for $\GL(n+1)\times\GL(n)$ and along the way prove new congruences between special values of Rankin-Selberg $L$-functions for $\GL(n+1)\times\GL(n)$ over arbitrary number fields. This allows us to control the behavior of $p$-adic $L$-functions under Tate twists and to prove the existence of non-abelian $p$-adic $L$-functions for Hida families on $\GL(n+1)\times\GL(n)$. As an application, we establish generic non-vanishing results for central $L$-values: We give sufficient local conditions for twisted central Rankin-Selberg $L$-values to be generically non-zero.
\end{abstract}

{\vspace*{-0.6em}\footnotesize\tableofcontents}

\markboth{Fabian Januszewski}{Non-abelian $p$-adic $L$-functions and non-vanishing of central $L$-values}

\section*{Introduction}\label{sec:introduction}

Fix a rational prime $p,$ a number field $F/\QQ$ and define the reductive group $$
G\;=\;\res_{F/\QQ}\GL(n+1)\times\GL(n),\quad n\geq 1.
$$
Consider the Rankin-Selberg $L$-function $L(s,\Pi\widehat{\otimes}\Sigma)$ attached to an irreducible cuspidal automorphic representation $\Pi\widehat{\otimes}\Sigma$ of $G(\Adeles)$ as studied by Jacquet, Piatetski-Shapiro and Shalika \citep{jpss1979a,jpss1979b,jpss1983}.

If $\Pi$ and $\Sigma$ are regular algebraic in the sense of Clozel \citep{clozel1990}, we expect $L(s,\Pi\widehat{\otimes}\Sigma)$ to agree (up to shift) with the $L$-function of the tensor product of the conjectural irreducible motives $M_\Pi$ and $M_\Sigma$ attached to $\Pi$ and $\Sigma$. In this context, we expect the special values of $L(s,\Pi\widehat{\otimes}\Sigma)$ to be intricately related to the arithmetic of $M_\Pi$ and $M_\Sigma$. In particular, when deforming $M_\Pi$ and $M_\Sigma$ in $p$-adic families, we expect these special values to vary $p$-adic analytically as well.

The aim of this paper is to establish this expected $p$-adic variation of $L$-values in the case when $\Pi$ and $\Sigma$ are nearly ordinary at $p$ in the sense of \citep{hida1995,hida1998}. In order to do so, we prove new congruences for the special values under consideration.\\\ \\{\bf Abelian $p$-adic interpolation.} 
Our results for abelian deformations are the following. Write $T\subseteq G$ for the standard diagonal maximal torus and $C_F(p^\infty)$ for the ray class group of level $p^\infty$ of $F$. Let $E/\QQ_p$ denote a finite extension which contains the fields of rationality of $\Pi$ and $\Sigma$.

Our first main result is (cf.\ Theorem \ref{thm:interpolation} in the text),

\begin{introtheorem}\label{introthm:firstpadicL}
  Let $\Pi\widehat{\otimes}\Sigma$ be an irreducible regular algebraic cuspidal automorphic representation of $G(\Adeles)$ of cohomological weight $\lambda$. Assume the following:
  \begin{itemize}
    \item[(i)] $\lambda$ is balanced (in the sense of eq.\ \eqref{eq:admissiblelambdacondition}).
    \item[(ii)] $\Pi\widehat{\otimes}\Sigma$ is nearly ordinary at a prime $p$ and $\vartheta:T(\QQ_p)\to\CC^\times$ the corresponding eigenvalue.
  \end{itemize}
  Then there are complex periods $\Omega_{\pm,j}\in\CC^\times$, indexed by the characters of $\pi_0(F\otimes\RR)^\times$ and $j\in\ZZ$ for which $s_0=\frac{1}{2}+j$ is critical for $L(s,\Pi\widehat{\otimes}\Sigma)$, and a unique $p$-adic measure $\mu_{\Pi\widehat{\otimes}\Sigma}\in\OO[[C_F(p^\infty)]]$ with the following property.
  For every $s_0=\frac{1}{2}+j$ critical for $L(s,\Pi\widehat{\otimes}\Sigma)$, 
  for all finite order Hecke characters $\chi$ of $F$ unramified outside $p\infty$ and such that $\chi_p\vartheta$ has fully supported constant conductor,
  \begin{align*}
    &
    \int\limits_{C_F(p^\infty)}\chi(x)\omega_F^j(x)\langle x\rangle_F^jd\mu_{\Pi\widehat{\otimes}\Sigma}(x)\;=\;\\
    &
    \absNorm(\mathfrak{f}_{\chi\vartheta})^{j\frac{(n+1)n}{2}-\frac{(n+1)n(n-1)}{6}}\cdot
    \prod_{\mu=1}^n\prod_{\nu=1}^\mu
      G(\chi\vartheta_{\mu,\nu})
    \cdot
    \frac{L(s_0,\Pi\widehat{\otimes}\Sigma\otimes\chi)}{\Omega_{(-1)^j\sgn\chi,j}}.
  \end{align*}
\end{introtheorem}

Hypothesis (i) is imminent also in previous and related work on modular symbols. A condition of this type is imposed upon us also for motivic reasons (cf.\ Deligne's Conjecture on special values of motivic $L$-functions and its integral refinements; for further details we refer to the discussion following Theorem \ref{thm:interpolation}). We refer the reader to recent work of Claus-G\"unter Schmidt \cite{schmidt2017} for a classification of the possible non-balanced cases whose treatment remains an important open problem.

The notion of {\em fully supported constant conductor} is formally defined in section \ref{sec:birchglobal} in the text. This technical condition is automatically satisfied whenever the ramification of $\chi$ is deep enough (deeper than the Nebentypus' ramification). For example, if $\Pi$ is spherical at all places $\mathfrak{p}\mid p$, then every $\chi$ ramified at all $\mathfrak{p}\mid p$ satisfies this condition.

The set of Hecke characters with fully supported constant conductor is Zariski dense in the $p$-adic rigid space of all $p$-adic characters. In particular, the interpolation property in Theorem \ref{introthm:firstpadicL} characterizes the $p$-adic measure $\mu_{\Pi\widehat{\otimes}\Sigma}(x)$ uniquely and confirms the Conjecture of Coates and Perrin-Riou \citep{coatesperrinriou1989,coates1989} in the cases where the interpolation property is known.

We expect an most identical interpolation property in the case of fully supported but not necessarily constant conductor to hold, whereas in the presence of unramified characters among $\chi\vartheta_{\mu,\nu}$ we should see a modified Euler factor of positive degree as a $p$-adic multiplier as predicted by Coates and Perrin-Riou. For $n=2$, i.e.\ $\GL(3)\times\GL(2)$, this was confirmed in the spherical case in unpublished work of D.\ Ungemach.

By the construction of $\mu_{\Pi\widehat{\otimes}\Sigma}(x)$, Theorem \ref{introthm:firstpadicL} has a straightforward generalization to finite slope forms and yields a unique locally analytic distribution on the cyclotomic line $\ZZ_p^\times$, provided the slope is bounded above by the number of critical values. However, to keep the exposition and notation simple, we emphasize the nearly ordinary case here.

Theorem \ref{introthm:firstpadicL} improves the main results of \citep{schmidt1993,kazhdanmazurschmidt2000,januszewski2011,januszewski2015,januszewski2016} in two ways. Firstly, we cover {\em nearly ordinary} representations and not only Iwahori spherical representations at places $v\mid p$. Secondly, we construct a {\em single} $p$-adic $L$-function interpolating {\em all} critical values at once. In the case $n=1,$ Theorem \ref{introthm:firstpadicL} recovers Namikawa's recent construction of abelian $p$-adic $L$-functions for $\GL(2)$ over number fields from \citep{namikawa2016}. In this case, the interpolation formula is known is all cases (cf.\ loc.\ cit.).

The non-vanishing of the complex periods $\Omega_{\pm,j}$ is an important recent result of Sun \citep{sunjams}. Their dependence on $j$ is studied in \citep{januszewski2018,grobnerlin2017,harderraghuram2017}. However, as of now there is no integral relation between the various $\Omega_{\pm,j}$ known for $n>1$, although our results suggest that such relations should indeed exist (cf.\ Remark \ref{rmk:omegajvariance} in the text).

Among all critical values, central $L$-values are expected to be arithmetically the most interesting ones. The Conjecture of Birch and Swinnerton-Dyer and its generalizations suggest that the central value of a motivic $L$-function should vanish only for specific arithmetic reasons.

From an automorphic perspective, non-vanishing of central values is believed to be equivalent to the existence of non-zero periods. For example, by Ginzburg-Jiang-Rallis' \citep{ginzburgjiangrallis2004}, non-vanishing of the central value of $L(s,\Pi\widehat{\otimes}\Sigma)$ for suitable $\Pi$ and $\Sigma$ is equivalent to non-vanishing of certain period integrals on related groups.

Since $\mu_{\Pi\widehat{\otimes}\Sigma}$ is determined uniquely by the interpolation property for a single critical $s_0=\frac{1}{2}+j$, we deduce from Theorem \ref{introthm:firstpadicL} the following non-vanishing result (Theorem \ref{thm:abeliannonvanishing} and Corollary \ref{cor:abeliannonvanishing} in the text).

\begin{introtheorem}\label{introthm:firstnonvanishing}
  Let $\Pi$ and $\Sigma$ be irreducible cuspidal regular algebraic automorphic representations of $\GL_{n+1}(\Adeles_F)$ and $\GL_n(\Adeles_F)$ of balanced weight. Assume that $\Pi$ and $\Sigma$ are {\em unitary} and that $\Pi$ and $\Sigma$ are nearly ordinary at all $\mathfrak{p}\mid p$ for some prime $p.$

If $s_0=\frac{1}{2}$ is critical for $L(s,\Pi\widehat{\otimes}\Sigma)$ and if there exists a second critical value $s_0\neq\frac{1}{2}$, then
\begin{equation}
  L(\frac{1}{2},\Pi\widehat{\otimes}\Sigma\otimes\chi)\neq 0,
  \label{eq:intrononvanishing}
\end{equation}
generically for $\chi$ varying over all finite order Hecke characters of $F$ unramified outside $p\infty$.

Moreover, the vanishing locus is transversal to the cyclotomic line, i.e.\ \eqref{eq:intrononvanishing} holds for {\em all but finitely many} characters of the form $\chi=\chi'\circ N_{F/\QQ}$ where $\chi'$ is a Dirichlet character of $p$-power conductor.
\end{introtheorem}

We may always pass from general regular algebraic automorphic representations $\Pi$ and $\Sigma$ to unitary twists $\Pi^{\rm u}=\Pi\otimes|\cdot|_\Adeles^{-w}$ and $\Sigma^{\rm u}=\Sigma\otimes|\cdot|_\Adeles^{-w'}$ for suitable $w,w'\in\frac{1}{2}\ZZ$.

The existence and location of critical values and the notion of being of balanced weight is entirely governed by the cohomological weights of $\Pi$ and $\Sigma$, i.e.\ by the infinity types $\Pi_\infty$ and $\Sigma_\infty$. Likewise, being nearly ordinary at places $\mathfrak{p}$ above $p$ is a local condition on $\Pi_{\mathfrak{p}}$ and $\Sigma_{\mathfrak{p}}$ (cf.\ section \ref{sec:abelianpadicL}). Therefore, Theorem \ref{introthm:firstnonvanishing} provides sufficient local conditions on $\Pi$ and $\Sigma$ for generic non-vanishing of twists.

In particular, it is easy to see that if $\Pi$ and $\Sigma$ satisfy the hypotheses of Theorem \ref{introthm:firstnonvanishing}, then so do $\Pi\otimes\chi'$ and $\Sigma\otimes\chi''$ for arbitrary finite order Hecke characters $\chi'$ and $\chi''$ over $F$.

Theorem \ref{introthm:firstnonvanishing} implies {\em simultaneous} generic non-vanishing for any finite collection of representations $\Pi_1,\dots,\Pi_r$, $\Sigma_1,\dots,\Sigma_r$ on $\GL_{n_1+1}(\Adeles_F),\dots,\GL_{n_r+1}(\Adeles_F),\GL_{n_1}(\Adeles_F),\dots,\GL_{n_r}(\Adeles_F),$ such that the pairs $(\Pi_i,\Sigma_i)$ satisfy the hypothesis of Theorem \ref{introthm:firstnonvanishing} for the same prime $p$. We may allow the pairs $(\Pi_i,\Sigma_i)$ to live over different number fields $F_i$, at the cost of obtaining simultaneous non-vanishing only for all but finitely many norm-inflated Dirichlet characters of $p$-power conductor.

Using symmetric power functoriality for $\GL(2)$, it is easy to produce automorphic representations $\Pi$ and $\Sigma$ satisfying the hypotheses of Theorems \ref{introthm:firstpadicL} and \ref{introthm:firstnonvanishing}. By \citep{gelbartjacquet1978,kim2003,kimshahidi2002}, the symmetric power functoriality $\Sym^n$ from $\GL(2)$ to $\GL(n+1)$ is known for $n\leq 4$ over arbitrary number fields. Thanks to recent progress by Clozel-Thorne \citep{clozelthorne2014,clozelthorne2015,clozelthorne2017}, we know that $\Sym^n$ exists for $n\leq 8$ for Hilbert modular forms over totally real fields $F$ under mild hypotheses. Symmetric power lifts preserve near ordinarity and regular algebraicity (cf.\ Theorem 5.3 and Proposition 5.4 in \citep{raghuram2010} and Theorem 3.2 in \citep{raghuram2015}). Using base change in solvable extensions \citep{arthurclozel1989}, we may produce examples over more general number fields.

Iteratively, Theorems \ref{introthm:firstpadicL} and \ref{introthm:firstnonvanishing} imply the existence of $p$-adic meromorphic $L$-functions for symmetric power $L$-functions $L(s,\Sym^n f)$ of non-CM nearly ordinary Hilbert modular cusp forms $f$ over a totally real number field $F/\QQ$, which is linearly disjoint from $\QQ(e^{2\pi i/35}),$ and $1\leq n\leq 8$, provided that $f$ is of sufficiently large parallel weight. Our non-vanishing result also allows for the extension of the rationality results in \citep{raghuram2010,raghuram2015} to central $L$-values.

We refer to \citep{dimitrovjanuszewskiraghurampre} for a detailed discussion of examples and applications and a direct construction of $p$-adic $L$-functions for odd symmetric powers.

The result of Theorem \ref{introthm:firstnonvanishing} appears to be new for $n>1$, the case $n=1$ being to Shimura and Rohrlich \citep{shimura1977,rohrlich1984,rohrlich1989}. The idea to use congruences to deduce non-vanishing of central $L$-values goes back to Greenberg \citep{greenberg1983,greenberg1985} (the author kindly thanks Michael Harris for pointing this out). Non-vanishing is known for non-central critical values by results of Jacquet-Shalika \citep{jacquetshalika1981} and Shahidi \citep{shahidi1981}, generalizing previous results for non-central $L$-values for $\GL(n)$ in \citep{jacquetshalika1976}. At the moment, the most general result for central values known appears to be due to Luo \citep{luo2005} (see also Nastasescu's recent thesis \citep{nastasescu2016}). Independently of our work, Sugiyama and Tsuzuki recently established in \citep{sugiyamatsuzuki2018} a non-vanishing result for central values for $\GL(3)\times\GL(2)$ where the representation on the $\GL(2)$ factor arises from a cuspidal Maa\ss{} form. For yet another analytic approach to non-vanishing for $\GL(2)\times\GL(2),$ and possibly triple products we refer to \citep{vanorderpre}.\\\ \\{\bf Non-abelian $p$-adic interpolation.} Previous work on $p$-adic $L$-functions for $\GL(n+1)\times\GL(n)$ \citep{schmidt1993,kazhdanmazurschmidt2000,januszewski2011,januszewski2015,januszewski2016} as well as Theorem \ref{introthm:firstpadicL} is limited to abelian deformations, i.\,e.\ interpolation of twists by finite order Hecke characters. Furthermore, these methods did not apply to torsion classes.

In order to study the variation of the above $p$-adic $L$-functions in Hida families and to treat torsion classes, we first establish in Theorem \ref{thm:regularcontrol} a control theorem for nearly ordinary cohomology of $G$ following ideas of Hida \citep{hida1986b,hida1986a,hida1988a,hida1989,hida1992,hida1994,hida1995,hida1998} and using delicate results on the contribution of the residual spectrum of $\GL(n)$ to the cohomology of arithmetic groups \citep{jacquetshalika1981,jacquetshalika1981.2,moeglinwaldspurger1989,frankeschwermer1998,scholze2015,harderraghuram2017}.

Assume $F$ totally real, CM or that Conjecture \ref{conj:galoisrepresentations} on the existence of Galois representations for torsion classes holds for $\GL(m)$ over $F$ ($m\leq n+1$) (which is known for $F$ totally real or CM by Scholze's breakthrough \citep{scholze2015}). In this situation, we have a notion of non-Eisenstein maximal ideals $\mathfrak{m}$ in Hida's universal nearly ordinary Hecke algebra ${\bf h}_{\rm ord}(K_{\infty,\infty};\OO)$ for $G,$ which are characterized by corresponding to residual Galois representations $\overline{\rho}_{\mathfrak{m}}$ which are (semi-simplifications of) tensor products of {absolutely irreducible} residual representations for $\GL(n+1)$ and $\GL(n)$ respectively. In particular, the residual representation $\overline{\rho}_{\mathfrak{m}}$ attached to $\mathfrak{m}$ itself may be {\em reducible} (take a tensor product of symmetric powers of the same $2$-dimensional representation for example).

We consider universal nearly ordinary cohomology for $G$
$$
\mathcal H_{\rm ord}^{q_0+l_0}(K_{\infty,\infty};\OO)\;=\;\varprojlim_\alpha H_{\rm ord}^{q_0+l_0}(K_{\alpha,\alpha};\OO/p^\alpha\OO),
$$
which is canonically an ${\bf h}_{\rm ord}(K_{\infty,\infty};\OO)$-module and may contain torsion classes. Its localization at a (contragredient) non-Eistenstein maximal ideal $\mathfrak{m}^\vee$ is expected to be free of rank $1$ over ${\bf h}_{\rm ord}(K_{\infty,\infty};\OO)_{\mathfrak{m}^\vee}$ (cf.\ Conjecture \ref{conj:freeness}; see Theorem 4.9 in \citep{hansenthorne2017} for a partial result towards this conjecture). Independently of this conjecture we show (cf.\ Theorem \ref{thm:nonabelianinterpolation} in the text)

\begin{introtheorem}\label{introthm:secondpadicL}
  Assume $p\nmid (n+1)n$ and let $F/\QQ$ denote a totally real or CM field or otherwise assume the existence of Galois respresentations for torsion classes. Let $\mathfrak{m}$ denote a non-Eisenstein maximal ideal in ${\bf h}_{\rm ord}(K_{\infty,\infty};\OO)$. Then there exists an element
    $$L_{p,\mathfrak{m}}^{\rm univ}\;\in\;{\mathcal H}_{\rm ord}^{q_0+l_0}(K_{\infty,\infty};\OO)_{\mathfrak{m}^\vee}^{\pi_0(H(\RR))}$$
  with the following interpolation property. For every classical point
  $$
    \xi\;\in\;\Spec{\bf h}_{\rm ord}(K_{\infty,\infty};\OO)_{\mathfrak{m}}(\overline{E})
  $$
  of balanced weight $\lambda$ and Nebentypus $\vartheta$, such that $s_0=\frac{1}{2}$ is critical for $L(s,\Pi_\xi\widehat{\otimes}\Sigma_\xi),$ we have
  \begin{align*}
    \Omega_{\xi,p}^{-1}\cdot\xi^\vee(L_{p,\mathfrak{m}}^{\rm univ})
    \;=\;&
    \int_{C_F(p^\infty)}d\mu_{\Pi_\xi\widehat{\otimes}\Sigma_\xi}\\
    \;=\;&
    \absNorm(\mathfrak{f}_{\vartheta})^{\frac{(n+1)n(n-1)}{6}}\cdot
    \prod_{\mu=1}^n\prod_{\nu=1}^\mu
      G(\vartheta_{\mu,\nu})
    \cdot
    \frac{L(\frac{1}{2},\Pi_\xi\widehat{\otimes}\Sigma_\xi)}{\Omega_{\xi}},
  \end{align*}
where the second identity is valid whenever $\vartheta$ has fully supported constant conductor.

  Here $\Omega_{\xi,p}\in\OO[\xi]^\times$ is a $p$-adic period and $\Omega_{\xi}\in\CC^\times$ is a complex period.
\end{introtheorem}

By Theorem \ref{introthm:firstpadicL}, both periods $\Omega_{\xi,p}$ and $\Omega_{\xi}$ may be normalized in such a way that they are invariant under twists of $\Pi_\xi\widehat{\otimes}\Sigma_\xi$ by finite order Hecke characters $\chi$ unramified outside $p$.

Previous results on the variation of special values in Hida families in our setup were limited to the case $n=1$ and $F$ totally real, cf.\ \citep{dimitrov2011}. Even for $\GL(2)$ over a CM field our results are new.

A subtle consequence of Theorem \ref{introthm:secondpadicL} is that the \lq{}big\rq{} $p$-adic $L$-function $L_{p,\mathfrak{m}}^{\rm univ}$ does not see any possible defects to exact control, although we can currently only establish control up to finite error (cf.\ Theorem \ref{thm:regularcontrol}). This is in line with the following expectation.

Current methods of establishing ${\mathcal R}={\mathbb T}$ theorems usually establish the freeness and cyclicity of ${\mathcal H}_{\rm ord}^{q_0+l_0}(K_{\infty,\infty};\OO)_{\mathfrak{m}^\vee}^{\pi_0(H(\RR))}$ over ${\bf h}_{\rm ord}(K_{\infty,\infty};\OO)_{\mathfrak{m}^\vee}$ as a byproduct. In this situation, we may interpret $L_{p,\mathfrak{m}}^{\rm univ}$ as an element of the Hecke algebra, and if ${\mathcal R}={\mathbb T}$ is established for $\mathfrak{m}$, then $L_{p,\mathfrak{m}}^{\rm univ}$ may also be considered as an element of ${\mathcal R}$ (remark that unless the residual Galois representation $\overline{\rho}_{\mathfrak{m}}$ is absolutely irreducible, we cannot define $\mathcal R$ as a classical deformation ring). Assuming this freeness, we establish with Theorem \ref{thm:nonabelianinterpolation2} a refinement of Theorem \ref{introthm:secondpadicL}, which is in line with general expectations on the existence and properties of genuine $p$-adic $L$-functions (cf.\ Conjecture 4.2.1 and Question 4.4.1 in \citep{hida1996}).

As an application of Theorem \ref{introthm:secondpadicL}, we extend Theorem \ref{introthm:firstnonvanishing} to cases where the central value is the only critical value (cf.\ Corollary \ref{cor:nonabeliannonvanishing2} in the text).

\begin{introtheorem}\label{introthm:secondnonvanishing}
  Let $p\nmid (n+1)n$, $F/\QQ$ totally real or CM, or assume that Conjecture \ref{conj:galoisrepresentations} holds for $F$. Let $\mathfrak{m}$ denote a non-Eisenstein maximal ideal in ${\bf h}_{\rm ord}(K_{\infty,\infty};\OO)$.

  Assume that $\mathcal X\subseteq\Spec{\bf h}_{\rm ord}(K_{\infty,\infty};\OO)_{\mathfrak {m}}(\overline{E})$ is an irreducible component containing a classical point $\xi$ of balanced weight such that $L(s,\Pi_\xi\widehat{\otimes}\Sigma_\xi)$ admits at least two critical values.

  Assume furthermore that the classical points in $\mathcal X$ of balanced cohomological weight $\lambda=\lambda_{n+1}\otimes\lambda_n,$ are dense and that their $L$-functions admit $s_0=\frac{1}{2}$ as the unique critical value.

 Then there exist irreducible regular algebraic cuspidal automorphic representations $\Pi'$ and $\Sigma'$ of cohomological weights $\lambda_{n+1}$ and $\lambda_n$ contributing to $\mathcal X$ satisfying
  \begin{equation}
    L(\frac{1}{2},\Pi'\widehat{\otimes}\Sigma')\neq 0.
    \label{eq:intrononvanishing2}
  \end{equation}
  Furthermore,
  \begin{equation}
    L(\frac{1}{2},\Pi'\widehat{\otimes}\Sigma'\otimes\chi)\neq 0.
    \label{eq:intrononvanishing2.1}
  \end{equation}
  for all but finitely many finite order Hecke characters $\chi$ in the cyclotomic line.
\end{introtheorem}

In fact, \eqref{eq:intrononvanishing2} holds {\em generically} for classical points $\xi$ of weight $\lambda$. By construction, the representation $\Pi'\widehat{\otimes}\Sigma'$ also contributes to the non-Eisenstein component $\Spec{\bf h}_{\rm ord}(K_{\infty,\infty};\OO)_{\mathfrak {m}}$. In particular, it is nearly ordinary at $p$ and unramified outside $p$.

Remark that for $F$ not totally real, we know that there are components which do not contain a Zariski dense subset of classical points, cf.\ \citep{calegarimazur2009} for the case $n=1$, and the same is to be expected for general $n$.\\\ \\{\bf Outline of the paper.} In the first section, we recollect fundamental facts on Hecke algebras and $p$-stabilization. Our treatment diverges from \citep{schmidt1993,kazhdanmazurschmidt2000,januszewski2011,januszewski2015,januszewski2016}, since $p$-stabilization may not be achieved by the projectors constructed in loc.\ cit.\ due to the arbitrary ramification we allow: The known explicit formulae for the valuation of Whittaker functions (cf.\ \citep{kondoyasuda2012,matringe2013,miyauchi2014}) imply the vanishing of the ordinary projection of the essential vectors in these cases.

In section \ref{sec:birchlemma}, we prove local and global Birch Lemmata for general nearly ordinary automorphic representations in Theorems \ref{thm:localbirch} and \ref{thm:globalbirch}.

Section \ref{sec:lattices} contains the fundamental observations on the level of lattices in rational representations of $G$, which will allow us to prove the congruences which are necessary to establish Theorems \ref{introthm:firstpadicL} and \ref{introthm:secondpadicL}.

In section \ref{sec:cohomology}, exploiting results on the contribution of the residual spectrum of $\GL(n)$ to the cohomology of arithmetic groups \citep{jacquetshalika1981,jacquetshalika1981.2,moeglinwaldspurger1989,frankeschwermer1998,scholze2015,harderraghuram2017}, we establish with Theorem \ref{thm:regularcontrol} a control Theorem for $G$. We build on Hida's fundamental work \citep{hida1986b,hida1986a,hida1988a,hida1989,hida1992,hida1994} on the case $\GL(2)$ and make extensive use of his results for $\SL(n)$ and (inner forms of) $\GL(n)$ in \citep{hida1995,hida1998}.

Section \ref{sec:measures} contains our construction of $p$-adic measures, both in the Iwasawa algebra for abelian interpolation and also on Hida's universal nearly ordinary cohomology with torsion coefficients. In this section we use the beautiful formalism of automorphic symbols introduced in \citep{dimitrov2011} in the case of $\GL(2)$.

Finally, section \ref{sec:padicL} applies the results of the previous sections to the construction of abelian and non-abelian $p$-adic $L$-functions for automorphic representations of $G$. It also contains our applications to non-vanishing of central $L$-values. Two local non-vanishing results are crucial to our approach: At places dividing $p,$ we need the results on $p$-stabilization and on the local zeta integrals from sections \ref{sec:pstabilization} and \ref{sec:birchlemma}, and at archimedean places we rely on the non-vanishing results established in \citep{kastenschmidt2013,sunjams}.

The reader only interested in the construction of abelian $p$-adic $L$-functions or the proof of the non-vanishing result in Theorem \ref{introthm:firstnonvanishing} may safely skip sections \ref{sec:independenceofweight} to \ref{sec:controltheorem} and sections \ref{sec:modularsymbolsinfamilies}, \ref{sec:nonabelianpadicL} and \ref{sec:nonabeliannonvanishing}.

\subsubsection*{Acknowledgement}
The author would like to thank G\"unter Harder, Haruzo Hida, Claus-G\"unther Schmidt, Joachim Schwermer, Giovanni Rosso, Binyong Sun, Jack Thorne and Shunsuke Yamana for valuable discussions and helpful remarks. The author thanks Michael Harris for pointing out the references to Greenberg's work \citep{greenberg1983,greenberg1985}, as well as for informing the author that Ellen Eischen worked independently on a similar approach to non-vanishing in the case of unitary groups, which in the meantime appeared as \citep{eischenpre}. The author thanks the participants of the recent BIRS-CMO workshop on this topic for their interest in this work and the coorganizers for the opportunity to present it in a series of lectures. The author acknowledges support from the Max Planck Institute for Mathematics in Bonn while hosting him when part of this work was finalized. The author thanks the referees for helpful remarks and comments. Lastly, the author thanks his collaborators Mladen Dimitrov and A.\ Raghuram from \citep{dimitrovjanuszewskiraghurampre} for their enthusiasm.

\section*{Notation}

\noindent{\bf Fields and Algebras.} In the body of the paper, $F$\nomenclature[]{$F$}{number field or finite extension of $\QQ_p$} denotes a number field, i.e.\ a finite extension of $\QQ_p$ or of $\QQ$.

In the first case, we write $\OO\subseteq F$\nomenclature[]{$\OO$}{valuation ring in finite extension of $\QQ_p$ (in $E$ or in $F$)} for the valuation ring, $\mathfrak{p}\subseteq\OO$\nomenclature[]{$\mathfrak{p}$}{prime ideal in $F$ dividing $p$} for the maximal ideal, and $q=\absNorm(\mathfrak{p})$\nomenclature[]{$q$}{residue field cardinality} for the cardinality of the residue field $\OO/\mathfrak{p}$. We normalize the valuation $\absnorm{\cdot}$ on $F$ in such a way that $|\varpi|=q^{-1}$ for a uniformizer $\varpi\in\mathfrak{p}$\nomenclature[]{$\varpi$}{uniformizer of $\mathfrak{p}$}.

In the case of a number field $F/\QQ$, let $\Adeles_F=\Adeles_\QQ\otimes_\QQ F$\nomenclature[]{$\Adeles_F$}{ring of ad\`eles over $F$} denote the ring of ad\`eles over $F$ and abbreviate $\Adeles=\Adeles_\QQ$\nomenclature[]{$\Adeles$}{ring of ad\`eles over $\QQ$}. We write $\Adeles_F=\Adeles_{F,\infty}\otimes\Adeles_F^{(\infty)}$ where $\Adeles_F^{(\infty)}$ denotes the ring of finite adeles over $F$. More generally, $\Adeles_F^{(p\infty)}$ denotes the ring of finite adeles outside $p$, etc.\nomenclature[]{$\Adeles^{(\infty)}_F$}{finite ad\`eles over $F$}\nomenclature[]{$\Adeles^{(p\infty)}_F$}{finite ad\`eles outside $p$ over $F$}

The letter $p$ always denotes a prime in $\ZZ$\nomenclature[]{$p$}{a fixed rational prime}. $E/\QQ_p$\nomenclature[]{$E$}{finite extension of $\QQ_p$, field of coefficients} is a finite extension whose valuation ring is also denoted by $\OO$ (there is no confusion possible). We let $\overline{E}$\nomenclature[]{$\overline{E}$}{algebraic closure of $E$} denote an algebraic closure and $\overline{\OO}\subseteq\overline{E}$\nomenclature[]{$\overline{\OO}$}{valuation ring inside $\overline{E}$} the corresponding valuation ring.

\noindent{\bf Groups.} If $L$ is an algebraic or topological group, we write $L^0$\nomenclature[]{$L^0$}{identity component of (topological) group $L$} for the connected component containing the identity, and $\pi_0(L)=L/L^0$\nomenclature[]{$\pi_0(L)$}{group of connected components of (topological) group $L$} for the component group. We denote by $L^{\rm{der}}=\overline{[L,L]}$\nomenclature[]{$L^{\rm{der}}$}{derived group of $L$} the derived group. We let $B_n$ denote the standard Borel subgroup of $\GL_n$\nomenclature[]{$B_n$}{standard upper triangular Borel subgroup in $\GL(n)$} of upper triangular matrices, $U_n$ denotes its unipotent radical\nomenclature[]{$U_n$}{unipotent radical of $B_n$} and we choose the subgroup $B_n^-\subseteq\GL_n$\nomenclature[]{$B_n^-$}{lower triangular Borel subgroup in $\GL(n)$} of lower triangular matrices as opposite to $B_n$. Then $T_n=B_n\cap B_n^-$ is the standard diagonal torus\nomenclature[]{$T_n$}{standard diagonal torus in $\GL(n)$}.

$W(\GL_n,T_n)$\nomenclature[]{$W(\GL_n,T_n)$}{Weyl group of $\GL(n)$ with respect to $T_n$} denotes the Weyl group of $\GL_n$ with respect to $T_n,$ realized as the subgroup of permutation matrices in $\GL_n$. It is canonically isomorphic to the symmetric group $S_n$ on $\{1,\dots,n\}$\nomenclature[]{$S_n$}{symmetric group on $\{1,\dots,n\}$, occationally identified with $W(\GL_n,T_n)$}: To each $\omega\in W(\GL_n,T_n)$ we associate a permutation $\sigma\in S_n$ via the rule\nomenclature[]{$\omega$}{element in $W(\GL_n,T_n)\cong S_n$, realized as permuation matrix in $\GL(n)$}
$$
\omega\cdot b_k = b_{\sigma^{-1}(k)}
$$
for $1\leq k\leq n$. Here $b_k$ denotes the $k$-th standard basis vector of $\ZZ^n$\nomenclature[]{$b_k$}{$k$-th standard basis vector in $\ZZ^n$}. Then for any $a=(a_i)_{1\leq i\leq n}\in A^n$, $A$ a (commutative) ring, we have $\omega a=(a_{\sigma(i)})_{1\leq i\leq n}$. The map $\omega\mapsto\sigma^{-1}$ is an isomorphism $W(\GL_n,T_n)\to S_n$. Let $w_n\in W(\GL_n,T_n)$ denote the longest element. It is explicitly given by\nomenclature[]{$w_n$}{longest Weyl element in $W(\GL_n,T_n)$}
$$
w_{n}=
\begin{pmatrix}
&& 1\\
&\adots&\\
1&&
\end{pmatrix}.
$$
We have the embedding\nomenclature[]{$j_n$}{embedding $\GL(n)\to\GL(n+1)$}
$$
j_{n}\quad:\GL(n)\to\GL(n+1),
$$
$$
g\mapsto
\begin{pmatrix}
g&0\\
0&1
\end{pmatrix}.
$$

Set\nomenclature[]{$G$}{$\res_{F/\QQ}\GL(n+1)\times\GL(n)$} $$G\;:=\;\res_{F/\QQ}\left(\GL(n+1)\times\GL(n)\right)$$
for a number field $F$. Consistent with the notation for $\GL_n$, let $B\subseteq G$ denote the standard upper triangular Borel subgroup\nomenclature[]{$B$}{upper triangular Borel subgroup in $G$} with unipotent radical $U$\nomenclature[]{$U$}{unipotent radical of $B$}, $B^-$ its standard opposite of lower triangular matrices\nomenclature[]{$B^-$}{lower triangular Borel subgroup in $G$}, and $T=\res_{F/\QQ}T_{n+1}\times T_n$ the diagonal torus\nomenclature[]{$T$}{diagonal torus in $G$}. The longest element in the Weyl group $W(G,T)$ with respect to $B$ is denoted $$w_0\;:=\;(w_{n+1},w_n)\;\in\; G(\ZZ).$$\nomenclature[]{$w_0$}{longest Weyl element in $W(G,T)$} Dominance for $G$ and $\GL_n$ is understood with respect to $B$ and $B_n$ respectively.

Write $\Delta=j_n\times{\bf1}:\res_{F/\QQ}\GL(n)\to G$\nomenclature[]{$\Delta$}{diagonal embedding $\res_{F/\QQ}\GL(n)\to G$} for the diagonal embedding and set\nomenclature[]{$H$}{$\res_{F/\QQ}\GL(n)$, diagonally embedded into $G$}
$$H\;:=\;\Delta(\res_{F/\QQ}\GL(n))\;\subseteq\; G$$
for the diagonally embedded copy of $\GL(n)$, which we freely identify with the latter. It comes with a distinguished character\nomenclature[]{$N_H$}{norm character $H\to\GL_1$}
$$
  N_H:=N_{F/\QQ}\circ\det:\quad H\to\GL_1
$$
defined over $\QQ$.

Write $X_\QQ(H)$ for the lattice of $\QQ$-rational characters of $H$\nomenclature[]{$X_\QQ(H)$}{character lattice of $H$}, which is generated by $N_H$. Put\nomenclature[]{$\widetilde{H}$}{subgroup of $H$ complementary to the maximal $\QQ$-split torus in its center}
$$
\widetilde{H}\;:=\;\bigcap_{\chi\in X_\QQ(H)}\ker\chi.
$$
For a dominant weight $\lambda$ of $G$\nomenclature[]{$\lambda$}{$B$-dominant weight of $G$}, write $L_{\lambda,E}$ for the irreducible rational representation of highest weight $\lambda$ defined over some field $E/\QQ$\nomenclature[]{$L_{\lambda,E}$}{irreducible rational $G$-representation of highest weight $\lambda$ over the field $E$}. We call $\lambda$ {\em balanced}, if
\begin{equation}
  H^0(\widetilde{H}; L_{\lambda,E})\;\neq\;0.
  \label{eq:admissiblelambdacondition}
\end{equation}
This is the same to say that there is a non-zero $H$-invariant functional\nomenclature[]{$\eta_j$}{$H$-intertwining $L_{\lambda,E}\to E_{(j)}$}\nomenclature[]{$E_{(j)}$}{one-dimensional $H$-representation $N_H^{\otimes j}$}
$$
\eta_j:\quad L_{\lambda,E}\;\to\;E_{(j)}:=N_H^{\otimes j}\otimes E
$$
for some $j\in\ZZ$, i.e.\ 
\begin{equation}
  \Hom_H(L_{\lambda,E},E_{(j)})\;\neq\;0.
  \label{eq:jfunctionals}
\end{equation}
We call such a non-zero $\eta_j$ {\em admissible} for $\lambda$. We refer to section 2.4.5 of \citep{raghuram2015} for an explicit description of condition \eqref{eq:jfunctionals} in terms of highest weights for $\GL(n)$.

Let $\lieg_{\ZZ}$ denote the Lie algebra of $G$ over $\ZZ$\nomenclature[]{$\lieg$}{Lie algebra of $G$}, i.e.\ more precisely we take the restriction of scalars $G=\res_{\OO_F/\ZZ}\GL(n+1)\times\GL(n)$ of the ring of integers in $F$ to $\ZZ$ of the standard smooth group scheme $\GL(n+1)\times\GL(n)$ over $\OO_F$ and use the same notation for $B$, $H$, ..., with the similar standard choice of smooth models over $\ZZ$ in each case. For any $\ZZ$-algebra $A/\ZZ$ we set
$$
\lieg_A\;:=\;\lieg_{\ZZ}\otimes_{\ZZ}A,
$$
again likewise for the other Lie algebras under consideration. Let $U(\lieg_A)$ denote the universal enveloping algebra of $\lieg_A$ over $A$\nomenclature[]{$U(\lieg_A)$}{universal enveloping Lie algebra of $\lieg_A$ over $A$}.

\noindent{\bf Matrices.} We introduce the matrix\nomenclature[]{$h_n$}{matrix in $\GL(n+1)$}
$$
h_n\;:=\;
\begin{pmatrix}
  &&&1\\
  &w_n&&\vdots\\
  &&&\vdots\\
  0&\hdots&0&1
\end{pmatrix}\;\in\;\GL_{n+1}(\ZZ).
$$
For $e=(e_1,\dots,e_n)\in\ZZ^n$\nomenclature[]{$e$}{tuple $e=(e_1,\dots,e_n)\in\ZZ^n$} and $a\in A^\times$, define the matrix\nomenclature[]{$(-)^e$}{application of the coroot corresponding to $e$}
$$
a^e:=\diag(a^{e_1},\dots,a^{e_n})\;\in\;\GL_n(A).
$$
We consider any $\delta\in\ZZ$ as the constant tuple denoted $(\delta)\in\ZZ^n$\nomenclature[]{$(\delta)$}{constant tuple $(\delta)\in\ZZ^n$}. Then $a^\delta\cdot a^{e}=a^{(\delta)+e}$.

The character lattice of $T_n$ is canonically identified with $\ZZ^n$ in such a way that the dominant weights for $\GL(n)$ with respect to $B_n$ are ordered $n$-tuples $\lambda=(\lambda_1,\dots,\lambda_n)$ of decreasing integers
$$
\lambda_1\geq\lambda_2\geq\cdots\geq \lambda_n.
$$
Then the sum $2\rho_n$ of the positive roots of $\GL(n)$ is represented by the tuple\nomenclature[]{$\rho_n$}{half sum of positive roots in $\GL(n)$}
$$
2\rho_n\;=\;(n-1,n-3,\dots,3-n,1-n)\;\in\;\ZZ^n.
$$

For $x\in A^\times$ define\nomenclature[]{$d_x$}{diagonal matrix with determinant $x$}\nomenclature[]{$t_x$}{diagonal matrix $\diag(x^n,x^{n-1},\dots,x)$}
\begin{equation}
d_x\;:=\;
\Delta(\diag(x,1,\dots,1))\;\in\;G(A),
\label{eq:definitionofdx}
\end{equation}\begin{equation}
t_{x}\;:=\;x^{\rho_n+((n-1)/2)}\;=\;\diag(x^{n},x^{n-1},\dots,x)\;\in\;\GL_n(A).
\label{eq:definitionoftx}
\end{equation}
We also consider the latter element as an element of $G(A)$ via the diagonal embedding. With this notation at hand, set\nomenclature[]{$h$}{matrix in $G(\ZZ_p)$}\nomenclature[]{${\bf1}_n$}{identity matrix in $\GL(n)$}
\begin{equation}
h\;:=\;
\left(
h_nj_n(t_{-1}),
{\bf1}_n\right)\;\in\;G(\ZZ_p).
\label{eq:definitionofh}
\end{equation}

\section{Hecke Algebras}\label{sec:hecke}

\subsection{Hecke pairs}

For any Hecke pair $(R,S)$ consider the free $\ZZ$-module $\mathcal H_\ZZ(R,S)$ over the set of all double cosets $RsR$, which naturally embeds into the free $\ZZ$-module $\mathscr R_\ZZ(R,S)$\nomenclature[]{$\mathcal R_\ZZ(R,S)$}{free $\ZZ$-module of right cosets $sR$} over the set of the right cosets $sR$, $s\in S$:
$$
RsR=\bigsqcup_i s_iR\mapsto\sum_i s_iR.
$$
Identify $\mathcal H_\ZZ(R,S)$ with its image under this embedding. Then $\mathcal H_\ZZ(R,S)$ is the $\ZZ$-module of $R$-invariants under the action
$$
R\times \mathscr R_\ZZ(R,S)\to \mathscr R_\ZZ(R,S),\;\;\;(r,sR)\mapsto rsR.
$$
It is well known that $\mathcal H_\ZZ(R,S)$ admits a structure of an associative $\ZZ$-algebra with multiplication
$$
\left(\sum_i s_iR\right)\cdot\left(\sum_jt_jR\right):=\sum_{i,j}s_it_jR.
$$
This algebra is unitary if and only if $R\cap S\neq\emptyset$. For any commutative ring $A$ set\nomenclature[]{$\mathcal H_A(R,S)$}{double coset Hecke algebra over the ring $A$}
$$
\mathcal H_A(R,S):=\mathcal H_\ZZ(R,S)\otimes_\ZZ A.
$$
Then $\mathcal H_A(R,S)$ is an associative $A$-algebra, the {\em Hecke algebra} of the pair $(R,S)$ over $A$.

For a locally compact topological group $G$ and a compact open subgroup $K\leq G,$ the module $\mathscr R_A(K,G)$ may be interpreted as the $A$-module of locally constant right $K$-invariant mappings $f:G\to A$ with compact support and $\mathcal H_A(K,G)$ is just the submodule of left $K$-invariant mappings. If $A\subseteq\CC$, then multiplication is nothing but convolution
$$
\alpha*\beta\;:\;
x\mapsto \int_G \alpha(g)\beta(xg^{-1})dg,
$$
where $dg$ is the right invariant Haar measure on $G$ which assigns measure $1$ to $K$. This integral is eventually a finite sum with integer coefficients. Therefore, this interpretations is valid even without the assumption $A\subseteq\CC$.

\subsection{$p$-adic Hecke algebras}\label{sec:padicheckealgebras}

Let $F/\QQ_p$ denote a $p$-adic field with integer ring $\OO$. Write $\mathfrak{p}\subseteq\OO$ for the maximal ideal, $\varpi\in\mathfrak{p}$ for a uniformizer and $q=\absNorm(\mathfrak{p})$ the norm of $\mathfrak{p}$\nomenclature[]{$\absNorm(-)$}{absolute norm of an ideal}. For any $\alpha\geq\alpha'\geq 0$\nomenclature[]{$\alpha$}{non-negative integer, at least as large as $\alpha'$}\nomenclature[]{$\alpha'$}{non-negative integer, bounded above by $\alpha$} write $I_{\alpha',\alpha}^n\subseteq\GL_n(\OO)$ for the subgroup of matrices becoming upper triangular modulo $\mathfrak{p}^\alpha$ and which lies in $U_n(\OO/\mathfrak{p}^{\alpha'})$ when considered modulo $\mathfrak{p}^{\alpha'}$\nomenclature[]{$I_{\alpha',\alpha}^n$}{subgroup of $\GL_n(\OO)$, upper triangular mod $p^\alpha$, unipotent mod $p^{\alpha'}$}. Set $I_{\alpha}^n:=I_{0,\alpha}^n$.\nomenclature[]{$I_\alpha^n$}{subgroup of $\GL_n(\OO)$, upper triangular mod $p^\alpha$}

Recall that a tuple $e=(e_1,\dots,e_n)\in\ZZ^n$ is {\em dominant} if
$$
e_1\geq e_2\geq\cdots\geq e_n.
$$
Consider the semigroup\nomenclature[]{$\Delta_{F,n}$}{semigroup in $T_n(F)$}
$$
\Delta_{F,n}\;:=\;T_n(\OO^\times)\cdot\{\varpi^e\mid e\in\ZZ_{\geq 0}^n\;\text{dominant}\}\;\subseteq\;T_n(F)
$$
and define the Hecke algebra\nomenclature[]{$\mathcal H_A^n(\alpha',\alpha)$}{abstract Hecke algebra for $\GL(n)$ of level $I_{\alpha',\alpha}$ for cosets represented by $\Delta_{F,n}$}
$$
\mathcal H_A^n(\alpha',\alpha)\;:=\;\mathcal H_A(I_{\alpha',\alpha}^n,I_{\alpha',\alpha}^n\Delta_{F,n}I_{\alpha',\alpha}^n).
$$
Whenever $e$ is dominant, $e_n\geq 0$ and $\alpha>0$ we define a Hecke operator\nomenclature[]{$U_{\varpi}^e$}{generalized $U_p$-operator in $\mathcal H_A^n(\alpha',\alpha)$}
\begin{equation}
  U_{\varpi}^e
  \;:=\;I_{\alpha',\alpha}^n\varpi^e I_{\alpha',\alpha}^n\;=\;
  \bigsqcup_{u\in U_n(\OO)/\varpi^e U_n(\OO)\varpi^{-e}}
    u\varpi^e I_{\alpha',\alpha},
    \label{eq:definitionofUpie}
\end{equation}
which depends on the choice of $\varpi$ whenever $\alpha'>0$. It is well known that these operators commute \citep{iwahorimatsumoto1965,gritsenko1992,hida1995,hida1998}. Moreover, we have the relation
$$
U_{\varpi}^e\cdot U_{\varpi}^{e'}\;=\;U_{\varpi}^{e+e'}
$$
for any dominant $e,e'\in\ZZ_{\geq 0}^n$. Therefore, writing
$$
\omega_{\nu}\;:=\;(\underbrace{1,\dots,1}_{\nu},\underbrace{0,\dots,0}_{n-\nu})
$$
for the $\nu$-th fundamental weight with $\nu$ leading $1$'s and $n-\nu$ tailing $0$'s\nomenclature[]{$\omega_{\nu}$}{$\nu$-th fundamental weight for $\GL(n)$}, the operators\nomenclature[]{$V_\nu$}{$U_p$ operator for $\GL(n)$ for $\nu$-th fundamental weight}
$$
V_\nu\;:=\;U_{\varpi}^{\omega_{\nu}},\quad 1\leq\nu\leq n,
$$
generate $\mathcal H_A^n(0,\alpha)$.

Sending $U_{\varpi}^e\in \mathcal H_A^n(0,\alpha)$ to $U_{\varpi}^e\in \mathcal H_A^n(\alpha',\alpha)$ defines an inclusion
$$
\mathcal H_A^n(0,\alpha)\;\subseteq\;
\mathcal H_A^n(\alpha',\alpha),
$$
depending on the choice of uniformizer $\varpi$. We see that
\begin{equation}
\mathcal H_A^n(\alpha',\alpha)\;=\;\mathcal H_A^n(0,\alpha)[I_\alpha^n/I_{\alpha',\alpha}^n]
\;=\;\mathcal H_A^n(0,\alpha)[T_n(\OO/\mathfrak{p}^{\alpha'})],
\label{eq:iwahoriheckeasgroupring}
\end{equation}
which is a finitely generated commutative $A$-algebra (cf.\ \citep{hida1995,hida1998}).

\subsection{Parabolic Hecke algebras}

Define\nomenclature[]{$I_{\alpha'}^{B_n}$}{compact open in $B_n(F)$}
$$
I_{\alpha'}^{B_n}\;:=\;B_n(F)\cap I_{\alpha',\alpha}^n.
$$
As the notation suggests, this compact open subgroup of $B_n(F)$ is independent of $\alpha$. Restriction induces a canonical isomorphism
\begin{equation}
\mathcal H_A^n(\alpha',\alpha)\;\cong\;
\mathcal H_A(I_{\alpha'}^{B_n},I_{\alpha'}^{B_n}\Delta_{F,n}I_{\alpha'}^{B_n}),
\label{eq:iwahoriheckeasparabolichecke}
\end{equation}
which on cosets is explicitly given by the map
$$
 g I_{\alpha',\alpha}^n\mapsto g I_{\alpha'}^{B_n}.
$$
Existence of the Iwasawa decomposition shows that this is well defined. Set\nomenclature[]{$\mathcal H_A^{B_n}(\alpha')$}{parabolic Hecke algebra for $\GL(n)$}
$$
\mathcal H_A^{B_n}(\alpha')\;:=\;
\mathcal H_A(I_{\alpha'}^{B_n},I_{\alpha'}^{B_n}(T_n(F)\cap \OO^n)I_{\alpha'}^{B_n}).
$$
Then by \eqref{eq:iwahoriheckeasparabolichecke}, $\mathcal H_A^n(\alpha',\alpha)$ is a subalgebra of $\mathcal H_A^{B_n}(\alpha')$.

\subsection{The $U_{\mathfrak{p}}$-operators}

In $\mathcal H_A^{B_n}(\alpha')$ we have the Hecke operators\nomenclature[]{$\widetilde{U}_i$}{parabolic Hecke operators}
$$
\widetilde{U}_i:=I_{\alpha'}^{B_n}
\begin{pmatrix}
{\bf1}_{i-1}&0&0\\
0&\varpi&0\\
0&0&{\bf1}_{n-i}
\end{pmatrix}
I_{\alpha'}^{B_n},\quad 1\leq i\leq n.
$$
With \eqref{eq:iwahoriheckeasgroupring} we see
$$
\mathcal H_A^{B_n}(\alpha')\;=\;\mathcal H_A^{B_n}(0)[T_n(\OO/\mathfrak{p}^{\alpha'})].
$$

\begin{proposition}\label{prop:TnuviaUi}
We have for $0\leq \nu\leq n$,
\begin{equation}
  q^{\frac{\nu(\nu-1)}{2}}\cdot V_\nu\;=\;
  \widetilde{U}_1 \widetilde{U}_2\cdots \widetilde{U}_\nu.
  \label{eq:TnuviaUi}
\end{equation}
\end{proposition}

\begin{proof}
The proof of Lemma 4.1 in \citep{kazhdanmazurschmidt2000} remains valid in our setting.
\end{proof}

Set\nomenclature[]{$U_{\mathfrak{p}}$}{$U_p$-operator in $\GL(n)$}
$$
U_{\mathfrak{p}}\;:=\;\prod_{\nu=1}^{n-1}V_{\nu},
$$
and\nomenclature[]{$U_{\mathfrak{p}}'$}{$U_p$-operator in $\GL(n)$, including central contribution}
$$
U_{\mathfrak{p}}'\;:=\;V_{n}\cdot U_\mathfrak{p}=\prod_{\nu=1}^{n}V_{\nu}.
$$

\subsection{Decomposition of Hecke polynomials}

Consider the standard Hecke operators\nomenclature[]{$T_\nu$}{standard spherical Hecke operator for $\GL(n)$}
$$
T_\nu\;:=\;I_{0,0}\varpi^{\omega_\nu}I_{0,0}\;\in\;\mathcal H_A^n(0,0)
$$
in the spherical Hecke algebra. The reciprocal Hecke polynomial\nomenclature[]{$H_F(X)$}{reciprocal spherical Hecke polynomial for $\GL(n)$ over nonarchimedean field $F$}
\begin{equation}
  H_F(X):=\sum_{\nu=0}^n (-1)^\nu q^{\frac{\nu(\nu-1)}{2}}T_\nu X^{n-\nu}\in\mathcal H_A^n(\alpha',\alpha)
  \label{eq:Fheckepolynomial}
\end{equation}
admits a factorization
\begin{equation}
H_F(X)=\prod_{i=1}^n(X-\widetilde{U}_i),
\label{eq:gritsenko}
\end{equation}
cf.\ \citep[Theorem 2]{gritsenko1992}. Although $H_F(X)$ is defined via the spherical Hecke algebra, it is relevant for us in the ramified case as well due to its relation to the Hodge polygon (cf.\ \citep{hida1998}).

\subsection{$p$-stabilization in principal series representations}\label{sec:pstabilization}

Let $E/\QQ$ denote a field of characteristic $0$. Recall that the norm $\absnorm{\cdot}$ on $F$ is normalized such that $|\varpi|=q^{-1}$ and consider its values in $E$.

For an admissible representation $(V,\pi)$ of $\GL_n(F)$ over $E$ the Jacquet module is defined as\nomenclature[]{$V_{B_n}$}{Jaquet module of $V$}
$$
V_{B_n}\;:=\;V/\langle uv-v\mid u\in U_n(F), v\in V\rangle.
$$
This is an admissible representation of $T_n(F)$. For a representation $W$ of $B_n(F)$ over $E,$ define the space\nomenclature[]{$W^{U_n(\OO)}$}{space of $U_n(\OO)$-fixed vectors in $W$}
$$
W^{U_n(\OO)}\;:=\;\{w\in W\mid \forall u\in U_n(\OO): uw=w\}
$$
of invariants. If $W$ is of finite length, then $W^{U_n(\OO)}$ is naturally an $\mathcal H_E^{B_n}(\alpha')$-module for $\alpha'\gg0$ sufficiently large. The operator
$$
w\;\mapsto\;\int_{U_n(\OO)}uw\; du
$$
for the normalized Haar measure $du$ on $U_n(\OO)$ is a projector $W\to W^{U_n(\OO)}$. Therefore, taking invariants is an exact functor and we have an epimorphism
\begin{equation}
  V^{U_n(\OO)}\;\to\;V_{B_n}^{U_n(\OO)}.
  \label{eq:epimorphicUnOinvariants}
\end{equation}

Fix a continuous character $\lambda:T_n(F)\to E^\times$, where $E^\times$ is topologized with the discrete topology. We introduce a modified character\nomenclature[]{$\widetilde{\lambda}$}{modification of character $\lambda$}
$$
\widetilde{\lambda}:\quad \diag(t_1,\dots,t_n)\;\mapsto\;\prod_{i=1}^n\absnorm{t_i}^{n-i}\lambda_i(t_i),
$$
where $\lambda_i$ denotes the restriction of $\lambda$ to the $i$-th component of $T_n(F)$. Considering $\tilde{\lambda}$ as a character of $B_n(F)$, define an algebraically induced principal series representation\nomenclature[]{$I_{B_n}^{\GL_n}(\lambda)$}{unnormalized induction of $\widetilde{\lambda}$}
$$
I_{B_n}^{\GL_n}(\lambda)\;:=\;\Ind_{B_n(F)}^{\GL_n(F)}\widetilde{\lambda},
$$
where $\Ind_{B_n(F)}^{\GL_n(F)}$ denotes unnormalized algebraic induction in the category of smooth representations (cf.\ D\'efinition 1.9 in \citep{clozel1990}). Then, if $E=\CC$,
$$
\absnorm{\det(\cdot)}^{\frac{1-n}{2}}\otimes I_{B_n}^{\GL_n}(\lambda)
$$
agrees with the normalized induction of $\lambda$ from $B_n(F)$ to $\GL_n(F)$.

Recall that the Weyl group $W(\GL_n,T_n)$ acts naturally on the set of characters of $T_n(F)$ from the right. Then for every $\omega\in W(\GL_n,T_n)$,
$$
\left(I_{B_n}^{\GL_n}(\lambda)\right)^{ss}\;=\;\left(I_{B_n}^{\GL_n}\left(\lambda^\omega\right)\right)^{ss},
$$
where the superscript $(-)^{ss}$ denotes semisimplification\nomenclature[]{$(-)^{ss}$}{semisimplification}.

\begin{proposition}[Hida]\label{prop:hida}
  Let $E/\QQ$ denote a field of characteristic $0$ and let $\lambda:T_n(F)\to E^\times$ denote a continuous character as above. Assume that the characters $\lambda^\omega$ for $\omega\in W(\GL_n,T_n)$ are pairwise distinct. Then
  \begin{enumerate}[(i)]
  \item The Jacquet module of $I_{B_n}^{\GL_n}(\lambda)$ is a semisimple $T_n(F)$-module and as such
    \begin{equation}
      \left(I_{B_n}^{\GL_n}(\lambda)\right)_{B_n}\;=\;\bigoplus_{\omega\in W(\GL_n,T_n)}\widetilde{\lambda^\omega}.
      \label{eq:jacquetmoduledecomposition}
    \end{equation}
  \item Each $v$ in the $\widetilde{\lambda^\omega}$-isotypic component in \eqref{eq:jacquetmoduledecomposition} is a simultaneous eigenvector of $V_\nu$, $1\leq\nu\leq n$, and
    \begin{equation}
      V_\nu v\;=\;q^{-\frac{\nu(\nu-1)}{2}}(\lambda^\omega)(\varpi^{\omega_\nu})\cdot v.
      \label{eq:jacquetmoduleeigenvalues}
    \end{equation}
  \end{enumerate}
\end{proposition}

\begin{proof}
  This is a restatement of Proposition 5.4 and Corollary 5.5 in \citep{hida1998}, taking into account that Hida works with a different normalization stemming from the right actions considered in loc.\ cit., where we work with left actions. Hida showed in particular that
  \begin{equation}
    \left(\left(I_{B_n}^{\GL_n}(\lambda)\right)_{B_n}\right)^{U_n(\OO)}\;=\;
    \left(I_{B_n}^{\GL_n}(\lambda)\right)_{B_n}.
    \label{eq:jacquetmoduleinvariants}
 \end{equation}
 Therefore, for $\alpha'$ sufficiently large, $\mathcal H_E^{B_n}(\alpha')$ acts on the Jacquet module canonically.
\end{proof}

Let $\pi$ be a generic irreducible admissible representation of $\GL_n(F)$ with Whittaker model $\mathscr W(\pi,\psi)$\nomenclature[]{$\mathscr W(\pi,\psi)$}{$\psi$-Whittaker model of $\pi$} with respect to a generic character $\psi$ of $U_n(F)$ trivial on $U_n(\OO)$.

\begin{proposition}\label{prop:nonvanishingofeigenvectors}
  Let $\pi$ be a generic irreducible admissible representation of $\GL_n(F)$. Assume that $\pi$ occurs as a subquotient of $I_{B_n}^{\GL_n}(\lambda)$ for a character $\lambda$ satisfying the condition that $\lambda^\omega$ for $\omega\in W(\GL_n,T_n)$ are pairwise distinct. Then every simultaneous $V_{\nu}$-eigenvector $W\in\mathscr W(\pi,\psi)^{U_n(\OO)}$, $1\leq\nu<n$, with non-zero eigenvalues enjoys the following properties:
  \begin{itemize}
    \item[(i)] There exists a unique $\omega\in W(\GL_n,T_n)$ such that for all $1\leq\nu\leq n:$
      \begin{equation}
        V_{\nu}W\;=\;
        q^{-\frac{\nu(\nu-1)}{2}}(\lambda^{\omega})(\varpi^{\omega_\nu})\cdot W.
        \label{eq:VnuWeigenvalue}
      \end{equation}
    \item[(ii)]
      $W$ lies in a unique line in $\mathscr W(\pi,\psi)$ characterized by \eqref{eq:VnuWeigenvalue} and
      $$
        W({\bf1}_n)\;\neq\;0.
      $$
   \end{itemize}
\end{proposition}

\begin{proof}
  By equation (5.4) on page 678 of \citep{hida1998} we know that
  \begin{equation}
    W(\pi,\psi)^{U_n(\OO)}\;=\;W(\pi,\psi)_{B_n}^{U_n(\OO)}\oplus
    \kernel\left( W(\pi,\psi)^{U_n(\OO)}\to W(\pi,\psi)_{B_n}^{U_n(\OO)}\right),
    \label{eq:Winvariantsdecomposition}
  \end{equation}
  and each $V_\nu$ acts nilpotently on the second summand on the right hand side. Therefore, by our hypothesis on $W$, this Whittaker vector maps to a non-zero $V_\nu$-eigenvector $W_{B_n}\in W(\pi,\psi)_{B_n}^{U_n(\OO)}$ with same eigenvalue. By the hypothesis on $\lambda$, relation \eqref{eq:jacquetmoduleeigenvalues} shows that there is a unique $\omega\in W(\GL_n,T_n)$ satisfying (i), because distinct elements of the Weyl group yield different sets of eigenvalues. This also implies the uniqueness of $W$ up to a scalar.

  This also shows that the Jacquet module of $\pi$ admits $\widetilde{\lambda^{\omega}}$ as a direct summand. Therefore, by Frobenius reciprocity, $\pi$ occurs as a submodule of $I_{B_n}^{\GL_n}(\lambda^{\omega w_n})$.

  For the second statement in (ii), we may assume $E=\CC$ and realize $I_{B_n}^{\GL_n}(\lambda^{\omega w_n})$ inside the space of functions
  $$
    \{f:\GL_n(F)\to\CC\mid \forall g\in\GL_n(F),b\in B_n(F):\;f(bg)=\widetilde{\lambda^{\omega w_n}}(b)f(g)\},
  $$
  Remark that for $\alpha\geq 1$, the double coset
  \begin{equation}
    B_n(F)w_n I_{\alpha,\alpha}^n\;=\;B_n(F)w_n U_n(\OO)
    \label{eq:opencell}
  \end{equation}
  is independent of $\alpha$. Put
  $$
  f_0:\quad g\mapsto
  \begin{cases}
    \widetilde{\lambda^{\omega w_n}}(b)
    ,&\text{if}\;g=bw_nr,\,b\in B_n(F),r\in U_n(\OO),\\
    0,&\text{else}.
  \end{cases}
  $$
  By construction, $f_0$ is an element of $I_{B_n}^{\GL_n}\left(\lambda^{\omega w_n}\right)^{U_n(\OO)}$ satisfying
  \begin{equation}
    V_\nu f_0\;=\;q^{-\frac{\nu(\nu-1)}{2}}
    \lambda^{\omega}\left(\varpi^{\omega_\nu}\right)\cdot f_0,\quad 1\leq\nu\leq n.
    \label{eq:f0eigenvaluerelations}
  \end{equation}
  For every $f\in V$ supported on $B_n(F)w_n B_n(F)$ and $g\in B_n(F)w_n B_n(F),$ consider for the normalized Haar measure $du$ on $U_n(F)$ the integral
  \begin{equation}
    W_f(g)\;:=\;\int_{U_n(F)}f(w_nug)\overline{\psi}(u)du.
    \label{eq:whittakerintegral}
  \end{equation}
  By a well known result of Rodier, this integral converges and extends uniquely to an intertwining operator
  $$
    I_{B_n}^{\GL_n}\left(\lambda^{\omega w_n}\right)\to\Ind_{U_n(F)}^{\GL_n(F)}\psi,
  $$
  cf.\ Corollary 1.8 in \citep{casselmanshalika1980}. By \eqref{eq:whittakerintegral}, the vector $f_0$ is sent to a Whittaker vector which evaluates at ${\bf1}_n$ to
  \begin{align*}
    W_{f_0}({\bf1}_n)=&\int_{U_n(F)}f_0(w_nu)\overline{\psi}(u)du\\
    =&\int_{U_n(\OO)}du\\
    \neq&0,
  \end{align*}
  because the integrand vanishes for $u\not\in U_n(\OO)$ and assumes the value $f_0(w_0)=1$ for $u\in U_n(\OO)$.

  Now since $\pi$ is a submodule of $I_{B_n}^{\GL_n}\left(\lambda^{\omega w_n}\right)$, the uniqueness of Whittaker models for $I_{B_n}^{\GL_n}\left(\lambda^{\omega w_n}\right)$ shows that we have a commutative square
  $$
    \begin{CD}
      I_{B_n}^{\GL_n}\left(\lambda^{\omega w_n}\right) @>>> \Ind_{U_n(F)}^{\GL_n(F)}\psi\\
      @AAA @|\\
      \pi @>>> \Ind_{U_n(F)}^{\GL_n(F)}\psi
    \end{CD}
  $$
  Therefore, $W_{f_0}$ maps to a simultaneous eigenvector $W'$ of the operators $V_\nu$ in $\mathscr W(\pi,\psi)$. By the multiplicity one property of the eigenspaces characterized by \eqref{eq:f0eigenvaluerelations}, there is a non-zero scalar $c\in F^\times$ satisfying
  $$
    W\;=\;c\cdot W',
  $$
  and therefore
  $$
    W({\bf1})\;=\;c\cdot W'({\bf1}_n)\;=\;c\cdot W_{f_0}({\bf1}_n)\;\neq\;0.
  $$
  This concludes the proof.
\end{proof}

\begin{remark}\label{rmk:Lfunctiontoeigenvalues}
  The relation between the eigenvalues of the operators $V_\nu$ and the local $L$-function $L(s,\pi)$ attached to $\pi$ as in \citep{godementjacquet1972} is the following\nomenclature[]{$L(s,\pi)$}{$L$-function of local or global representaiton $\pi$ with complex variable $s$}. Let
  $$
    L(s,\pi)\;=\;\prod_{i=1}^n\frac{1}{1-\alpha_iq^{-s}},
  $$
  for $\alpha_i\in\CC$. By Corollary 3.6 in loc.\ cit. we know that there is a polynomial $P(X)\in\CC[X]$ satisfying $P(1)=1$ and
  $$
    L(s,\pi)\;=\;P(q^{-s})\cdot L\left(s,I_{B_n}^{\GL_n}(\lambda)\right)\;=\;\frac{P(q^{-s})}{\prod\limits_{\begin{subarray}ci=1\\\lambda_i(\OO^\times)=1\end{subarray}}^n \left(1-\lambda_i(\varpi)q^{-\frac{n-1}{2}-s}\right)},
  $$
  where the product in the denominator on the right hand side runs over all $i$ for which $\lambda_i$ is unramified. Hence, $\alpha_i$ equals either $\lambda_i(\varpi)q^{-\frac{n-1}{2}}\neq 0$ or $0$. Assuming without loss of generality that $\alpha_1,\dots,\alpha_\ell\neq 0$ and
  $$
    \alpha_{\ell+1}=\cdots=\alpha_n=0,
  $$
    for some $0\leq\ell\leq n$, we have $\ell=n$ if and only if $\pi$ is spherical. Furthermore, $\pi$ is Iwahori-spherical if and only if the characters $\lambda_1,\dots,\lambda_n$ are unramified.
\end{remark}

\section{A Birch Lemma for $p$-nearly ordinary automorphic forms}\label{sec:birchlemma}

In this section we generalize the Birch Lemma from \citep{januszewski2011,januszewski2015} to arbitrary $p$-nearly ordinary forms. The local main result is Theorem \ref{thm:localbirch}, whose proof will occupy section \ref{sec:birchlocal}. The global main result is Theorem \ref{thm:globalbirch}.

\subsection{The twisted local Zeta integral}\label{sec:birchlocal}

We use the notation of \cite[Section 2]{januszewski2011} in the modified setting of \citep{januszewski2015} with minor modifications. In particular, in this section $F$ denotes a non-archimedean local field with valuation ring $\OO\subseteq F$. Fix again a uniformizer $\varpi\in\OO$ and write $\mathfrak{p}\subset\OO$ for the maximal ideal and $q=\absNorm(\mathfrak{p})$ as before.

If $\chi:F^\times\to\CC^\times$ is a quasi-character\nomenclature[]{$\chi$}{a local or global quasi-character}, we write $\mathfrak{f}_\chi$ for its conductor\nomenclature[]{$\mathfrak{f}_\chi$}{conductor of $\chi$} and assume that it is generated by $f_\chi=\varpi^{e_\chi}$\nomenclature[]{$f_\chi$}{power of $\varpi$ generating $\mathfrak{f}_\chi$}, $e_\chi\geq 0$\nomenclature[]{$e_\chi$}{exponent satisfying $\mathfrak{f}_\chi=\mathfrak{p}^{e_\chi}$}. By abuse of notation, we occasionally write $\chi(g)$ for $\chi(\det(g))$\nomenclature[]{$\chi(g)$}{$\chi$ evaluated at $\det g$}, $g\in\GL_n(F)$.

Fix a non-trivial additive character $\psi:F\to\CC^\times$ of conductor $\OO$. The choice of $\psi$ normalizes the Gau\ss{} sum\nomenclature[]{$G(\chi)$}{Gau\ss{} sum of $\chi$ with respect to $\psi$}
$$
G(\chi)\;:=\;
\sum_{x+\mathfrak{f}_\chi\in(\OO/\mathfrak{f}_\chi)^\times}\chi(x/f_\chi)\psi(x/f_\chi)\;=\;
\frac{1}{\absNorm(\mathfrak{f}_\chi)}\cdot\int_{F^\times}\chi(x)\psi(x)dx,
$$
where the second identity is only valid for $\mathfrak{f}_\chi\neq 1$ and $dx$ denotes the {\em additive} Haar measure on $F$ which attaches volume $1$ to $\OO$.

Implicit in the second identity is the fact that for any $0\neq g\in\OO$, we have with
$\mathfrak{h}:=\mathfrak{f}\cap g\OO$ the relation
\begin{equation}
\sum_{x+\mathfrak{h}\in(\OO/\mathfrak{h})^\times}\chi(x/g)\psi(x/g)\;=\;
\begin{cases}
\absNorm(\mathfrak h/\mathfrak f_\chi)\cdot
\chi(g/f_\chi)\cdot G(\chi),&\text{if\;}\mathfrak{f}_\chi=
\OO g,\\
0,&\text{otherwise}.
\end{cases}
\label{eq:gaussnull}
\end{equation}

Extend $\psi$ to $U_n(F)$ by the rule\nomenclature[]{$\psi$}{a non-trivial character of $F$ (locally) or $\Adeles_F$, extended to generic character of $U_n(F)$ (locally) or $U_n(\Adeles_F)$}
\begin{equation}
\psi(u):=\prod_{i=1}^{n-1}\psi(u_{ii+1})
\label{eq:psitounextension}
\end{equation}
for $u=(u_{ij})\in U_n(F)$.

The Haar measure $dg$ on $\GL_n(F)$\nomenclature[]{$dg$}{normalized Haar measure} is normalized such that the maximal compact subgroup $\GL_n(\OO)$ has measure $1$.

Denote by $I^n$ the Iwahori subgroup of $\GL_n(\OO)$\nomenclature[]{$I^n$}{Iwahori subgroup of $\GL_n(\OO)$}, i.e.\ the group of matrices $g\in\GL_n(\OO)$ that become upper triangular modulo $\mathfrak{p}$. Fixing another element $f=\varpi^\alpha\in\OO$ with $\alpha\geq 1$, write $I^{n}_{\alpha}$ for the subgroup of elements of $\GL_n(\OO)$ lying in $B_n(\OO/\mathfrak{f})$ modulo $\mathfrak{f}:=\OO\cdot f$.

All quantities that are defined relative to $f$ in \citep{januszewski2011} keep their meaning, i.e.\ the matrices $A_n$, $B_n$, $C_n$, $D_n$, $E_n$, $\phi_n$ are all defined with respect to $f=\varpi^\alpha$\nomenclature[]{$E_n$}{matrix depending on $f$ defined in \citep{januszewski2011}}. We recall their definition below.

In particular,\nomenclature[]{$D_n$}{diagonal matrix $\diag(f^{-(n-1)},\dots,f^{n-1})\in\GL_n(F)$}
$$
D_n\;:=\;\diag(f^{-(n-1)},f^{-(n-3)},\dots,f^{n-3},f^{n-1})\;\in\;\GL_n(F)
$$
and for any $\delta\in\ZZ$ the definition of the linear form\nomenclature[]{$\lambda_n^\delta$}{map $F^{n\times n}\to F$}
$$
\lambda_n^\delta:F^{n\times n}\to F,\;\;\;g\mapsto \varpi^{-\delta}\cdot b_n^t\cdot g\cdot \phi_n,
$$
where\nomenclature[]{$\phi_n$}{column vector depending on $f$}
$$
\phi_n := (f^{-n},f^{-(n-1)},\dots,f^{-1})^t.
$$
Again for $\delta\in\ZZ$ we have\nomenclature[]{$j_{n,\delta}$}{map $\GL_n(F)\to\GL_{n+1}(F)$}
$$
j_{n,\delta}:\GL_n(F)\to\GL_{n+1}(F),
$$
$$
g\mapsto
\begin{pmatrix}
g&0\\
0&\varpi^\delta
\end{pmatrix}.
$$
Then $j_n=j_{n,0}$.

Deviating slightly from notation in previous works, put\nomenclature[]{$J^n_{\ell}$}{principal congruence subgroup mod $\mathfrak{f}^\ell$ in $\GL_n(\OO)$}
$$
J^n_{\ell}\;:=\;\kernel\left[\GL_n(\OO)\to\GL_n(\OO/\mathfrak{f}^{\ell})\right].
$$

Assume $\ell\geq 2n$. Then
$$
J^n_{\ell}\;\subseteq\; I^{n}_{\alpha}\cap w_nD_n^{-1}I^{n}_{\alpha}D_nw_n.
$$

Fix a system $R_\ell$ of representatives for $\OO/\mathfrak{f}^\ell$\nomenclature[]{$R_\ell$}{system of representatives for $\OO/\mathfrak{f}^\ell$} and let $R_\ell^\times\subseteq R_\ell$\nomenclature[]{$R_\ell^\times$}{system of representatives for $(\OO/\mathfrak{f}^\ell)^\times$ in $R_\ell$} be a system of representatives of $\left(\OO/\mathfrak{f}^\ell\right)^\times$. To simplify notation in the sequel, we assume that
\begin{equation}
0,\pm1,\pm f,\dots,\pm f^{\ell-1}\in R_\ell.
\label{eq:specialrep}
\end{equation}
Set\nomenclature[]{$\mathfrak{R}_{n,\ell}$}{matrices in $I^n$ with entries in $R_\ell$}
$$
\mathfrak{R}_{n,\ell}:=
\{
(r_{ij})\in I^{n}
\mid
r_{ij}\in R_{\ell}
\}.
$$
As in \citep{januszewski2011} (where $l$ is our $\ell$), $\mathfrak{R}_{n,\ell}$ is a system of representatives for $I^{n}/J^n_{\ell}$ and as such may be endowed with the natural group structure which is induced by matrix multiplication modulo $\mathfrak{f}^{\ell}$.

Define for any $\omega\in W(\GL_n,T_n),$\nomenclature[]{$\mathfrak{R}_{n,\ell}^\omega$}{defined as the intersection $\mathfrak{R}_{n,\ell}\cap \omega^{-1} B_n^-(\OO)\omega$}
$$
\mathfrak{R}_{n,\ell}^\omega:=
\mathfrak{R}_{n,\ell}\cap \omega^{-1} B_n^-(\OO)\omega.
$$
Then $\mathfrak{R}_{n,\ell}^{\omega}$ is a subgroup of $\mathfrak{R}_{n,\ell}$.

We consider the action of the compact torus
$$
T_n(\OO)=\left(\OO^\times\right)^n
$$
on $\GL_n(F)$, which for $\gamma=(\gamma_1,\dots,\gamma_n)\in T_n(\OO)$ is given by\nomenclature[]{$g^{\gamma}$}{right multiplication of $g$ with $\gamma\in T_n(\OO)$, extended to representatives}
$$
{}^\gamma\cdot:\GL_n(F)\to\GL_n(F),\;\;\;g\mapsto {}^{\gamma}g:=g\cdot \diag(\gamma_1,\dots,\gamma_n).
$$
Then $T_n(\OO)$ acts naturally on the set of representatives $\mathfrak{R}_{n,\ell}^\omega$ via its action on the quotient $I^n/J^n_{\ell}$. This action factors over the finite torus
$$
T_{n}(\OO/\mathfrak{f}^\ell)=\left(\OO^\times/(1+\mathfrak{f}^{\ell})\right)^n.
$$
The action of the latter on $\mathfrak{R}_{n,\ell}^\omega$ is faithful. We fix a system of representatives $\mathcal T_{n,\ell}\subseteq T_{n}(\OO)$ for $T_{n}(\OO/\mathfrak{f}^\ell)$\nomenclature[]{$\mathcal T_{n,\ell}$}{system of representatives for $T_{n}(\OO/\mathfrak{f}^\ell)$}.

If $\sigma\in S_n$ corresponds to $\omega$ and $\sigma(n)=n,$ set\nomenclature[]{$\tilde{\mathfrak{R}}_{n,\ell}^\omega$}{representatives with fixed last column $(f^{n-1},-f^{n-2},\dots,-1)^t$}
$$
\tilde{\mathfrak{R}}_{n,\ell}^\omega:=
\{
(r_{ij})\in\mathfrak{R}_{n,\ell}^\omega
\mid
r_{n1}= f^{n-1},\;r_{nj}= -f^{n-j},\;2\leq j\leq n
\}.
$$

\begin{proposition}[Proposition 2.4 in \citep{januszewski2011}]\label{prop:rtildebahn}
If $\sigma(n)=n$ we have for any $r\in\tilde{\mathfrak{R}}_{n,\ell}^\omega$
$$
\#\left(T_{n}(\OO)\cdot r\cap\tilde{\mathfrak{R}}_{n,\ell}^\omega\right)=
\absNorm(\mathfrak{f})^{\frac{n(n-1)}{2}}.
$$
In other words, the orbit of $r$ under the action of $T_n(\OO)$ on $\mathfrak{R}_{n,\ell}^\omega$ contains $\absNorm(\mathfrak{f})^\frac{n(n-1)}{2}$ elements of $\tilde{\mathfrak{R}}_{n,\ell}^\omega$.
\end{proposition}

As in \citep{januszewski2011}, define the matrices\nomenclature[]{$A_n$}{matrix in $I^n$ depending on $f$}\nomenclature[]{$B_n$}{matrix in $I^n$ depending on $f$}
$$
A_n:=
\begin{pmatrix}
1&f^{-1}&0&\hdots&0\\
0&1&-f^{-1}&\ddots&\vdots\\
\vdots&\ddots&\ddots&\ddots&0\\
\vdots&&\ddots&1&-f^{-1}\\
0&\hdots&\hdots&0&1\\
\end{pmatrix}\in I^n,
$$
$$
B_n:=
\begin{pmatrix}
1&0&\hdots&\hdots&0\\
f&-1&\ddots&&\vdots\\
0&f&-1&\ddots&\vdots\\
\vdots&\ddots&\ddots&\ddots&0\\
0&\hdots&0&f&-1\\
\end{pmatrix}\in I^n,
$$
and\nomenclature[]{$C_n$}{matrix in $I^n$ depending on $f$}
$$
C_n:=
\begin{pmatrix}
1&0&\cdots&\cdots&0\\
0&\ddots&\ddots&&\vdots\\
\vdots&\ddots&\ddots&\ddots&\vdots\\
0&\cdots&0&1&0\\
f^{n-1}&-f^{n-2}&\hdots&-f&-1
\end{pmatrix}\in I^n,
$$
subject to the convention that $B_1=C_1={\bf1}_1$ and $B_0:={\bf1}_0$. This guarantees that for all $n\geq 0$ relation \eqref{eq:bncnbn} below holds.

Define the projection\nomenclature[]{$p$}{$F$-linear projection $F^n\times F^n\to F^{n-1\times n-1}$}
$$
p:F^{n\times n}\to F^{n-1\times n-1},\;\;\;
(g_{ij})\mapsto (g_{ij})_{1\leq i,j\leq n-1}.
$$

\begin{proposition}[Proposition 2.5 in \citep{januszewski2011}]\label{prop:rtildeindex}
If $\sigma(n)=n$, we have for $\tilde{\omega}:=p(\omega)$, $\tilde{r}:=p(r)$,
$$
\#\tilde{\mathfrak{R}}_{n,\ell}^\omega
=
\#\mathfrak{R}_{n-1,\ell}^{\tilde{\omega}}.
$$
Furthermore, the projection $p$ induces a bijection
$$
p:
\tilde{\mathfrak{R}}_{n,\ell}^\omega\to
\mathfrak{R}_{n-1,\ell}^{\tilde{\omega}}
$$
and
$$
\GL_{n-1}(F)\to\GL_n(F),\;\;\;
\tilde{g}\mapsto j_{n-1}(\tilde{g})\cdot C_n
$$
induces the inverse of $p$.
\end{proposition}

Recall the well known decomposition
\begin{equation}
\GL_n(F)=\bigsqcup_{
\begin{subarray}{c}
\omega\in W(\GL_n,T_n)\\e\in\ZZ^n
\end{subarray}
}U_n(F)\varpi^e\omega I^n.
\label{eq:iwasawa}
\end{equation}
by Iwahori-Matsumoto \citep{iwahorimatsumoto1965}, Proposition 2.33, and Satake \citep{satake1963}, section 8.2. The definition of $\mathfrak{R}_{n,\ell}^{\omega}$ is justified by the following refinement of \eqref{eq:iwasawa}.

\begin{proposition}[Proposition 2.2 in \citep{januszewski2011}]\label{prop:repr}
The set $\varpi^e \omega \mathfrak{R}_{n,\ell}^\omega$ is a system of representatives for the double cosets
$$
U_n(F)\varpi^e\omega r J^n_{\ell},\;\;\;r\in I^n
$$
in $U_n(F)\varpi^e\omega I^n$. Fix an $\ell(e)\geq 2n$ for any $e\in\ZZ^n$. Then
$$
\GL_n(F)=\bigsqcup_{
\begin{subarray}{c}
e\in\ZZ^n\\
\omega\in W(\GL_n,T_n)\\
r\in\mathfrak{R}_{n,\ell(e)}^\omega
\end{subarray}
} U_n(F)\varpi^e\omega r J^n_{\ell(e)}.
$$
\end{proposition}

We will also need the following refinement of Proposition 2.3 in \citep{januszewski2011}.

\begin{proposition}\label{prop:ujvolumen}
For any $e\in\ZZ^n$, $\omega\in W(\GL_n,T_n)$ and $r\in\mathfrak{R}_{n,\ell}^\omega$ the measure
$$
\int_{U_n(\OO)\varpi^e\omega r J^n_{\ell}}dg
$$
is independent of $\omega$ and $r$. If $e\in\ZZ^n$ and $\ell>0$, then
$$
\int_{U_n(\OO)\varpi^e\omega r J^n_{\ell}}dg=
[\varpi^e U_n(\OO)\varpi^{-e}:U_n(\OO)]\cdot
\prod_{\mu=1}^n
\left({1-\absNorm(\mathfrak{p})^{-\mu}}\right)^{-1}
\cdot
\absNorm(\mathfrak{f})^{-\frac{\ell (n+1)n}{2}}.
$$
\end{proposition}

\begin{proof}
  The first statement and the second in the case $e=0$ follows as in Proposition 2.3 in \citep{januszewski2011}. For general dominant $e\in\ZZ^n,$ observe that
  $$
  \int_{U_n(\OO)\varpi^e J^n_{\ell}}dg\;=\;
  [\varpi^e U_n(\OO)\varpi^{-e}:U_n(\OO)]\cdot
  \int_{U_n(\OO)\omega r J^n_{\ell}}dg,
  $$
  whence the claim.
\end{proof}

Let $\theta_1,\dots,\theta_{n+1}:F^\times\to\CC^\times$ denote quasi-characters with conductors dividing $f$\nomenclature[]{$\theta_\mu$}{quasi-character $F^\times\to\CC^\times$ of conductor dividing $f$}. These characters give rise to a character\nomenclature[]{$\theta$}{extension of the characters $\theta_1,\dots,\theta_{n+1}$ to $I^{n+1}_{\alpha}\to\CC^\times$}
\begin{align}
  \theta:\quad&I^{n+1}_{\alpha}\to\CC^\times,
  \label{eq:Inplusonecharacter}\\
  &(r_{ij})_{ij}\;\mapsto\;\prod_{i=1}^{n+1}\theta_i(r_{ii}).\nonumber
\end{align}
Assume furthermore given another set of (finite order) characters $\theta_1',\dots,\theta_n'$ again with conductors dividing $f.$\nomenclature[]{$\theta_\nu'$}{quasi-character $F^\times\to\CC^\times$ of conductor dividing $f$} These give likewise rise to a character\nomenclature[]{$\theta'$}{extension of the characters $\theta_1',\dots,\theta_{n}'$ to $I^{n}_{\alpha}\to\CC^\times$}
\begin{equation}
  \theta':\quad I^{n}_{\alpha}\to\CC^\times.
  \label{eq:Incharacter}
\end{equation}
Remark that
$$
J_{1}^{n+1}\;\subseteq\;\kernel\theta,\quad J^n_{1}\;\subseteq\;\kernel\theta'.
$$

Let $w$ and $v$ denote $\psi$- resp.\ $\psi^{-1}$-Whittaker functions on $\GL_{n+1}(F)$ resp.\ $\GL_{n}(F)$\nomenclature[]{$w$}{local $\psi$-Whittaker function on $\GL_{n+1}(F)$}\nomenclature[]{$v$}{local $\psi^{-1}$-Whittaker function on $\GL_n(F)$}, with the additional property that $w$ and $v$ transform under $I^{n+1}_{\alpha}$ and $I^{n}_{\alpha}$ (from the right) via $\theta$ resp.\ $\theta'$:
\begin{align}
\forall g\in\GL_{n+1}(F),r\in I^{n+1}_{\alpha}:& &w(gr)\;=\;\theta(r)\cdot w(g),\label{eq:wtransform}\\
\forall g\in\GL_{n}(F),r\in I^{n}_{\alpha}:& &v(gr)\;=\;\theta'(r)\cdot v(g).\label{eq:vtransform}
\end{align}

We need the following statement which generalizes Lemma 2.6 in \citep{januszewski2011} and Lemma 4.1 in \citep{januszewski2015}.

\begin{lemma}\label{lem:birchinductionrelation}
We have for any $\delta\in\ZZ$ an identity
$$
w\left(j_{n,-\delta}(g)C_{n+1}\cdot D_{n+1}w_{n+1}\right) v(g)=
$$
\begin{equation}
  \psi\left(\lambda_n^\delta(gB_n)\right) w\left(j_{n,-\delta}(gB_n\cdot D_nw_n)\right) v(gB_n)
  \cdot\theta'(B_n^{-1}).
\label{eq:lemma26}
\end{equation}\end{lemma}

\begin{proof}
  The proof proceeds as the proof of Lemma 4.1 in \citep{januszewski2015}, with the following additional observations:
  $$
  w_{n+1}D_{n+1}^{-1}A_{n+1}D_{n+1}w_{n+1}\;\in\;J_{1}^{n+1}\;\subseteq\;\kernel\theta,
  $$
  and
  $$
  v(g)\;=\;\theta'(B_n^{-1})\cdot v(gB_n).
  $$
\end{proof}

Let $\chi$ be a quasi-character of $F$. We suppose that the condition
\begin{equation}
\alpha\;\geq\;e_\chi
\label{eq:chiconductorcondition}
\end{equation}
is satisfied in all what follows. Let $e\in\ZZ^n$, $\omega\in W(\GL_n,T_n)$,
\begin{equation}
\ell\;\geq\;\max\{2n,n-e_1/\alpha,\dots,n-e_n/\alpha\},
\label{eq:ellcondition}
\end{equation}
and $\delta\in\ZZ$.

The decomposition of $\mathfrak{R}_{n,\ell}^\omega$ into $T_n(\OO)$-orbits leads to partial sums\nomenclature[]{$Z(s;w,v,\delta,e,\omega,r)$}{local partial sum}
\begin{align*}
Z_n(s;w,v,\delta,e,\omega,r)\;:=\;&
\sum_{\gamma\in \mathcal T_{n,\ell}}
\psi(\lambda_n^{\delta}(\varpi^e\omega {}^\gamma r))\cdot
w\left(
\varpi^e\omega {}^\gamma r\cdot D_n w_n
\right)\cdot\\
&
v(\varpi^e\omega{}^\gamma r)\cdot
\chi(\varpi^e\omega{}^\gamma r)\cdot
\absnorm{\det(\varpi^e\omega{}^\gamma r)}^{s-\frac{1}{2}},
\end{align*}
for any $r\in\mathfrak{R}_{n,\ell}^\omega$ and $s\in\CC$.

For any $g\in \GL_{n+1}(F)$ set\nomenclature[]{$\theta^g$}{conjugate of $\theta$ by $g$}
$$
\theta^g(r)\;:=\;\theta(grg^{-1}).
$$

\begin{lemma}\label{lem:preinduction}
  Let $s\in\CC$, $e\in\ZZ^n$, $\omega\in W(\GL_n,T_n)$, $\ell\in\ZZ$ subject to \eqref{eq:ellcondition}, and $\delta\in\ZZ$. Then
  \begin{itemize}
  \item[(i)] $Z_n(s;w,v,\delta,e,\omega,r)$ is independent of the choice of $\mathcal T_{n,\ell}$.
  \item[(ii)] Assume that for all $1\leq\nu\leq n$ and $\mu=n+1-\nu$,
\begin{equation}
  e_{\chi\theta_\mu\theta_\nu'}\;>\;0,
  \label{eq:singlecharactercondition}
\end{equation}
is satisfied. Then $Z_n(s;w,v,\delta,e,\omega,r)$ vanishes unless
\begin{equation}
  e_n\;=\;\delta+\alpha\cdot(n+1-\sigma(n))-e_{\chi\theta_{n+1-\sigma(n)}\theta_{\sigma(n)}'}.
  \label{eq:encondition}
\end{equation}
\item[(iii)] If conditions \eqref{eq:singlecharactercondition} and \eqref{eq:encondition} are satisfied, and if the exponent in \eqref{eq:singlecharactercondition} is independent of $\nu$, then $Z_n(s;w,v,\delta,e,\omega,r)$ vanishes unless $\sigma(n)=n$ and for $1\leq\nu\leq n$,
\begin{equation}
\absnorm{r_{n\nu}}\;=\;\absnorm{f^{n-\nu}}.
  \label{eq:rnormcondition}
\end{equation}
\item[(iv)] If the hypotheses of (iii) are satisfied, we may assume without loss of generality that
\begin{equation}
  r_{n1}=f^{n-1},\quad\text{and}\quad r_{n\nu}\;=\;-f^{n-\nu}\;\text{for}\;2\leq\nu\leq n.
  \label{eq:explicitchoiceforrn}
\end{equation}
If additionally, \eqref{eq:singlecharactercondition} holds for all $1\leq\nu\leq\mu\leq n$, then
\begin{align*}
  Z_n(s;w,v,\delta,e,\omega,r)\;=\;&
\chi\theta^{w_n}\theta'(B_n)\cdot
\prod_{\nu=1}^n
\chi\theta_{\nu}^{w_n}\theta'_\nu\left(f_{\chi\theta^{w_n}_\nu\theta_\nu'}\right)\cdot
\absNorm(\mathfrak{f}^\ell/\mathfrak{f}_{\chi\theta^{w_n}_\nu\theta_\nu'})\cdot
G(\chi\theta^{w_n}_\nu\theta_\nu')\cdot\\
&
w\left(
\varpi^e\omega r\cdot D_n w_n
\right)\cdot
v(\varpi^e\omega r)\cdot
\chi(\varpi^e\omega r)\cdot\absnorm{\det(\varpi^e)}^{s-\frac{1}{2}}.
\end{align*}
  \end{itemize}
  \end{lemma}

\begin{proof}
The claimed property of $Z_n(s;w,v,\delta,e,\omega,r)$ in (i) is clear by definition.

The relation
$$
\varpi^e\omega {}^\gamma r\cdot D_n w_n=
\varpi^e\omega r\cdot D_n  w_n\cdot \left(w_n{}^\gamma {\bf1}_n w_n\right),
$$
with
$$
w_n{}^\gamma {\bf1}_n w_n\in I^{n}_{\alpha},
$$
yields
\begin{align*}
Z_n(s;w,v,\delta,e,\omega,r)\;=\;&
\chi\theta'(r)\cdot
\sum_{\gamma\in\mathcal T_{n,\ell}}
\chi\theta^{w_n}\theta'\left({}^\gamma{\bf1}_n\right)\cdot
\psi(\lambda_n^0(\varpi^{e-(\delta)}\omega{}^\gamma r))\cdot\\
&w\left(
\varpi^e\omega r\cdot D_n w_n
\right)\cdot
v(\varpi^e\omega)\cdot
\absnorm{\det(\varpi^e)}^{s-\frac{1}{2}}\cdot\chi(\varpi^e\omega).
\end{align*}
Unfolding gives
$$
\psi(\lambda_n(\varpi^{e-(\delta)}\omega{}^\gamma r))=
\prod_{\nu=1}^n
\psi\left(\varpi^{e_n-\delta} f^{\nu-n-1} r_{\sigma(n)\nu}\cdot\gamma_\nu\right),
$$
which in turn shows that
\begin{align}
&\sum_{\gamma\in\mathcal T_{n,\ell}}
\chi\theta^{w_n}\theta'\left({}^\gamma{\bf1}_n\right)\cdot
\psi(\lambda_n(\varpi^{e-(\delta)}\omega{}^\gamma r))
\nonumber
\\
=\;&
\prod_{\nu=1}^{n}
\sum_{\gamma_\nu\in\left(\OO/\mathfrak{f}^{l}\right)^\times}
\chi
\theta_\nu'
\theta^{w_n}_\nu(\gamma_\nu)\cdot
\psi\left(\varpi^{e_n-\delta} f^{\nu-n-1} r_{\sigma(n)\nu}\cdot\gamma_\nu\right).
\label{eq:partialgausssum}
\end{align}
In the case $\nu=\sigma(n),$ the entry $r_{\sigma(n)\nu}$ is a unit, and hypothesis \eqref{eq:singlecharactercondition} together with the vanishing relation \eqref{eq:gaussnull} therefore implies (ii).

Under the hypotheses of (iii), if $\sigma(n)\neq n$, we have $\sigma(n) < n$ and we see with \eqref{eq:encondition} that therefore
$$
e_n-\delta-\alpha\;>\;-e_{\chi\theta^{w_n}_\nu\theta_\nu'}.
$$
for $1\leq\nu\leq n$ (recall that the right hand side is independent of $\nu$ by our hypothesis). In the case $\nu=n$, we obtain
$$
\absnorm{\varpi^{e_n-\delta} f^{n-n-1} r_{\sigma(n)n}\cdot\gamma_{n}}\leq
\absnorm{\varpi^{e_n-\delta} f^{-1}}
<\absnorm{f_{\chi\theta_{1}\theta_n'}^{-1}},
$$
whence $Z_n(s;w,v,\delta,e,\omega,r)$ vanishes by \eqref{eq:gaussnull}.

In the case $\sigma(n)=n$, by \eqref{eq:gaussnull} once again any violation of condition \eqref{eq:rnormcondition} implies vanishing of \eqref{eq:partialgausssum}. This proves (iii).

In case (iv), we may by (i) and \eqref{eq:rnormcondition} assume without loss of generality that $r_{n\bullet}$ is chosen as in \eqref{eq:explicitchoiceforrn}. Then by \eqref{eq:gaussnull}, the value of \eqref{eq:partialgausssum} is given by
$$
\chi\theta^{w_n}\theta'(B_n)
\prod_{\nu=1}^n
\chi\theta^{w_n}_\nu\theta_\nu'\left(f_{\chi\theta^{w_n}_\nu\theta_\nu'}\right)\cdot
\absNorm(\mathfrak{f}^\ell/\mathfrak{f}_{\chi\theta^{w_n}_\nu\theta_\nu'})\cdot
G(\chi\theta^{w_n}_\nu\theta_\nu'),
$$
and (iv) follows.
\end{proof}

Under the hypotheses in statement (iii) of the previous lemma set for $\gamma\in\ZZ$\nomenclature[]{$d_{n,\gamma}$}{multi-exponent in $\ZZ^n$}
$$
d_{n,\gamma}:=(\gamma\cdot(n+1-i))_{1\leq i\leq n}\in\ZZ^n.
$$
For $1\leq\mu\leq n$ we denote by $w_\mu\in\GL(n+1)$ the long Weyl element in $S_\mu$ as embedded as antidiagonal matrix in the upper $\mu\times\mu$-block.
\begin{lemma}\label{lem:zentrales}
  Assume that for all $1\leq\mu\leq n$ and all $1\leq\nu\leq\mu:$
  $$
    e_{\chi\theta^{w_\mu}_\nu\theta_\nu'}\;>\;0.
  $$
Assume additionally that these exponents are independent of $\mu$ and $\nu$. Let $\mathfrak{f}_{\chi\theta\theta'}$ denote the common conductor of $\chi\theta_\nu^{w_\mu}\theta'_\nu$ and define
$$
\gamma\;:=\;\alpha-e_{\chi\theta^{w_n}_\nu\theta_\nu'}.
$$
Then for all $e\in\ZZ^n$, $\omega\in W(\GL_n,T_n)$, $\ell\geq\max\{2n,n-e_1/\alpha,\dots,n-e_n/\alpha\}$ and $\delta\in\ZZ,$ we have
\begin{align*}
  &
  \vol\left({U_n(\OO)\varpi^e J^n_{\ell}}\right)\cdot\!\!\!
  \sum_{g\in\varpi^{e}\omega \mathfrak{R}_{n,\ell}^{\omega}}
\!\!\!\psi(\lambda_n^{\delta}(g))\cdot
w(
g\cdot D_n w_n
)\cdot
v(g)\cdot
\chi(\det(g))\cdot
\absnorm{\det(g)}^{s-\frac{1}{2}}\\
=\;&\absNorm(\mathfrak{f}_{\chi\theta\theta'})^{
  -\frac{(n+1)n(n-1)}{2}}\cdot
  \prod_{\mu=1}^n \theta^{w_\mu}(B_\mu)\cdot \prod_{\nu=1}^\mu\absNorm(\mathfrak{f}_{\chi\theta^{w_\mu}_\nu\theta_\nu'})^{-1}\cdot\chi\theta^{w_\mu}_\nu\theta_\nu'(f_{\chi\theta^{w_\mu}_\nu\theta_\nu'})
  \cdot G(\chi\theta^{w_\mu}_\nu\theta_\nu')\cdot\\
  &
   w(\varpi^{e}) v(\varpi^{e})
\chi(\varpi^{e})
\absnorm{\varpi^{e}}^{s-\frac{1}{2}},
\end{align*}
if $(e,\omega)=(d_{n,\gamma}+(\delta),{\bf1}_n)$ and $0$ otherwise.
\end{lemma}

\begin{proof}
We proceed by induction on $n$. If $n=0$, then $W_0=\GL_0(\OO)=\{{\bf1}_0\}$, $\mathfrak{R}_0^\omega=\{{\bf1}_0\}$, $\ZZ^0=\{0\}$. The case $\omega\neq{\bf1}_0$ or $e\neq 0$ actually never occurs. This concludes the case $n=0$. Now let $n\geq 1$ and suppose that the claim is true for $n-1$.


By the constancy of the partial sums on $T_n(\OO/\mathfrak{f}^\ell)$-orbits, we have
\begin{equation}
\sum_{g\in\varpi^{e}\omega \mathfrak{R}_{n,\ell}^{\omega}}
\psi(\lambda_n^{\delta}(g))
w(g\cdot D_n w_n)
v(g)\chi(g)
\absnorm{\det(g)}^{s-\frac{1}{2}}\;=\;
\absNorm(\mathfrak{f})^{-\ell n}\cdot\!\!\!
\sum_{r\in\mathfrak{R}_{n,\ell}^{\omega}}
Z_n(s;w,v,\delta,e,\omega,r).
\label{eq:partialsummen}
\end{equation}

By Lemma \ref{lem:preinduction} (iii), we already know that this expression vanishes unless $e_n$ is given by \eqref{eq:encondition} and $\sigma(n)=n$, which we henceforth assume. In that case the summands on the right hand side of \eqref{eq:partialsummen} vanish unless the representatives $r$ satisfy \eqref{eq:rnormcondition}.

By Propositions \ref{prop:rtildebahn} and \ref{prop:rtildeindex}, the right hand side of \eqref{eq:partialsummen} therefore simplifies to
$$
\absNorm(\mathfrak{f})^{-\frac{n(n-1)}{2}}\cdot\!\!\!
\sum_{r\in\tilde{\mathfrak{R}}_{n,\ell}^\omega}
Z_n(s;w,v,\delta,e,\omega,r)\;=\;
\absNorm(\mathfrak{f})^{-\frac{n(n-1)}{2}}\cdot\!\!\!
\sum_{\tilde{r}\in\mathfrak{R}_{n-1,\ell}^{p(\omega)}}
Z_n(s;w,v,\delta,e,\omega,j_{n-1,0}(\tilde{r})C_{n}).
$$
Statement (iv) in Lemma \ref{lem:preinduction} together with Lemma \ref{lem:birchinductionrelation} yields with the abbreviations $\tilde{\omega}:=p(\omega)$ and $\tilde{e}:=(e_\nu)_{1\leq\nu\leq n-1}$\nomenclature[]{$\tilde{\omega}$}{projection of $\omega$ to $\GL_{n-1}(F)$}\nomenclature[]{$\tilde{e}$}{projection of $e$ to $\ZZ^{n-1}$}, and the relation
\begin{equation}
  j_{n-1,0}(B_{n-1}^{-1})C_n\;=\;B_n^{-1},
  \label{eq:bncnbn}
\end{equation}
the identity
\begin{align*}
  &Z_n(s;w,v,\delta,e,\omega,j_{n-1,0}(\tilde{r})C_{n} )\\
  =\;&
  \absnorm{\varpi^{e_n}}^{s-\frac{1}{2}}\cdot\chi(\varpi^{e_n})\cdot
  \chi\theta^{w_n}\theta'(B_n)\cdot
  \prod_{\nu=1}^n
  \chi\theta_{n+1-\nu}\theta'_\nu\left(f_{\chi\theta^{w_n}_\nu\theta_\nu'}\right)\cdot
  \absNorm(\mathfrak{f}^\ell/\mathfrak{f}_{\chi\theta^{w_n}_\nu\theta_\nu'})\cdot
  G(\chi\theta^{w_n}_\nu\theta_\nu')\cdot\\
  &
  w\left(
  j_{n,e_n}(\varpi^{\tilde{e}}\tilde{\omega}\tilde{r})C_{n}\cdot D_n w_n
  \right)\cdot
  v(j_{n,e_n}(\varpi^{\tilde{e}}\tilde{\omega}\tilde{r})C_n)\cdot
  \chi(j_{n,e_n}(\varpi^{\tilde{e}}\tilde{\omega}\tilde{r})C_n)\cdot
  \absnorm{\varpi^{\tilde{e}}}^{s-\frac{1}{2}}\\
  =\;&
  \prod_{\nu=1}^n
  \chi\theta_{n+1-\nu}\theta'_\nu\left(f_{\chi\theta^{w_n}_\nu\theta_\nu'}\right)\cdot
  \absNorm(\mathfrak{f}^\ell/\mathfrak{f}_{\chi\theta^{w_n}_\nu\theta_\nu'})\cdot
  G(\chi\theta^{w_n}_\nu\theta_\nu')\cdot\\
  &
  \absnorm{\varpi^{e_n}}^{s-\frac{1}{2}}\cdot\chi(\varpi^{e_n})\cdot
  \chi\theta'(C_n)\cdot
  \chi\theta^{w_n}\theta'(B_n)\cdot\\
  &
  w\left(
  j_{n,e_n}(\varpi^{\tilde{e}}\tilde{\omega}\tilde{r})C_{n}\cdot D_n w_n
  \right)\cdot
  v(j_{n,e_n}(\varpi^{\tilde{e}}\tilde{\omega}\tilde{r}))\cdot
  \chi(j_{n,e_n}(\varpi^{\tilde{e}}\tilde{\omega}\tilde{r}))\cdot
  \absnorm{\varpi^{\tilde{e}}}^{s-\frac{1}{2}}\\
  =\;&
  \prod_{\nu=1}^n
  \chi\theta_{n+1-\nu}\theta'_\nu\left(f_{\chi\theta^{w_n}_\nu\theta_\nu'}\right)\cdot
  \absNorm(\mathfrak{f}^\ell/\mathfrak{f}_{\chi\theta^{w_n}_\nu\theta_\nu'})\cdot
  G(\chi\theta^{w_n}_\nu\theta_\nu')\cdot\\
  &
  \absnorm{\varpi^{e_n}}^{s-\frac{1}{2}}\cdot\chi(\varpi^{e_n})\cdot
  \chi\theta'\left(j_{n-1,0}(B_{n-1}^{-1})C_n\right)\cdot
  \chi\theta^{w_n}\theta'(B_n)\cdot\\
  &
  \psi(\lambda_{n-1}^{-e_n}(\varpi^{\tilde{e}}\tilde{\omega}\tilde{r}B_{n-1})\cdot
  w\left(
  j_{n,e_n}(\varpi^{\tilde{e}}\tilde{\omega}\tilde{r}B_{n-1}\cdot D_{n-1} w_{n-1})
  \right)\cdot
  v(j_{n,e_n}(\varpi^{\tilde{e}}\tilde{\omega}\tilde{r}B_{n-1}))\cdot\\
  &
  \chi(j_{n,e_n}(\varpi^{\tilde{e}}\tilde{\omega}\tilde{r}B_{n-1}))\cdot
  \absnorm{\varpi^{\tilde{e}}}^{s-\frac{1}{2}}\\
  =\;&
  \prod_{\nu=1}^n
  \chi\theta_{n+1-\nu}\theta'_\nu\left(f_{\chi\theta^{w_n}_\nu\theta_\nu'}\right)\cdot
  \absNorm(\mathfrak{f}^\ell/\mathfrak{f}_{\chi\theta^{w_n}_\nu\theta_\nu'})\cdot
  G(\chi\theta^{w_n}_\nu\theta_\nu')\cdot\\
  &
  \absnorm{\varpi^{e_n}}^{s-\frac{1}{2}}\cdot\chi(\varpi^{e_n})\cdot
  \theta^{w_n}(B_n)\cdot\\
  &
  \psi(\lambda_{n-1}^{-e_n}(\varpi^{\tilde{e}}\tilde{\omega}\tilde{r}B_{n-1})\cdot
  w\left(
  j_{n,e_n}(\varpi^{\tilde{e}}\tilde{\omega}\tilde{r}B_{n-1}\cdot D_{n-1} w_{n-1})
  \right)\cdot
  v(j_{n,e_n}(\varpi^{\tilde{e}}\tilde{\omega}\tilde{r}B_{n-1}))\cdot\\
  &
  \chi(j_{n,e_n}(\varpi^{\tilde{e}}\tilde{\omega}\tilde{r}B_{n-1}))\cdot
  \absnorm{\varpi^{\tilde{e}}\tilde{\omega}\tilde{r}B_{n-1}}^{s-\frac{1}{2}}.
\end{align*}

Right multiplication by $B_{n-1}\in I_{n-1}$ permutes the double cosets $U_{n-1}(F)\varpi^{\tilde{e}} \tilde{r} J_{n-1,\ell}$, but leaves the double cosets $U_{n-1}(F)\varpi^{\tilde{e}} I_{n-1}$ invariant, whence induces a permutation of the system of representatives $\mathfrak{R}_{n-1,\ell}^{\tilde{\omega}}$.

Therefore, summing over all $\tilde{r}\in\mathfrak{R}_{n-1,\ell}^{\tilde{\omega}}$, gives
\begin{align}
  &\sum_{\tilde{r}\in\tilde{\mathfrak{R}}_{n-1,\ell}^{\tilde{\omega}}}
\psi\left(
\lambda_{n-1}^{-e_n}(\varpi^{\tilde{e}}\tilde{\omega}\tilde{r}B_{n-1})
\right)\cdot\nonumber\\
&\quad\quad\quad
w\left(
j_{n-1,e_n}(\varpi^{\tilde{e}}\tilde{\omega}\tilde{r}B_{n-1}\cdot D_{n-1} w_{n-1})
\right)\cdot
v\left(j_{n-1,e_n}(\varpi^{\tilde{e}}\tilde{\omega}\tilde{r}B_{n-1})\right)\cdot\nonumber\\
&\quad\quad\quad\absnorm{\det(\varpi^{\tilde{e}}\tilde{\omega}\tilde{r}B_{n-1})}^{s-\frac{1}{2}}\cdot\chi(\varpi^{\tilde{e}}\tilde{\omega}\tilde{r}B_{n-1})\nonumber\\
=\;&
\sum_{\tilde{r}\in\tilde{\mathfrak{R}}_{n-1,\ell}^{\tilde{\omega}}}
\psi\left(
\lambda_{n-1}^{-e_n}(\varpi^{\tilde{e}}\tilde{\omega}\tilde{r})
\right)\cdot\nonumber\\
&\quad\quad\quad w\left(
j_{n-1,e_n}(\varpi^{\tilde{e}}\tilde{\omega}\tilde{r}\cdot D_{n-1} w_{n-1})
\right)\cdot
v\left(j_{n-1,e_n}(\varpi^{\tilde{e}}\tilde{\omega}\tilde{r})\right)\cdot\label{eq:inductionsum}\\
&\quad\quad\quad\absnorm{\det(\varpi^{\tilde{e}}\tilde{\omega}\tilde{r})}^{s-\frac{1}{2}}\cdot\chi(\varpi^{\tilde{e}}\tilde{\omega}\tilde{r}).\nonumber
\end{align}

By the induction hypothesis this expression vanishes unless
\begin{equation}
  (\tilde{e},\tilde{\omega})\;=\;(d_{n-1,e_n}+(e_n),{\bf1}_{n-1}).
  \label{eq:nminusonevanishingcondition}
\end{equation}
With \eqref{eq:encondition} we see that $e_n=\gamma+\delta$, whence
\begin{align*}
  \left(d_{n-1,\gamma}+(e_n)\right)_i\;=\;&\gamma\cdot(n+1-i)+\delta\\
  =\;&\left(d_{n,\gamma}+(\delta)\right)_i.
\end{align*}
Therefore, condition \eqref{eq:nminusonevanishingcondition} is equivalent to
$$
(e,\omega)\;=\;(d_{n,\gamma}+(\delta),{\bf1}_{n}).
$$
Under condition \eqref{eq:nminusonevanishingcondition}, the induction hypothesis shows with Proposition \ref{prop:ujvolumen} that \eqref{eq:inductionsum} takes the value
\begin{align*}
  &\absNorm(\mathfrak{f})^{\frac{\ell n(n-1)}{2}-\frac{n(n-1)(n-2)}{6}}\cdot
  \prod_{\mu=1}^{n-1}\prod_{\nu=1}^\mu\absNorm(\mathfrak{f}_{\chi\theta^{w_\mu}_\nu\theta_\nu'})^{-1}\cdot\\
  &
  \prod_{\mu=1}^{n-1}\prod_{\nu=1}^\mu\chi\theta^{w_\mu}_\nu\theta_\nu'(f_{\chi\theta^{w_\mu}_\nu\theta_\nu'})
  \cdot G(\chi\theta^{w_\mu}_\nu\theta_\nu')\cdot\\
  &
  \prod_{\mu=1}^{n-1}\theta^{w_\mu}(B_\mu)\cdot
   w(\varpi^{e}) v(\varpi^{e})\chi(\varpi^{\tilde{e}})
  \absnorm{\varpi^{\tilde{e}}}^{s-\frac{1}{2}}.
\end{align*}
Whence \eqref{eq:partialsummen} is given by
\begin{align*}
  &\absNorm(\mathfrak{f})^{\frac{\ell n(n-1)}{2}+\ell n -\left(\frac{n(n-1)(n-2)}{6}+\frac{n(n-1)}{2}\right)}\cdot\\
  &
  \prod_{\mu=1}^n\prod_{\nu=1}^\mu
  \absNorm\left(\mathfrak{f}_{\chi\theta^{w_\mu}_\nu\theta_\nu'}\right)^{-1}\cdot
  \chi\theta_{\nu}^{w_\mu}\theta'_\nu\left(f_{\chi\theta^{w_\mu}_\nu\theta_\nu'}\right)\cdot
  G(\chi\theta^{w_n}_\nu\theta_\nu')\cdot\\
  &
  \prod_{\mu=1}^n\theta^{w_\mu}(B_\mu)\cdot
  w(\varpi^{e})\cdot
  v(\varpi^{e})\cdot
  \chi(\varpi^{e})\cdot
  \absnorm{\varpi^{e}}^{s-\frac{1}{2}},
\end{align*}
and the claim follows by another application of Proposition \ref{prop:ujvolumen}.
\end{proof}

Recall the definition of the matrix $t_x$ in \eqref{eq:definitionoftx}.
\begin{theorem}[Local Birch Lemma]\label{thm:localbirch}
  Let $w$ and $v$ be $\psi$- (resp. $\psi^{-1}$-) Whittaker functions on $\GL_{n+1}(F)$ resp.\ $\GL_n(F)$, which satisfy relations \eqref{eq:wtransform} and \eqref{eq:vtransform}. Assume furthermore that $\chi:F^\times\to\CC^\times$ is a quasi-character with the property that for all $1\leq\mu\leq n$ and all $1\leq\nu\leq\mu$ the conductors of $\chi\theta^{w_\mu}_\nu\theta_\nu'$ are non-trivial, all agree, and are generated by an element $f_{\chi\theta\theta'}=\varpi^{e_{\chi\theta\theta'}}$. Set $\mathfrak{f}_{\chi\theta\theta'}:=\OO f_{\chi\theta\theta'}$. Then for every $s\in\CC$,
  \begin{align*}
    &\int\limits_{U_{n}(F)\backslash{}\GL_{n}(F)}
    w\left(
    j_{n,0}(g)
    \cdot h_n\cdot j_{n}(t_{-f})
    \right)
    v\left(g\cdot t_{f}\right)
    \chi(\det(g))
    \absnorm{\det(g)}^{s-\frac{1}{2}}dg\\
    =\;&
    \prod_{\mu=1}^{n}
    \left({1-q^{-\mu}}\right)^{-1}
    \cdot
    \absNorm(\mathfrak{f}_{\chi\theta\theta'})^{-\frac{(n+2)(n+1)n}{6}}\cdot
    \absnorm{t_{f_{\chi\theta\theta'}}}^{\frac{1}{2}-s}\cdot\\
    &
    \prod_{\mu=1}^n\prod_{\nu=1}^\mu\left[\theta^{w_\mu}_\nu\theta_\nu'(f_{\chi\theta\theta'})
    \cdot G(\chi\theta^{w_\mu}_\nu\theta_\nu')\right]\cdot
    w(j_{n,0}(t_{f f_{\chi\theta\theta'}^{-1}}))\cdot
    v(t_{f f_{\chi\theta\theta'}^{-1}}).
  \end{align*}
\end{theorem}

We emphasize that there are modifications in the formulation of Theorem \ref{thm:localbirch} compared to previous statements in \citep{schmidt2001,januszewski2011,januszewski2015,januszewski2016}. In particular, the matrix
$$
h_n\cdot j_{n}(t_{(-1)}).
$$
plays the role of the matrix $h_n$ in loc.\ cit.

\begin{proof}
  Introduce the matrix\nomenclature[]{$\tilde{E}_{n+1}$}{matrix in $J_{n+1},1$}
  $$
    \tilde{E}_{n+1}\;:=\;
    \begin{pmatrix}
      1&-f&-f^2&\cdots&-f^n\\
      0&1&f&\cdots&f^{n-1}\\
      \vdots&\ddots&\ddots&\ddots&\vdots\\
      \vdots&&\ddots&\ddots&f\\
      0&\cdots&\cdots&0&1
    \end{pmatrix}\;\in\;J_{n+1,1}.
  $$
  In the notation of \citep{januszewski2011} this matrix agrees with $w_{n+1}E_{n+1}$. A direct computation shows
  \begin{equation}
    h_n\cdot j_n(t_{f})\;=\;
    j_n(t_{f})\cdot B_{n+1}^{-1}\cdot D_{n+1}\cdot w_{n+1}\cdot\tilde{E}_{n+1}.
    \label{eq:hdecomposition}
  \end{equation}
  We deduce the relation
  $$
    w\left(
    j_n(g)
    \cdot h_n\cdot j_n(t_{-f})
    \right)\;=\;
    \theta(t_{-1})
    w\left(
    j_n(g)
    \cdot j_n(t_{f})\cdot B_{n+1}^{-1}\cdot D_{n+1}w_{n+1}
    \right).
  $$
  Together with \eqref{eq:bncnbn} and the right invariance of the Haar measure $dg$, we obtain
  \begin{align*}
    &
    \int
    w\left(
    j_n(g)
    \cdot h_n\cdot j_n(t_{f})
    \right)
    v\left(g\cdot t_{f}\right)
    \chi(g)
    \absnorm{g}^{s-\frac{1}{2}}dg\\
    =\;&
    \chi(t_{f}^{-1})\absnorm{t_{f}}^{\frac{1}{2}-s}
    \int
    w\left(
    j_n(gB_n^{-1})
    \cdot C_{n+1} D_{n+1}w_{n+1}
    \right)
    v(g)
    \chi(g)
    \absnorm{g}^{s-\frac{1}{2}}dg\\
    =\;&
    \chi(t_{f}^{-1})\absnorm{t_{f}}^{\frac{1}{2}-s}
    \int
    w\left(
    j_n(g)
    \cdot C_{n+1} D_{n+1}w_{n+1}
    \right)
    v(gB_n)
    \chi(gB_n)
    \absnorm{gB_n}^{s-\frac{1}{2}}dg\\
    =\;&
    \chi\theta'(B_n)
    \chi(t_{f}^{-1})\absnorm{t_{f}}^{\frac{1}{2}-s}
    \int
    w\left(
    j_n(g)
    \cdot C_{n+1} D_{n+1}w_{n+1}
    \right)
    v(g)
    \chi(g)
    \absnorm{g}^{s-\frac{1}{2}}dg.
  \end{align*}
  At this point, invoke Lemma \ref{lem:birchinductionrelation} once again to obtain the expression
  \begin{align*}
    &\chi(B_n)
    \chi(t_{f}^{-1})\absnorm{t_{f}}^{\frac{1}{2}-s}
    \int
    \psi\left(\lambda_n^0(gB_n)\right )
    w\left(
    j_n(gB_n\cdot D_{n}w_{n})
    \right)
    v(gB_n)
    \chi(g)
    \absnorm{g}^{s-\frac{1}{2}}dg\\
    =\;&
    \chi(t_{f}^{-1})\absnorm{t_{f}}^{\frac{1}{2}-s}
    \int
    \psi\left(\lambda_n^0(g)\right )
    w\left(
    j_n(g\cdot D_{n}w_{n})
    \right)
    v(g)
    \chi(g)
    \absnorm{g}^{s-\frac{1}{2}}dg.
  \end{align*}
  By Lemma \ref{lem:zentrales}, we may evaluate the latter integral explicitly as
  \begin{align*}
    &
    \int
    \psi\left(\lambda_n^0(g)\right )
    w\left(
    j_n(g\cdot D_{n}w_{n})
    \right)
    v(g)
    \chi(g)
    \absnorm{g}^{s-\frac{1}{2}}dg\\
    =\;&
    \vol(U_n(\OO)\varpi^{e} J^n_{\ell})\cdot\!\!\!
    \sum_{g\in\varpi^{e}\omega \mathfrak{R}_{n,\ell}^{\omega}}
      \!\!\!
      \psi(\lambda_n^{\delta}(g))\cdot
      w(
      g\cdot D_n w_n
      )\cdot
      v(g)\cdot
      \chi(g)\cdot
      \absnorm{g}^{s-\frac{1}{2}}\\
      =\;&
      \absNorm(\mathfrak{f}_{\chi\theta\theta'})^{-\frac{(n+1)n(n-1)}{6}}\cdot
      \prod_{\mu=1}^n\theta^{w_\mu}(B_\mu)\cdot\prod_{\nu=1}^\mu\absNorm(\mathfrak{f}_{\chi\theta^{w_\mu}_\nu\theta_\nu'})^{-1}\cdot
      \chi\theta^{w_\mu}_\nu\theta_\nu'(f_{\chi\theta^{w_\mu}_\nu\theta_\nu'})
      \cdot G(\chi\theta^{w_\mu}_\nu\theta_\nu')\cdot
      \\&
      w(t_{f f_{\chi\theta\theta'}^{-1}}) v(t_{f f_{\chi\theta\theta'}^{-1}})
      \chi(t_{f f_{\chi\theta\theta'}^{-1}})
      \absnorm{t_{f f_{\chi\theta\theta'}^{-1}}}^{s-\frac{1}{2}}.
  \end{align*}
  Finally, remark that
  $$
    \theta(t_{(-1)})\;=\;\prod_{\mu=1}^n\theta^{w_\mu}(B_\mu),
  $$
  and the claim follows.
\end{proof}

\subsection{The generically nearly ordinary $p$-adic Hecke algebra}\label{sec:generichecke}

Fix a number field $F/\QQ$ with ring of integers $\OO_F$ and a rational prime $p$. Put\nomenclature[]{$F_p$}{defined as $F\otimes_\QQ\QQ_p$}\nomenclature[]{$\OO_p$}{defined as $\OO_F\otimes_\ZZ\ZZ_p$}
$$
F_p\;:=\;F\otimes_\QQ\QQ_p,\quad\text{and}\quad \OO_p:=\OO_F\otimes_\ZZ\ZZ_p.
$$
Let\nomenclature[]{$I_\alpha$}{elements in $G(\ZZ_p)$ which are upper triangular mod $p^\alpha$}
$$
I_{\alpha}\;:=\;\{r\in G(\ZZ_p)\mid r\!\!\!\pmod{p^{\alpha}}\in B(\ZZ_p/p^{\alpha}\ZZ_p)\},
$$
$$
I_{\alpha',\alpha}\;:=\;\{r\in I_\alpha\mid r\!\!\!\pmod{p^{\alpha'}}\in U(\ZZ_p/p^{\alpha'}\ZZ_p)\},
$$
for $\alpha\geq\alpha'\geq 0$\nomenclature[]{$I_{\alpha',\alpha}$}{elements in $I_\alpha$ which are unipotent mod $p^{\alpha'}$}.

Define the semigroup\nomenclature[]{$\Delta_G$}{semigroup in $T(\QQ_p)$}
$$
\Delta_{G}\;:=\;\prod_{v\mid p}\Delta_{F_v,n+1}\times\Delta_{F_v,n}\;\subseteq\;T(\QQ_p),
$$
and the corresponding Hecke algebra\nomenclature[]{$\mathcal H_A(\alpha',\alpha)$}{double coset Hecke algebra of level $I_{\alpha',\alpha}$ for $G$}
$$
\mathcal H_A(\alpha',\alpha)\;:=\;\mathcal H_A(I_{\alpha',\alpha},I_{\alpha',\alpha}\Delta_{G}I_{\alpha',\alpha}).
$$
Then $\mathcal H_A(\alpha;,\alpha)$ is the product of the products of the local Hecke algebras
$$
\mathcal H_A^{n+1}(v_{\mathfrak{p}}(p)\alpha',v_{\mathfrak{p}}(p)\alpha)\quad\text{and}\quad\mathcal H_A^{n}(v_{\mathfrak{p}}(p)\alpha',v_{\mathfrak{p}}(p)\alpha)
$$
introduced in section \ref{sec:padicheckealgebras}. Therefore, \eqref{eq:iwahoriheckeasgroupring} translates to
\begin{equation}
\mathcal H_A(\alpha',\alpha)\;=\;\mathcal H_A(0,\alpha)[I_\alpha/I_{\alpha',\alpha}]
\;=\;\mathcal H_A(0,\alpha)[T_n(\ZZ_p/p^{\alpha'}\ZZ_p)],
\label{eq:piwahoriheckeasgroupring}
\end{equation}
This is a finitely generated commutative $A$-algebra. By \eqref{eq:piwahoriheckeasgroupring}, there is a canonical map\nomenclature[]{$\langle-\rangle$}{canonical map $T_n(\ZZ_p/p^{\alpha'}\ZZ_p)\;\to\;\mathcal H_A(\alpha',\alpha)$}
$$
\langle\cdot\rangle:\quad T_n(\ZZ_p/p^{\alpha'}\ZZ_p)\;\to\;\mathcal H_A(\alpha',\alpha),
$$
$$
t\;\mapsto\;\langle t\rangle.
$$
Furthermore, we have a canonical isomorphism
$$
U:\quad A[\Delta_G]\;\cong\;\mathcal H_A(0,\alpha)
$$
given by\nomenclature[]{$U(\delta)$}{$I_\alpha$-double coset corresponding to $\delta\in\Delta_G$}
$$
\Delta_G\;\ni\;\delta\;\mapsto\;I_\alpha\delta I_\alpha=:U(\delta).
$$
We define the {\em generic nearly ordinary $p$-adic Hecke algebra} as the localization\nomenclature[]{$\mathcal H_A^{\rm ord}$}{generic nearly ordinary $p$-adic Hecke algebra for $G$}
$$
\mathcal H_A^{\rm ord}(\alpha',\alpha)\;:=\;
\mathcal H_A(\alpha',\alpha)[U(\delta)^{-1}|\delta\in\Delta_G].
$$
If $\langle\Delta_G\rangle$ denotes the group generated by $\Delta_G,$ we obtain a canonical isomorphism
$$
A[\langle\Delta_G\rangle]\;\cong\;\mathcal H_A^{\rm ord}(0,\alpha)
$$
extending $U$ and \eqref{eq:piwahoriheckeasgroupring} extends to
\begin{equation}
\mathcal H_A^{\rm ord}(\alpha',\alpha)
\;=\;\mathcal H_A(0,\alpha)^{\rm ord}[T_n(\ZZ_p/p^{\alpha'}\ZZ_p)].
\label{eq:pordinaryheckeasgroupring}
\end{equation}

If $\mathcal M$ is an $\mathcal H_A(\alpha',\alpha)$-module\nomenclature[]{$\mathcal M$}{Hecke module}, then the Hecke module structure on $\mathcal M$ extends to $\mathcal H_A^{\rm ord}(\alpha',\alpha)$ if and only if the operators $U(\delta)$, $\delta\in\Delta_G$, act invertibly on $\mathcal M$. Since $\Delta_G$ is finitely generated, it suffices to check this for an arbitrary finite set of generators of $\Delta_G,$ or a product thereof.

For any character $\vartheta:T(\QQ_p)\to A^\times$\nomenclature[]{$\vartheta$}{(quasi-)character $T(\QQ_p)\to A^\times$, also its canonical extension to $\mathcal H_A(\alpha,\alpha)\to A$, often a $p$-Nebentypus}
 whose restriction to $T(\ZZ_p)$ factors over $T(\ZZ_p/p^\alpha\ZZ_p)$, we consider $\vartheta|_{T(\ZZ_p)}$ as an algebra homomorphism $A[T(\ZZ_p/p^\alpha\ZZ_p)]\to A$. By \eqref{eq:piwahoriheckeasgroupring}, $\vartheta$ then extends to an $A$-algebra homorphism
  $$
    \vartheta:\quad \mathcal H_A(\alpha,\alpha)\;\to\;A
  $$
  by setting
  $$
    U(\delta)\;\mapsto\;\vartheta(\delta),\quad\delta\in\Delta_G.
  $$
  The values $\vartheta(\delta)$ being invertible, this extends uniquely to an $A$-algebra homomorphism
  $$
    \vartheta:\quad \mathcal H_A^{\rm ord}(\alpha,\alpha)\;\to\;A.
  $$

In order to streamline notation in the sequel, set for $\underline{\varepsilon}=(\varepsilon_\varpi)_{\varpi\mid p}$\nomenclature[]{$\underline{\varepsilon}$}{multi-exponent indexed by $\varpi\mid p$}
, $e_\varpi\in\ZZ$,
\begin{equation}
  p^{\underline{\varepsilon}}\;:=\;\prod_{\varpi\mid p}\varpi^{\varepsilon_\varpi},
  \label{eq:definitionofptothee}
\end{equation}
and define for $\alpha\geq\alpha'\geq0,$ $\alpha>0,$ the Hecke operator\nomenclature[]{$U_{p^{\underline{\varepsilon}}}$}{Hecke operator of level $I_{\alpha',\alpha}$ for $G(\QQ_p)$}
$$
U_{p^{\underline{\varepsilon}}}\;=\;I_{\alpha',\alpha}\Delta(t_{p^{\underline{\varepsilon}}}) I_{\alpha',\alpha}\;=\;
\bigsqcup_{u\in U(\ZZ_p)/\Delta(t_{p^{\underline{\varepsilon}}}) U(\ZZ_p)\Delta(t_{p^{\underline{\varepsilon}}}^{-1})}
u\Delta(t_{p^{\underline{\varepsilon}}}) I_{\alpha',\alpha}\;\in\;\mathcal H_A(\alpha',\alpha).
$$
Put\nomenclature[]{$U_p$}{the $U_p$-operator of level $I_{\alpha',\alpha}$ for $G(\QQ_p)$}
\begin{equation}
U_{p}\;:=\;I_{\alpha',\alpha} t_p I_{\alpha',\alpha}\;=\;\prod_{\mathfrak{p}\mid p}\left(U_{\mathfrak{p}}\otimes U_{\mathfrak{p}}'\right)^{v_{\mathfrak{p}}(p)}\;\in\;\mathcal H_A(\alpha',\alpha).
\label{eq:definitionofUp}
\end{equation}

\subsection{The twisted global Zeta integral}\label{sec:birchglobal}

Recall from the previous section that $F/\QQ$ is a number field with ring of integers $\OO_F$ and $p$ denotes a rational prime. Fix a non-trivial additive character $\psi:F\backslash\Adeles_F\to\CC^\times$ with local factors $\psi_v:F_v^\times\to\CC^\times$. Put\nomenclature[]{$\psi_p$}{additive character $F_p\to\CC^\times$ derived from $\psi$}
$$
\psi_p:=\underset{\mathfrak{p}\mid p}{\otimes}\psi_{\mathfrak{p}}:\quad F_p\to\CC^\times
$$
and assume without loss of generality that $\psi_p$ is of conductor $\OO_p$.

Let $\Pi$ and $\Sigma$ be irreducible cuspidal automorphic representations of $\GL_{n+1}(\Adeles_F)$ and $\GL_{n}(\Adeles_F)$ respectively\nomenclature[]{$\Pi$}{cuspidal automorphic representation of $\GL_{n+1}(\Adeles_F)$}\nomenclature[]{$\Sigma$}{cuspidal automorphic representation of $\GL_{n}(\Adeles)$}, which we also consider as an automorphic representation $\Pi\widehat{\otimes}\Sigma$ of $G(\Adeles)$\nomenclature[]{$\Pi\widehat{\otimes}\Sigma$}{cuspidal automorphic representation of $G(\Adeles)$ corresponding to $\Pi$ and $\Sigma$}. Let $S_\infty$ denote the set of infinite places of $F$\nomenclature[]{$S_\infty$}{archimedean places of number field $F$} and $S_{\Pi\widehat{\otimes}\Sigma}$ the set of finite places where $\Pi$ or $\Sigma$ ramifies\nomenclature[]{$S_{\Pi\widehat{\otimes}\Sigma}$}{finite places where $\Pi\widehat{\otimes}\Sigma$ ramifies}.

We make free use of the theory of Rankin-Selberg $L$-function $L(s,\Pi\widehat{\otimes}\Sigma)$ as developed in \citep{jpss1979a,jpss1979b,jpss1983,jacquetshalika1990s,cogdellpiatetskishapiro1994,cogdellpiatetskishapiro2004}\nomenclature[]{$L(s,\Pi\widehat{\otimes}\Sigma)$}{Rankin-Selberg $L$-function for the pair $(\Pi,\Sigma)$}.

At any archimedean places $v\in S_\infty$, we consider the smooth models $\Pi_v$ and $\Sigma_v$, i.e.\ these are smooth Fr\'echet representations of $\GL_{n+1}(F_v)$ and $\GL_n(F_v)$ of moderate growth and finite length. They agree with the Casselman-Wallach completions of the subspaces of $K$-finite vectors. We refer the reader to \citep{casselman1989,bernsteinkrotz2014} for the notion of Casselman-Wallach completion. We also write $\Pi_v\widehat{\otimes}\Sigma_v$ for the completed projective tensor product\nomenclature[]{$\widehat{\otimes}$}{(outer) completed projective tensor product of Casselman-Wallach representations}. This is a Casselman-Wallach representation of $G(F_v)$.

Recall that $\Pi$ and $\Sigma$ are always generic \citep{shalika1974}. At any place $v$ of $F$, consider for any vector\nomenclature[]{$W_v$}{local Whittaker vector on $\GL_{n+1}(F_v)\times\GL_n(F_v)$}
$$
W_v\;\in\;
{\mathscr{W}}(\Pi_v\otimes\Sigma_v,\psi_v\otimes\psi_v^{-1})=
{\mathscr{W}}(\Pi_v,\psi_v)\otimes{\mathscr{W}}(\Sigma_v,\psi_v^{-1})
$$
in the local Whittaker model of $\Pi_v\otimes\Sigma_v$ (or of $\Pi_v\widehat{\otimes}\Sigma_v$ for $v\mid\infty$) the local zeta integral\nomenclature[]{$\Psi_v(s,W_v)$}{local Rankin-Selberg zeta integral}
$$
\Psi_v(s,W_v)\;:=\;\int_{U_n(F_v)\backslash\GL_n(F_v)} W_v(\Delta(g_v))\absnorm{g_v}^{s-\frac{1}{2}}dg_v
$$
as in \citep{jpss1983}.

At archimedean places $v$, we know that for any $W_v$ the local zeta integral $\Psi_v(s,W_v)$ satisfies an identity\nomenclature[]{$\Omega_v(s,W_v)$}{ratio of local zeta integral to local $L$-function}
$$
  \Psi_v(s,W_v)\;=\;\Omega_v(s,W_v)\cdot L(s,\Pi_v\widehat{\otimes}\Sigma_v)
$$
for a function $\Omega_v(s,W_v)$ holomorphic in $s$. For $K$-finite $W_v$, it is known that $\Omega_v(s,W_v)$ is a {\em polynomial}. Moreover, there is a good $K$-finite test vector $W_v^0$ trivializing $\Omega_v(s,W_v)$ \citep{cogdellpiatetskishapiro2004,jacquet2009}.

At all finite places $v$, we fix a {\em good tensor} $W_v^0\in{\mathscr{W}}(\Pi_v,\psi_v)\otimes{\mathscr{W}}(\Sigma_v,\psi_v^{-1})$\nomenclature[]{$W_v^0$}{good local test vector in the Whittaker model (the new vector in the spherical case)} with the property that
$$
  L(s,\Pi_v\otimes\Sigma_v)=\Psi(s,W_v^0).
$$
We suppose that $W_v^0=W_{n+1,v}^0\otimes W_{n,v}^0$ for spherical Whittaker functions $W_{n+1,v}^0$ and $W_{n_v}^0$ at all places $v\not\in S_{\Pi\widehat{\otimes}\Sigma}\cup S_\infty$. By Shintani's explicit formula \citep{shintani1976}, for $v\not\in S_{\Pi\widehat{\otimes}\Sigma}\cup S_\infty$, the local $L$-function is given explicitly by
$$
  L(s,\Pi_v\otimes\Sigma_v)=
  \det\left({\bf 1}_{(n+1)n}-q_v^{-s}(A_{\Pi_v}\otimes A_{\Sigma_v})\right)^{-1},
$$
where $A_{\pi_\mathfrak{q}}$ and $A_{\sigma_\mathfrak{q}}$ denote the corresponding Frobenius-Hecke parameters (Satake parameters)\nomenclature[]{$A_{\pi_\mathfrak{q}}$}{Satake parameters for $\pi_{\mathfrak{q}}$} and $q_v=\absNorm(\mathfrak{p}_v)$.

Define the good test vector\nomenclature[]{$W^0$}{global good test vector in the Whittaker model}\nomenclature[]{$\mathscr W(\Pi\widehat{\otimes}\Sigma),\psi\otimes\psi^{-1})$}{global $\psi$-Whittaker model of $\Pi{\otimes}\Sigma$}
$$
  W^0\;:=\;\underset{v}{\otimes} W_v^0\;\in\;
  {\mathscr{W}}(\Pi\widehat{\otimes}\Sigma,\psi\otimes\psi^{-1}).
$$
Then for any other test vector
$$
  W\;\in\;{\mathscr{W}}(\Pi\widehat{\otimes}\Sigma,\psi\otimes\psi^{-1}),
$$
which agrees with $W^0$ at all places outside a finite set $S_W$, the global zeta integral\nomenclature[]{$\Psi(s,W)$}{global zeta integral for $G$}
$$
  \Psi(s,W)\;:=\;\int_{U_n(\Adeles_F)\backslash\GL_n(\Adeles_F)} W(\Delta(g))\absnorm{g}^{s-\frac{1}{2}}dg
$$
is well defined: $\Psi(s,W)$ converges absolutely for ${\rm Re}(s)\gg 0$ and has an Euler product decomposition
$$
  \Psi(s,W)\;=\;\prod_v\Psi_v(s,W_v)\;=\;\prod_{v\in S_W}\Omega_v(s,W_v)\cdot\prod_v L(s,\Pi_v\widehat{\otimes}\Sigma_v).
$$
Therefore,
\begin{equation}
  \Psi(s,W)\;=\;\Omega(s,W)\cdot L(s,\Pi_v\widehat{\otimes}\Sigma_v)
  \label{eq:psiomegaL}
\end{equation}
for a function\nomenclature[]{$\Omega(s,W)$}{ratio of global zeta integral to global $L$-function}
$$
\Omega(s,W)\;:=\;\prod_{v}\Omega_v(s,W_v)\;=\;\prod_{v\in S_W}\Omega_v(s,W_v)
$$
holomorphic in $s$. Moreover, $\Omega(s,W)$ is a polynomial in $s$ and $q_v^{-s}$, $v\in S_W$ finite, whenever $W$ is $K$-finite.

By Fourier transform, we may associate to $W$ an automorphic form $\phi_W$ on $G(\Adeles)$\nomenclature[]{$\phi_W$}{cusp form: the inverse Fourier transform of $W$}. Then we have an identity
\begin{equation}
\Psi(s,W)\;=\;
\Phi(s,\phi_W)\;:=\;
\int_{H(\QQ)\backslash H(\Adeles)}
\phi_W(g)
\absnorm{\det(g)}^{s-\frac{1}{2}}dg
\label{eq:whittakertoautomorphicform}
\end{equation}
for ${\rm Re}(s)\gg 0$, where the right hand side converges absolutely for all $s\in\CC$ and defines an analytic function in $s$. Then \eqref{eq:psiomegaL} extends by holomorphic continuation to the identity
\begin{equation}
  \Phi(s,\phi_W)\;=\;
  \Omega(s,W)\cdot L(s,\Pi_v\widehat{\otimes}\Sigma_v),
  \label{eq:phiomegaL}
\end{equation}
which is valid for {\em all} $s\in\CC$.

For any quasi-character $\chi:F^\times\backslash\Adeles_F^\times\to\CC^\times$ we identify
$$
L(s,\Pi\widehat{\otimes}\Sigma\otimes\chi)\;:=\;
L(s,\Pi\widehat{\otimes}(\Sigma\otimes\chi))\;=\;
L(s,(\Pi\otimes\chi)\widehat{\otimes}\Sigma).
$$
Then the twisted Whittaker function\nomenclature[]{$W_\chi$}{global Whittaker function twisted by $\chi$}
$$
W_\chi:\quad G(\Adeles)\to\CC,\quad (g_1,g_2)\mapsto \chi(\det(g_2))W(g_1,g_2),
$$
is an element of ${\mathscr{W}}(\Pi\widehat{\otimes}(\Sigma\otimes\chi),\psi\otimes\psi^{-1})$. At all places outside $S_W\cup (S_{\Pi\widehat{\otimes}\Sigma}\cap S_{\chi})$ the function $W_\chi$ coincides with a good test vector for the twisted $L$-function.

For a finite set of places $S$ of $F$ we write $L^S(s,\Pi\widehat{\otimes}\Sigma)$\nomenclature[]{$L^S(s,\Pi\widehat{\otimes}\Sigma)$}{global partial Rankin-Selberg $L$-function with Euler factors at places in $S$ removed} for the partial $L$-functions where the Euler factors at places in $S$ have been removed. Likewise, $\Omega^S(s,W)$ denotes the correction factor with the factors for places in $S$ removed\nomenclature[]{$\Omega^S(s,W)$}{ratio of partial zeta integral to partial $L$-function (outside $S$)}.

For a quasi-character $\vartheta:T(\QQ_p)\to\CC^\times,$ write $\vartheta_{\mu,\nu}:F_p^\times\to\CC^\times$ for the respective $(\mu,\nu)$-component according to the decomposition $T=T_{n+1}\times T_n$, where $1\leq\mu\leq n+1$, $1\leq\nu\leq n$\nomenclature[]{$\vartheta_{\mu,\nu}$}{the $(\mu,\nu)$-component of $\vartheta$}.

If $\chi_p:F_p^\times\to\CC^\times$ is another quasi-character, consider $\chi_p$ as a quasi-character of $T(\QQ_p)$ via pullback along the determinant in the second factor. We say that the resulting character $\chi_p\vartheta$ has {\em constant conductor}, if:
  \begin{itemize}
  \item[(C)] For all $1\leq\nu\leq\mu\leq n$ the conductors of
    $\chi_p\vartheta_{\mu,\nu}$ all agree.
  \end{itemize}
  We say furthermore that $\chi_p\vartheta$ has {\em fully supported conductor}, if
  \begin{itemize}
    \item[(F)] For all $1\leq\nu\leq \mu\leq n$, the conductors of the local characters $\chi_p\vartheta_{\mu,\nu}$ are supported at all $\mathfrak{p}\mid p$.
  \end{itemize}

  For a character satisfying (C) and (F), let\nomenclature[]{$\underline{\varepsilon}_{\chi\vartheta}$}{multi-exponent in the conductor of (each component of) $\chi_p\vartheta$}\nomenclature[]{$f_{\chi\vartheta}$}{generator of the conductor of (each component of) $\chi\vartheta$}
  \begin{equation}
    f_{\chi\vartheta}=p^{\underline{\varepsilon}_{\chi\vartheta}}\;\in\;\OO_{p}
    \label{eq:charactergenerator}
  \end{equation}
  denote a generator of the conductor $\mathfrak{f}_{\chi\vartheta}$ of any $\chi_p\vartheta_{\mu,\nu}$ for $\nu\leq\mu$.

  Recall that for $\alpha > 0$ with the property that $f=p^\alpha$ is divisible by $f_{\vartheta}$ the quasi-character $\vartheta$ extends to an algebra homomorphism
  $$
  \vartheta:\quad\mathcal H_\CC(\alpha,\alpha)\to\CC.
  $$
  Recall the definition of the matrix $h$ in \eqref{eq:definitionofh}. We consider $h$ as a diagonally (with respect to the places above $p$) embedded element in $G(\ZZ_p)$ and identify for any $f\in F_p$ the element $t_f\in\GL_n(F_p)$ with its image under $\Delta\circ(j_n\times{\bf1})$ in $G(\QQ_p)$. Let $S(p)$ denote the set of places of $F$ dividing $p$\nomenclature[]{$S(p)$}{places in $F$ above $p$}.

\begin{theorem}[Global Birch Lemma]\label{thm:globalbirch}
  Let $W\in{\mathscr{W}}(\Pi\widehat{\otimes}\Sigma,\psi\otimes\psi^{-1})$ be a Whittaker function with the following properties:
  \begin{itemize}
    \item[(i)] $W$ admits a factorization $W=W_p\otimes W^{S(p)}$ for a local Whittaker function $W_p$ at $p$ and a Whittaker function $W^{S(p)}$ outside $p$.
    \item[(ii)] $W_p$ is right invariant under $I_{\alpha,\alpha}$ for some $\alpha\geq 1$.
    \item[(iii)] $W_p$ is an eigenvector for $\mathcal H_\CC(\alpha,\alpha)$ for a quasi-character $\vartheta:T(\QQ_p)\to\CC^\times$.
  \end{itemize}
  Let $\chi:F^\times\backslash\Adeles_F^\times\to\CC^\times$ be a finite order character such that $\chi_p\vartheta$ has fully supported constant conductor $\mathfrak{f}_{\chi\theta}$ dividing $f=p^{\alpha}$ for an $\alpha>0$ which is sufficiently large satisfying (ii).

  Then for every $s\in\CC$,
  \begin{align*}
    &
      \int\limits_{H(\QQ)\backslash H(\Adeles)}
       \phi_W
       \left(
         g\cdot ht_{f}
       \right)
       \chi(\det(g))
       \absnorm{\det(g)}^{s-\frac{1}{2}}
      dg\\
    =\;&
      \Omega^{S(p)}(s,W_\chi)\cdot
      \delta(W_p)\cdot
      \absNorm(\mathfrak{f})^{-\frac{(n+2)(n+1)n +(n+1)n(n-1)}{6}}\cdot
      \absNorm(\mathfrak{f}_{\chi\vartheta})^{-\frac{(n+1)n(n-1)}{6}}\cdot
      \absnorm{t_{f_{\chi\theta}}}^{\frac{1}{2}-s}\cdot
      \\
    &
      \vartheta(t_{p}^{\alpha})\cdot
      \prod_{\mu=1}^n\prod_{\nu=1}^\mu
        G(\chi\vartheta_{\mu,\nu})\cdot
      L(s,\Pi\widehat{\otimes}\Sigma\otimes\chi),
  \end{align*}
  where\nomenclature[]{$\delta(W_p)$}{complex constant}
  $$
    \delta(W_p):=
    W_p\left({\bf1}_n\right)
    \cdot
    \prod_{\mu=1}^{n}
    \prod_{v\mid p}
    \left(1-q_v^{-\mu}\right)^{-1}.
  $$
\end{theorem}

\begin{remark}
By the discussion preceeding Theorem \ref{thm:globalbirch},
$$
\Omega^{S(p)}(s,W_\chi)\;=\;
\Omega_\infty(s,W_{\chi,\infty})\cdot
\Omega_{S_W\cup S_\chi}^{S_\infty\cup S(p)}(s,W).
$$
In particular, if $S_W\cup S_\chi\subseteq S_\infty\cup S(p)$, then
\begin{equation}
  \Omega^{S(p)}(s,W_\chi)\;=\;
\Omega_\infty(s,W_{\chi,\infty}).
\label{eq:omegacomparison}
\end{equation}
\end{remark}

\begin{proof}[Proof of Theorem \ref{thm:globalbirch}]
  For ${\rm Re}(s)\gg 0$ we have by \eqref{eq:whittakertoautomorphicform} and identity
  \begin{align*}
    &
    \int\limits_{H(\QQ)\backslash H(\Adeles)}
    \phi_W
    \left(
    g\cdot ht_{f}
    \right)
    \chi(g)
    \absnorm{g}^{s-\frac{1}{2}}
    dg\\
    =\;&
    \chi(t_{f})\absnorm{t_{f}}^{s-\frac{1}{2}}\cdot
    \Phi\left(s,\phi_{W_\chi}(-\cdot ht_{f})\right)\\
    =\;&
    \chi(t_{f})\absnorm{t_{f}}^{s-\frac{1}{2}}\cdot
    \Psi\left(s,W_\chi(-\cdot ht_{f})\right)\\
    =\;&
    \chi(t_{f})\absnorm{t_{f}}^{s-\frac{1}{2}}\cdot
    \Psi_p\left(s,W_{\chi,p}(-\cdot ht_{f})\right)\cdot
    \Omega^{S(p)}(s,W_\chi)\cdot
    L(s,\Pi\widehat{\otimes}\Sigma\otimes\chi).
  \end{align*}  
  By holomorphic continuation, this identity extends to all $s\in\CC$ and the value of the expression
  \begin{align*}
    &\chi(t_{f})\absnorm{t_{f}}^{s-\frac{1}{2}}\cdot
    \Psi_p\left(s,W_{\chi,p}(-\cdot ht_{f})\right)\\
    =\;&\prod_{v\mid p}
    \int_{U_n(F_v)\backslash\GL_n(F_v)} W_v\left(\Delta(g_v)\cdot h_vt_{f_v}\right)\cdot
    \chi_v(g_v)\absnorm{g_v}^{s-\frac{1}{2}}dg
  \end{align*}  
  is given by the Local Birch Lemma (Theorem \ref{thm:localbirch}). Therefore, the global zeta integral evaluates to
  \begin{align*}
    &
    \Omega^{S(p)}(s,W_\chi)\cdot
    \prod_{\mu=1}^{n}
    \prod_{v\mid p}
    \left(1-q_v^{-\mu}\right)^{-1}\cdot
    \absNorm(\mathfrak{f}_{\chi\vartheta})^{-\frac{(n+2)(n+1)n}{6}}\cdot
    \absnorm{t_{f_{\chi\theta}}}^{\frac{1}{2}-s}\cdot\\
    &
    W_p\left(t_{f f_{\chi\vartheta}^{-1}}\right)\cdot
    \prod_{\mu=1}^n\prod_{\nu=1}^\mu
      \left[\vartheta_{\mu+1-\nu,\nu}(f_{\chi\vartheta})\cdot
      G(\chi\vartheta_{\mu+1-\nu,\nu})\right]\cdot
    L(s,\Pi\widehat{\otimes}(\Sigma\otimes\chi)).
  \end{align*}
  Recall that $\psi_p(\OO_p)=1$ and $f_{\chi\theta}=p^{\underline{\varepsilon}_{\chi\theta}},$ whence
  \begin{align*}
    W_p\left(t_{f f_{\chi\vartheta}^{-1}}\right)
    &\;=\;
    W_p\left(t_{fp^{-\underline{\varepsilon}_{\chi\vartheta}}}\right)\\
    &\;=\;
    [U(\ZZ_p):t_{fp^{-\underline{\varepsilon}_{\chi\vartheta}}}U(\ZZ_p)
      t_{fp^{-\underline{\varepsilon}_{\chi\vartheta}}}^{-1}]^{-1}\cdot
    U_{fp^{-\underline{\varepsilon}_{\chi\vartheta}}}
     W_p
    ({\bf1})\\
     &\;=\;
     \absNorm(\mathfrak f_{\chi\vartheta}/\mathfrak f)^{\frac{(n+2)(n+1)n+(n+1)n(n-1)}{6}}\cdot
    \vartheta(t_{f_{\chi\vartheta}}^{-1})\cdot\vartheta(t_f)    \cdot
    W_p({\bf1}).
  \end{align*}
  The observation
  $$
    \vartheta(t_{f_{\chi\vartheta}})\;=\;
    \prod_{\mu=1}^n\prod_{\nu=1}^\mu
    \vartheta_{\mu+1-\nu,\nu}(f_{\chi\vartheta})
  $$
  concludes the proof.    
\end{proof}

\section{$p$-adic Lattices}\label{sec:lattices}

In this section we establish a general relation between lattices which will translate into congruences essential to establish the independence of the weight of the non-abelian $p$-adic $L$-function we construct.

\subsection{Integral algebras}

As before, we consider the element $t_p\in\GL_n(F_p)$ as a diagonally embedded element $t_p\in G(\QQ_p)$\nomenclature[]{$t_p$}{the element $\diag(p^n,\dots,p)$ diagonally embedded into $G(\QQ_p)$}. For any non-negative integer $\alpha\geq 0$ consider the element\nomenclature[]{$g_\alpha$}{the element $h\cdot t_p^\alpha$ in $G(\QQ_p)$}
$$
g_{\alpha}\;:=\;h\cdot t_p^\alpha\;\in\;G(\QQ_p),
$$
and set for any subgroup $L\subseteq G$,\nomenclature[]{$L^\alpha$}{the $g_\alpha$-conjugate of subgroup $L\subseteq G$}
$$
L^{\alpha}\;:=\;g_{\alpha}Lg_{\alpha}^{-1}.
$$
and likewise for sub Lie algebras of $\lieg$ with respect to the adjoint action on $\lieg$.\nomenclature[]{$\liel^\alpha$}{the $g_\alpha$-conjugate of a Lie subalgebra $\liel\subseteq\lieg$}

\begin{lemma}\label{lem:transversality}
For any $\alpha\geq 0$, the subgroups $H$ and $B^{-,\alpha}$ are transversal, i.e.\
\begin{equation}
\lieg_{E}\;=\;\lieh_{E}\oplus\lieb^{-,\alpha}_{E}.
\label{eq:transversality}
\end{equation}
Furthermore, we have $\lieb^{-,\alpha}_{E}=\lieb^{-,0}_{E}$, and likewise for $B$ replacing $B^-$.
\end{lemma}

\begin{proof}
Since $t_p$ normalizes $B^-$ over fields, we are reduced to the case $\alpha = 0$ and $E=\QQ_p$. Counting dimensions shows that \eqref{eq:transversality} is equivalent to
$$
\lieh_{\QQ_p}h\cap h \lieb^{-}_{\QQ_p}\;=\;0,
$$
which in turn is an easy excercise in linear algebra.
\end{proof}

By the Poincar\'e-Birkhoff-Witt Theorem we obtain
\begin{corollary}\label{cor:transversality}
For any $\alpha\geq 0$ we have a canonical isomorphism
\begin{equation}
U(\lieg_{E})\;=\;U(\lieh_{E})\otimes_{E}U(\lieb^{-,\alpha}_{E}).
\label{eq:envelopingtransversality}
\end{equation}
\end{corollary}

\begin{lemma}\label{lem:integralityrelation}
For any $\alpha\geq 0$ we have
\begin{equation}
U(\lieu_{\OO}^{\alpha})\;\subseteq 
\;\left(\OO+p^\alpha U(\lieh_{\OO})\right)\otimes_{\OO}\left(\OO+p^\alpha U(\lieb^{-,0}_{\OO})\right)
\label{eq:integralenvelopingtransversality}
\end{equation}
as subspaces of \eqref{eq:envelopingtransversality}.
\end{lemma}

\begin{proof}
In the case $\alpha=0$, it suffices to remark that $h\in G(\ZZ_p)$ and therefore $g_{0}\in G(\OO)$, whence
\begin{equation}
\lieg_{\OO}\;=\;\lieh_{\OO}\oplus\lieb_{\OO}^{-,0},
\label{eq:integraltransversality}
\end{equation}
by \eqref{eq:transversality}. For $\alpha>0,$ observe
\begin{equation}
t_p^\alpha\lieu_{\OO} t_p^{-\alpha}\;\subseteq\;p^\alpha
\lieu_{\OO},
\label{eq:tconjugationofu}
\end{equation}
which implies
$$
\lieu_{\OO}^{\alpha}\;\subseteq\;p^\alpha
\lieu_{\OO}^{0}.
$$
Therefore,
$$
U(\lieu_{\OO}^{\alpha})\;\subseteq\;\OO+p^\alpha U(\lieu_{\OO}^{0})
$$
and the claim follows with \eqref{eq:integraltransversality}.
\end{proof}

\subsection{$p$-integral structures on rational representations}

Consider a finite extension $E/\QQ_p$ and let $L_{\lambda,E}$ denote a rational representation of $G$ of $B$-highest weight $\lambda$ defined over $E$. Write $w_0=(w_{n+1},w_n)\in W(G,T)$ for the longest element in the Weyl group of $G$ with respect to the torus $T$ corresponding to the diagonal matrices in $G$ and the positive system $\Delta^+$ given by our choice of $\lieb$\nomenclature[]{$\Delta^+$}{system of positive roots in $\lieb$}. The (algebraic) differential of $\lambda$ is a canonical $E$-valued character of $\lieb$ (trivial on the radical), which in turn gives rise to a character $\lambda^{w_0}$ of $\lieb^-_{E}=w_0\lieb_Ew_0$. It extends uniquely to a character of $U(\lieb^-_{E})$ that we denote the same. We pull it back to a character $(\lambda^{w_0})^{\alpha}$ of $\lieb_{E}^{-,\alpha}$. Fix a highest weight vector $v_0\in L_{\lambda,E}$ once and for all\nomenclature[]{$v_0$}{a fixed highest weight vector in $L_{\lambda,E}$}. Then $g_{\alpha}\cdot v_0$ is a $B_E^{-,\alpha}$-highest weight vector of weight $\lambda^{\alpha}$.

Then any $t\in T(\QQ_p)$ acts on $v_0$ via the scalar $\lambda^{w_0}(t)\in E^\times$. Renormalize its action on $V_{\lambda,E}$ by defining\nomenclature[]{$t\bullet(-)$}{$p$-normalized action of $T(\QQ_p)$}
\begin{equation}
  t\bullet v\;:=\;(-\lambda^{w_0})(t)\cdot (tv),\quad v\in V_{\lambda,E}.
  \label{eq:taction}
\end{equation}
Inside $L_{\lambda,E}$ consider the $G(\OO)$-lattice $L_{\lambda,\OO}$\nomenclature[]{$L_{\lambda,\OO}$}{$G(\OO)$-lattice generated by $v_0$} generated by $v_0\in L_{\lambda,E}$. Then
\begin{equation}
  L_{\lambda,\OO}\;=\;U(\lieu_{\OO})\cdot v_0.
  \label{eq:integrallattice}
\end{equation}
In particular, by \eqref{eq:tconjugationofu} and $t_p\bullet v_0=v_0$ the lattice $L_{\lambda,\OO}$ is stable under the renormalized action of $t_p$. Recall the definition of $d_x$ in \eqref{eq:definitionofdx} and define the lattices\nomenclature[]{$L_{\lambda,\OO}^{x,\alpha}$}{translation of the lattice $L_{\lambda,\OO}$ by $d_xg_\alpha$ using normalized $T(\QQ_p)$-action}
\begin{align*}
  L_{\lambda,\OO}^{x,\alpha}
  &\;:=\;
  d_xh\cdot\left(t_p^{\alpha}\bullet L_{\lambda,\OO}\right)\\
  &\;=\;
  (-\lambda^{w_0})(t_p^\alpha)\cdot d_xg_{\alpha}\cdot L_{\lambda,\OO}.
\end{align*}

Recall that $X_\QQ(H)\cong\ZZ$ is generated by $N_H:H\to\GL_1$, and identify the $\QQ$-rational characters of $H$ likewise with the $H$-representations $\QQ_{(j)}$ for $j\in\ZZ$. Again, $A_{(j)}=A\otimes_\QQ\QQ_{(j)}$ for any $\QQ$-algebra $A$\nomenclature[]{$A_{(j)}$}{codomain of the $A$-valued $j$-th power of the Norm character}.

\begin{proposition}\label{prop:generalmaninlemma}
For all $x\in\OO^\times$, $\alpha\geq 0$ and $v\in L_{\lambda,\OO}^{x,\alpha}$, there is a constant $\Omega_p^{\alpha,v}\in\OO$ with the following property: For every non-zero $H$-invariant functional
$$
\eta_j:\quad L_{\lambda,E}\;\to\;E_{(j)},
$$
we have the congruence\nomenclature[]{$\Omega_p^{\alpha,v}$}{$p$-adic constant}
\begin{equation}
\eta_j(v)\;\equiv\;
N_{F/\QQ}(x)^j\cdot
\Omega_p^{\alpha,v}\cdot
\eta_{j}(g_{0}v_0)
\pmod{\OO\cdot p^\alpha
\eta_{j}(g_{0}v_0)},
\label{eq:manincongruence}
\end{equation}
with
\begin{equation}
  \eta_{j}(g_{0}v_0)\;\neq\;0.
  \label{eq:etajg0v0nonvanishing}
\end{equation}
Furthermore, if $v=d_x\cdot h\cdot v_0$ we have
\begin{equation}
\Omega_p^{\alpha,v}\;=\;1.
  \label{eq:specialomegaone}
\end{equation}
\end{proposition}

\begin{proof}
Observe that for any $v\in L_{\lambda,\OO}^{x,\alpha}$,
$$
 L_{\lambda,\OO}^{x,\alpha}\;=\;
 (-\lambda^{w_0})(t_p^\alpha)\cdot
 d_x\cdot U(\lieu_{\OO}^{\alpha})\cdot g_{\alpha}v_0.
$$
In particular, we find a $u\in U(\lieu_{\OO}^{\alpha})$ with the property that
$$
v\;=\;(-\lambda^{w_0})(t_p^\alpha)\cdot d_x\cdot u\cdot g_{\alpha}v_0.
$$
According to Lemma \ref{lem:integralityrelation}, applying the decomposition \eqref{eq:integralenvelopingtransversality} to $u$, we find
$$
r\;=\;r_0+p^\alpha r_1\;\in\; \OO+p^\alpha U(\lieh_{\OO})
$$
and
$$
s\;=\;s_0+p^\alpha s_1\;\in\;\OO+p^\alpha U(\lieb^{-,0}_{\OO})
$$
satisfying the relation $u=rs$. Therefore,
\begin{align*}
  \eta_{j}(v)
  &\;=\;
  \eta_{j}\left(d_x\cdot rs\cdot g_{0}\cdot (-\lambda^{w_0})(t_p^\alpha)\cdot t_p^{\alpha}v_0\right)\\
  &\;=\;
  \eta_{j}\left(d_x\cdot rs\cdot g_{0}v_0\right)\\
  &\;=\;
  (s_0+p^\alpha s_1)\cdot
  (r_0+p^\alpha r_1)\cdot\Delta(\diag(x,1,\dots,1))\cdot
  \eta_{j}(g_{0}v_0)\\
  &\;\equiv\;
  N_H(\diag(x,1,\dots,1))^j\cdot
  r_0s_0\cdot
  \eta_{j}(g_{0}v_0)
  \pmod{\OO\cdot p^\alpha
  \eta_{j}(g_{0}v_0)}.
\end{align*}
This proves the first claim. The non-vanishing statement \eqref{eq:etajg0v0nonvanishing} is an immediate consequence of Corollary \ref{cor:transversality}.
\end{proof}

\section{Universal $p$-ordinary cohomology}\label{sec:cohomology}

For a compact open subgroup $K\subseteq G(\Adeles^{(\infty)})$\nomenclature[]{$K$}{compact open in $G(\Adeles^{(\infty)}$} we consider the locally symmetric space\nomenclature[]{$\mathscr X(K)$}{locally symmetric space for $G$ of level $K$}
$$
\mathscr X(K)\;:=\;G(\QQ)\backslash G(\Adeles)/K\cdot \widetilde{K}_\infty,
$$
where\nomenclature[]{$\widetilde{K}_\infty$}{standard maximal compact $K_\infty$ times connected component of the split center in $G$}
$$
\widetilde{K}_\infty\;:=\;Z^{\rm s}(\RR)^0K_\infty^0\subseteq G(\RR)
$$
with $Z^{\rm s}\subseteq G$ the maximal $\QQ$-split torus in the center $Z\subseteq G$ \nomenclature[]{$Z$}{center in $G$}\nomenclature[]{$Z^{\rm s}$}{split center in $G$} and $K_\infty\subseteq G(\RR)$ a standard maximal compact subgroup\nomenclature[]{$K_\infty$}{standard maximal compact in $G(\RR)$}. If $K=K_{n+1}\times K_n$ with compact open subgroups $K_{m}\subseteq\GL_m(\Adeles_F^{(\infty)})$\nomenclature[]{$K_{m}$}{compact open in $\GL_m(\Adeles_F^{(\infty)})$, a factor of $K$}, we have
\begin{equation}
  \mathscr X(K)\;=\;\mathscr X_{n+1}(K_{n+1})\times\mathscr X_n(K_n)
  \label{eq:Xdecomposition}
\end{equation}
with\nomenclature[]{$\mathscr X_m(K_m)$}{locally symmetric space for $\GL_m(\Adeles_F)$ of level $K_m$}
$$
  \mathscr X_m(K_m)\;:=\;GL_{m}(F)\backslash \GL_{m}(\Adeles_F)/K_{m}\cdot \widetilde{K}^{m}_\infty,
$$
$$
  \widetilde{K}_\infty^m\;:=\;Z_m^{\rm s}(F\otimes \RR)^0K^{m,0}_\infty\subseteq \GL_m(F\otimes\RR),
$$
  $Z_m^{\rm s}\subseteq \res_{F/\QQ}\GL_m$ the $\QQ$-split center\nomenclature[]{$Z_m$}{center of $\GL_m$}\nomenclature[]{$Z_m^{\rm s}$}{split center of $\res_{F/\QQ}\GL_m$} and $\widetilde{K}_\infty^{m}\subseteq \GL_n(F\otimes\RR)$ the corresponding standard maximal compact subgroup\nomenclature[]{$\widetilde{K}_\infty^m$}{standard maximal compact times (connected component of) the center in $\GL_m(F\otimes_\QQ\RR)$}.

\subsection{Arithmetic subgroups}

  We call $K$ torsion free or neat if for all $g\in G(\Adeles^{(\infty)})$ the arithmetic group\nomenclature[]{$\Gamma_g$}{arithmetic subgroup in $G(\QQ)$ corresponding to $gKg^{-1}$}
  $$
    \Gamma_g\;:=\;G(\QQ)\cap gKg^{-1}
  $$
  is torsion free resp.\ neat. Each neat $K$ is torsion-free and has the property that the arithmetic subgroups
  $$
    \Gamma_g\;\subseteq\;G(\RR)^0
  $$
  contain only totally positive elements. Every $K$ contains a neat $K$ of finite index.

  Each rational representation $L_{\lambda,E}$ gives rise to a sheaf $\underline{L}_{\lambda,E}$ on $\mathscr X(K)$\nomenclature[]{$\underline{L}_{\lambda,E}$}{sheaf associated to $L_{\lambda,E}$}.

  For $K$ neat, $\mathscr X(K)$ is a manifold and the sheaf cohomology of $\underline{L}_{\lambda,E}$ is a sum of the cohomologies of the arithmetic subgroups $\Gamma_g$ corresponding to $K$. The sheaf $\underline{L}_{\lambda,E}$ is non-trivial, provided that the {\em algebraicity condition}
  \begin{equation}
    H^0(Z(\QQ)\cap K;L_{\lambda,E})\;=\;L_{\lambda,E}
    \label{eq:lambdaalgebraicity}
  \end{equation}
  is satisfied, i.e.\ the centers of the arithmetic groups $\Gamma_g$ act trivially on $L_{\lambda,E}$. Condition \eqref{eq:lambdaalgebraicity} only depends on the Zariski closure of $\Gamma_g$ in $G$ and therefore is independent of $K$ if $K$ is sufficiently small.

  A dominant weight $\lambda\in X_\CC(\res_{F/\QQ}T_m)$ of $\res_{F/\QQ}\GL_m$ corresponds to a tuple $\lambda=(\lambda_{\tau,i})_{\tau:F\to\CC,1\leq i\leq m}$ of dominant weights $(\lambda_{\tau,1},\dots,\lambda_{\tau,m})\in\ZZ^m$, $\tau:F\to\CC$ running through all field embeddings. Complex conjugation canonically acts on the set of dominant weights $\lambda$ via its action on the embeddings $\tau:F\to\CC$, sending $\tau$ to $\tau^{\rm c}$, the postcomposition of $\tau$ with complex conjugation. Write $\lambda^{\rm c}$ for the complex conjugate weight attached to $\lambda$. We say that $\lambda$ is {\em essentially conjugate self-dual over $\QQ$} if
  $$
  \lambda=\lambda^{\vee,\rm c}+(w)
  $$
  for some $w\in \ZZ$. This is the same to say that
  $$
  L_{\lambda,\CC}\;\cong\;L_{\lambda,\CC}^{\vee,\rm c}\otimes(N_{F/\QQ}\circ\det)^{\otimes w}.
  $$
  The absolute Galois group $\Gal(\overline{\QQ}/\QQ)$ of $\QQ$ also acts on the dominant weights $\lambda$ via its action on $\res_{F/\QQ}T_m$, which permutes the entries in each tuple $(\lambda_{\tau,i})_{\tau:F\to\CC},$ $1\leq i\leq m$ via precomposition with $\tau$.

   We say that $\lambda$ is {\em arithmetic} or {\em strongly pure} (in the terminology of \citep{raghuram2015}), if
  $$
    \lambda^\sigma=(\lambda^{\vee,\rm c})^\sigma+(w_{\lambda})
  $$
  for all $\sigma\in\Gal(\overline{\QQ}/\QQ)$ and $w_{\lambda}\in\ZZ$ {\em independent} of $\sigma$. We call $w_{\lambda}$ the {\em (purity) weight} of $\lambda$\nomenclature[]{$w_\lambda$}{purity weight of $\lambda$}. In other words,
  $$
    \lambda_{\tau,i}+\lambda_{\tau^{\rm c},m+1-i}\;=\;w_{\lambda},
  $$
  for all $\tau$ and all $i$.
    
  We adopt the same terminology for dominant weights $\lambda=\lambda_{n+1}\otimes\lambda_n$ of $G$, that we call {\em arithmetic} if $\lambda_{n+1}$ and $\lambda_n$ both are arithmetic. For such a $\lambda,$ define
  $$
    w_\lambda\;:=\;w_{\lambda_{n+1}}+w_{\lambda_{n}}.
  $$

  Cuspidal cohomology
  $$
    H_{\rm cusp}^\bullet(\mathscr X(K);\underline{L}_{\lambda,\CC})
  $$
  vanishes if $\lambda$ is not arithmetic (use the K\"unneth Theorem to reduce to Clozel's \lq{}Lemme de puret\'e\rq{} for $\GL(n)$ in \citep{clozel1990}). Put\nomenclature[]{$l_0$}{length of cuspidal range for $G$}\nomenclature[]{$\rank G(\RR)$}{split rank of $G(\RR)$ as real Lie group}
  $$
    l_0\;:=\;\rank G(\RR)-\rank \widetilde{K}_\infty,
  $$
  and\nomenclature[]{$q_0$}{bottom degree for cuspidal cohomology of $G$}
  $$
    q_0\;:=\;\frac{\dim{\mathscr X}(K)-l_0}{2}.
  $$
  Then $q_0$ is an integer, which is known as the {\em bottom degree} of $G$, because cuspidal cohomology vanishes in degree $q<q_0$ and if it is non-zero, then it is non-zero precisely in degrees $q_0\leq q\leq q_0+l_0$, $q_0+l_0$ being the {\em top degree}.

\subsection{Nearly ordinary cohomology}

  Recall the compact open subgroups
$$
I_\alpha\;=\;\prod_{\mathfrak{p}\mid p}I_{v_{\mathfrak{p}}(p)\alpha}^{n+1}\times I_{v_{\mathfrak{p}}(p)\alpha}^n
$$
and
$$
I_{\alpha',\alpha}\;=\;\prod_{\mathfrak{p}\mid p}I_{v_{\mathfrak{p}}(p)\alpha',v_{\mathfrak{p}}(p)\alpha}^{n+1}\times I_{v_{\mathfrak{p}}(p)\alpha',v_{\mathfrak{p}}(p)\alpha}^n
$$
of $G(\QQ_p)$ from section \ref{sec:generichecke} and the corresponding Hecke algebras $\mathcal H_A(\alpha)$ and $\mathcal H_A(\alpha',\alpha)$ for $\alpha\geq\alpha'\geq 0,$ $\alpha>0,$ which contain the distinguished Hecke operator $U_p$ defined in \eqref{eq:definitionofUp}.

  Consider any family of compact open subgroups $K_{\alpha',\alpha}\subseteq G(\Adeles^{(\infty)})$, $\alpha\geq\alpha'\geq 0$ and $\alpha>0$, which admits a decomposition $K_{\alpha',\alpha}=I_{\alpha',\alpha}\times K^{(p)}$ with $K^{(p)}$ a compact open outside $p$, trivial at $p$ and independent of $\alpha',\alpha$\nomenclature[]{$K_{\alpha',\alpha}$}{compact open in $G(\Adeles^{(\infty)})$ with $p$-component $I_{\alpha',\alpha}$}. Define $\alpha_0^K$ to be the minimal $\alpha_0>0$ such that $K_{0,\alpha_0}$ is neat\nomenclature[]{$\alpha_0^K$}{minimal $\alpha_0$ such that $K_{0,\alpha_0}$ is neat}.

The cohomology
$$
H_{?}^{\bullet}({\mathscr X}(K_{\alpha',\alpha}); \underline{L}_{\lambda,E})
$$
for $?\in\{-,\rm c,!\}$ is naturally a module over the Hecke algebra of level $K_{\alpha',\alpha}$. At $p$ we renormalize the action of $U_p$ by multiplication by the scalar $\lambda^\vee(t_p)=(-\lambda^{w_0})(t_p)$\nomenclature[]{$\lambda^\vee$}{highest weight of the contragredient to $L_{\lambda,E}$}. Likewise, we may renormalize every Hecke operator in $\mathcal H_\OO(0,\alpha)\subseteq\mathcal H_\OO(\alpha',\alpha)$ via \eqref{eq:taction}. Then $\mathcal H_\OO(0,\alpha)$ acts on cohomology\nomenclature[]{$H_{?}^\bullet(-;-)$}{sheaf cohomology with support condition $?$}
\begin{equation}
  H_{?}^{\bullet}({\mathscr X}(K_{\alpha',\alpha}); \underline{L}_{\lambda,\OO})
  \label{eq:integralcohomology}
\end{equation}
with $p$-integral coefficients $p$-optimally. Attached to this action is an ordinary projector $e_p$\nomenclature[]{$e_p$}{Hida's ordinary projector}, which projects onto the subspace\nomenclature[]{$H_{?,\rm ord}^\bullet(-;-)$}{$U_p$-ordinary part of $H_{?}^\bullet(-;-)$}
\begin{equation}
  H_{?,\rm ord}^{\bullet}({\mathscr X}(K_{\alpha',\alpha}); \underline{L}_{\lambda,\OO})
  \;:=\;e_p  H_{?}^{\bullet}({\mathscr X}(K_{\alpha',\alpha}); \underline{L}_{\lambda,\OO})
  \label{eq:ordinarycohomology}
\end{equation}
of \eqref{eq:integralcohomology} on which $U_p$ acts invertibly. More generally, we will consider the spaces
$$
  H_{?,\rm ord}^{\bullet}({\mathscr X}(K_{\alpha',\alpha}); \underline{L}_{\lambda,A})
$$
for\nomenclature[]{$A$}{from section 4 on, $A$ is in $\{
  {\OO/p^\alpha\OO},
  {p^{-\alpha}\OO/\OO},
  {\OO},
  {K/\OO}\}$}
\begin{equation}
  A\;\in\;\{
  {\OO/p^\alpha\OO},
  {p^{-\alpha}\OO/\OO},
  {\OO},
  {K/\OO}\},
  \label{eq:choiceofA}
\end{equation}
where\nomenclature[]{$L_{\lambda,A}$}{tensor product of $L_{\lambda,\OO}$ with $A$}
$$
  L_{\lambda,A}\;=\;L_{\lambda,\OO}\otimes_\OO A.
$$
Since $U_p$ is a product of the local operators $V_\nu\otimes 1$, $1\otimes V_\nu,$ $1\leq\nu\leq n,$ and since $V_{n+1}\otimes 1$ acts invertibly as well (for every place $v\mid p$) from section \ref{sec:padicheckealgebras}, the action of $\mathcal H_\OO(\alpha',\alpha)$ on these nearly ordinary cohomology naturally extends to an action of $\mathcal H_\OO^{\rm ord}(\alpha',\alpha)$ and
$$
  H_{?,\rm ord}^{\bullet}({\mathscr X}(K_{\alpha',\alpha}); \underline{A})\;=\;
  \mathcal H_\OO^{\rm ord}(\alpha',\alpha)\otimes_{\mathcal H_\OO(\alpha',\alpha)}
  H_{?}^{\bullet}({\mathscr X}(K_{\alpha',\alpha}); \underline{A})
$$
as Hecke modules for $A\in\{\OO/p^\beta\OO,p^{-\beta}\OO/\OO,\OO,E/\OO\}$. Put
$$
  H_{?,\rm ord}^{\bullet}({\mathscr X}(K_{\alpha',\alpha}); \underline{E})\;:=\;
  H_{?,\rm ord}^{\bullet}({\mathscr X}(K_{\alpha',\alpha}); \underline{\OO})\otimes_\OO E.
$$
Then in all cases, passing to ordinary parts is an exact functor.

\begin{proposition}\label{prop:independenceofalpha}
  For $A$ as in \eqref{eq:choiceofA}, any dominant weight $\lambda,$ and every $\alpha\geq\alpha'\geq 0$ with $\alpha\geq\alpha_0^K$, we have for every $\alpha''\geq\alpha$ a canonical isomorphism
  $$
    H_{?,\rm ord}^{\bullet}({\mathscr X}(K_{\alpha',\alpha}); \underline{L}_{\lambda,A})\;=\;
    H_{?,\rm ord}^{\bullet}({\mathscr X}(K_{\alpha',\alpha''}); \underline{L}_{\lambda,A})
  $$
  of Hecke modules.
\end{proposition}

\begin{proof}
  The proof proceeds as the proof of the isomorphism (4.7c) on p.445 of \citep{hida1995}, using the explicit left coset decomposition of the Hecke operators $T_1, T_2, T_3$ of loc.\ cit., adapted to right coset decompositions we are working with.
\end{proof}

\begin{remark}
  As in \citep{hida1995}, Proposition \ref{prop:independenceofalpha} holds for  more general coefficient sheaves, in particular for the ones considered in the proof of the Control Theorem (Theorem \ref{thm:regularcontrol}) below.
\end{remark}

\subsection{Independence of weight}\label{sec:independenceofweight}

Write $\OO[\lambda^{w_0}]$ for the $\OO$-module of rank $1$ on which $B^-$ acts via $\lambda^{w_0}$\nomenclature[]{$\OO[\lambda^{w_0}$}{free $\OO$-module of rank $1$ on which $B^-$ acts via $\lambda^{w_0}$}. We assume that we are given a fixed generator $1\in\OO[\lambda^{w_0}]$, which we use to identify this space with $\OO$. Consider the inclusion
$$
i:\quad \OO[\lambda^{w_0}]\;\to\;L_{\lambda,\OO},\quad c\mapsto c\cdot v_0,
$$
and the projection
$$
p:\quad L_{\lambda,\OO}\;\to\; \OO[\lambda^{w_0}],
$$
which projects $T$-equivariantly onto the lowest weight space. By our identification of $\OO[\lambda^{w_0}]$ with $\OO$, we obtain maps\nomenclature[]{$i_\lambda$}{inclusion $\OO\to L_{\lambda,\OO}$ of lowest weight space}
$$
i_{\lambda}:\quad \OO\to L_{\lambda,\OO},
$$
and\nomenclature[]{$p_\lambda$}{projection $L_{\lambda,\OO}\to \OO$ to lowest weight space}
$$
p_{\lambda}:\quad L_{\lambda,\OO}\to \OO.
$$
\begin{theorem}\label{thm:independenceofweight}
For $\alpha\geq\alpha'\geq 0$, $\alpha\geq\alpha_0^K,$ and $A\in\{\OO/p^{\alpha'}\OO,p^{-\alpha'}\OO/\OO\}$ the maps $i_\lambda$ and $p_\lambda$ induce isomorphisms\nomenclature[]{$\iota_\lambda$}{map (isomorphism) on (universal nearly) ordinary cohomology induced by $i_\lambda$}
\begin{equation}
  \iota_\lambda:\quad
  H_{?,\rm ord}^{\bullet}({\mathscr X}(K_{\alpha',\alpha}); \underline{A})
  \to
    H_{?,\rm ord}^{\bullet}({\mathscr X}(K_{\alpha',\alpha}); \underline{L}_{\lambda,A})
  \label{eq:iotaiso}
\end{equation}
and\nomenclature[]{$\pi_\lambda$}{map (isomorphism) on (universal nearly) ordinary cohomology induced by $p_\lambda$}
\begin{equation}
  \pi_\lambda:\quad
  H_{?,\rm ord}^{\bullet}({\mathscr X}(K_{\alpha',\alpha}); \underline{L}_{\lambda,A})
  \to
    H_{?,\rm ord}^{\bullet}({\mathscr X}(K_{\alpha',\alpha}); \underline{A}),
  \label{eq:piiso}
\end{equation}
which are inverses of each other, Hecke-equivariant outside $p$, and satisfy
\begin{equation}
T\circ\iota_\lambda\;=\;\iota_\lambda\circ T,
\quad\text{and}\quad
T\circ\pi_\lambda\;=\;\pi_\lambda\circ T,
  \label{eq:Upiso}
\end{equation}
for $T\in\mathcal H_\OO(0,\alpha)$ and for every $t\in T(\ZZ_p/p^{\alpha'}\ZZ_p)$,
\begin{equation}
  \langle t\rangle\circ\iota_\lambda\;=\;
  \iota_\lambda\circ \lambda^{w_0}(t)\langle t\rangle,
\quad\text{and}\quad
\lambda^{w_0}(t)\langle t\rangle\circ\pi_\lambda\;=\;\pi_\lambda\circ \langle t\rangle,
  \label{eq:diamondiso}
\end{equation}
\end{theorem}

\begin{proof}
  We discuss the case $A=\OO/p^{\alpha'}\OO$, the other case follow similarly. Consider the short exact sequences
\begin{equation}
  0\to \OO/p^{\alpha'}\OO\to L_{\lambda,\OO/p^{\alpha'}\OO}\to\coker i_\lambda\otimes\OO/p^{\alpha'}\OO\to 0,
  \label{eq:icokernelsequence}
\end{equation}
  $$
  0\to \ker p_\lambda\otimes\OO/p^{\alpha'}\OO\to L_{\lambda,\OO/p^{\alpha'}\OO}\to\OO/p^{\alpha'}\OO\to 0.
  $$
  We have
  $$
  [I_{\alpha',\alpha} t_p I_{\alpha',\alpha}]\;=\;\bigsqcup_{u\in U(\OO)/t_p U(\OO)t_p^{-1}} ut_p I_{\alpha,\alpha},
  $$
  which shows that
  $$
      H_{?}^{\bullet}({\mathscr X}(K_{\alpha',\alpha}); \underline{\coker i_\lambda\otimes\OO/p^{\alpha'}\OO}).
  $$
  is a $U_p$-module, and likewise for $\kernel p_\lambda\otimes\OO/p^{\alpha'}\OO$.

  By construction, $(-\lambda^{w_0})(t_p)\cdot t_p$ acts nilpotently on both $\coker i_\lambda\otimes\OO/p^{\alpha'}\OO$ and $\ker p_\lambda\otimes\OO/p^{\alpha'}\OO$. Therefore, $U_p$ acts nilpotently on
  $$
    H_{?}^{\bullet}({\mathscr X}(K_{\alpha',\alpha}); \underline{\coker i_\lambda\otimes\OO/p^{\alpha'}\OO}).
    $$
  This shows that
  $$
    H_{?,\rm ord}^{\bullet}({\mathscr X}(K_{\alpha',\alpha}); \underline{\coker i_\lambda\otimes\OO/p^{\alpha'}\OO})\;=\;0.
  $$
  Since projection to the ordinary part is an exact functor, the long exact sequence attached to \eqref{eq:icokernelsequence} implies that $i_\lambda$ must be an isomorphism of $\OO$-modules on the ordinary part. The same argument shows that $p_\lambda$ induces the inverse isomorphism.
\end{proof}

We introduce for every dominant weight $\lambda$ the notation\nomenclature[]{$L_{\lambda^\vee,A}^\circ$}{(Pontryagin) dual lattice}
$$
L_{\lambda^\vee,E}^\circ\;:=\;\Hom_E(L_{\lambda,E},E)
$$
for the contragredient. It comes with a canonical perfect pairing
$$
\langle-,-\rangle:\quad L_{\lambda,E}\otimes_E L_{\lambda^\vee,E}^\circ\to E.
$$
Define the dual lattice for $L_{\lambda,\OO}$ accordingly as
$$
L_{\lambda^\vee,\OO}^\circ\;:=\;\{f\in L_{\lambda^\vee,E}^\circ\mid f(L_{\lambda,\OO})\subseteq\OO\}.
$$
Then it is easy to check that
$$
\langle-,-\rangle:\quad L_{\lambda,\OO}\otimes_\OO L_{\lambda^\vee,\OO}^\circ\to \OO
$$
is still a perfect pairing: As a sublattice of $L_{\lambda^\vee,E}^\circ$ is generated by a $B$-highest weight vector $v_0^\circ$ dual to $v_0$. In particular, $L_{\lambda^\vee,\OO}^\circ$ is canonically identified with the $\OO$-dual of $L_{\lambda,\OO}$.

Along the same lines we see that $L_{\lambda^\vee,\OO}^\circ$ is also canonically identified with the Pontryagin dual of $L_{\lambda,E/\OO}$ and likewise for finite torsion coefficients.

With this notation at hand, define the universal nearly ordinary cohomology with torsion coeffients as\nomenclature[]{$\mathcal H_{?,\rm ord}^\bullet(K_{\infty,\infty};\lambda,A)$}{universal nearly ordinary homology of weight $\lambda$ with coefficients in $A\in\{\OO,E/\OO\}$}
$$
  \mathcal H_{\rm ord}^\bullet(K_{\infty,\infty};\lambda,E/\OO)\;:=\;
  \varinjlim\limits_{\alpha,\alpha'}
  H_{\rm ord}^{\bullet}({\mathscr X}(K_{\alpha',\alpha}); \underline{L}_{\lambda,p^{-\alpha'}\OO/\OO}).
$$
Likewise, we obtain with respect to the transfer maps its Pontryagin dual
$$
  \mathcal H_{\rm c,ord}^\bullet(K_{\infty,\infty};\lambda^\vee,\OO)\;:=\;
  \varprojlim_{\alpha,\alpha'}
  H_{\rm c,ord}^{\bullet}({\mathscr X}(K_{\alpha',\alpha}); \underline{L}_{\lambda^\vee,\OO/p^{\alpha'}\OO}^\circ).
$$
We also consider the mutually Pontryagin dual nearly ordinary cohomology groups of the form
$$
  \mathcal H_{\rm ord}^\bullet(K_{\infty,\infty};\lambda,\OO)\;:=\;
  \varprojlim_{\alpha,\alpha'}
  H_{\rm ord}^{\bullet}({\mathscr X}(K_{\alpha',\alpha}); \underline{L}_{\lambda,\OO/p^{\alpha'}\OO}^\circ),
$$
$$
  \mathcal H_{\rm c,ord}^\bullet(K_{\infty,\infty};\lambda^\vee,E/\OO)\;:=\;
  \varinjlim_{\alpha,\alpha'}
  H_{\rm c,ord}^{\bullet}({\mathscr X}(K_{\alpha',\alpha}); \underline{L}_{\lambda^\vee,p^{-\alpha'}\OO/\OO}).
$$
We define analogously the maps $\iota_{\lambda^\vee}^\circ$ and $p_{\lambda^\vee}^\circ$ as the Pontryagin duals of the maps $p_{\lambda}$ and $\iota_{\lambda}$ respectively.\nomenclature[]{$\iota_{\lambda^\vee}^\circ$}{Pontryagin dual to $p_{\lambda}$}\nomenclature[]{$p_{\lambda^\vee}^\circ$}{Pontryagin dual to $\iota_{\lambda}$}
  
Then Theorem \ref{thm:independenceofweight} shows
\begin{corollary}\label{cor:independenceofweight}
  The map $\pi_\lambda$ induces an isomorphism
  $$
    \pi_\lambda:\quad
    \mathcal H_{?,\rm ord}^\bullet(K_{\infty,\infty};\lambda,E/\OO)\;\to\;
    \mathcal H_{?,\rm ord}^\bullet(K_{\infty,\infty};0,E/\OO)\;=:\;
    \mathcal H_{?,\rm ord}^\bullet(K_{\infty,\infty};E/\OO)
  $$
  respecting the actions of $\mathcal H_\OO(0,\alpha)$ and of the Hecke operators outside $p$, and for every $t\in T(\ZZ_p)$,
  $$
    \lambda^{w_0}(t)\langle t\rangle\circ\pi_\lambda\;=\;\pi_\lambda\circ \langle t\rangle.
  $$
  Likewise, the map $\pi_\lambda^\circ$ induces an isomorphism
  $$
    \pi_\lambda^\circ:\quad
    \mathcal H_{?,\rm ord}^\bullet(K_{\infty,\infty};\lambda,\OO)\;\to\;
    \mathcal H_{?,\rm ord}^\bullet(K_{\infty,\infty};0,\OO)\;=:\;
    \mathcal H_{?,\rm ord}^\bullet(K_{\infty,\infty};\OO)
  $$
  respecting the actions of $\mathcal H_\OO(0,\alpha)$ and of the Hecke operators outside $p$, and for every $t\in T(\ZZ_p)$,
  $$
    \lambda(t)\langle t\rangle\circ\pi_\lambda^\circ\;=\;\pi_\lambda^\circ\circ \langle t\rangle.
  $$
\end{corollary}

\subsection{The universal nearly ordinary Hecke algebra}

Put\nomenclature[]{$\Lambda$}{completed $\OO$-group ring for $T(\ZZ_p)$}\nomenclature[]{$\Lambda^\circ$}{completed $\OO$-group ring for $T(\ZZ_p)/({\rm torsion})$}
$$
\Lambda\;:=\;\OO[[T(\ZZ_p)]]\;=\;\varprojlim_\alpha \OO[T(\ZZ_p/p^\alpha\ZZ_p)].
$$
$$
\Lambda^\circ\;:=\;\OO[[T(\ZZ_p)/T(\ZZ_p)_{\rm tors}]]
$$
Then $\Lambda$ is a complete Noetherian semi-local ring with each local factor isomorphic to $\Lambda^\circ$ and each $\mathcal H_\OO(\alpha',\alpha)$ carries a canonical $\Lambda$-module structure. Following Hida, define the universal nearly ordinary Hecke algebra $\mathbf{h}_{?,\rm ord}(K_{\infty,\infty};\OO)$ as the $\Lambda$-subalgebra generated by the image of the canonical map\nomenclature[]{${\bf h}_{?,\rm ord}(K_{\infty,\infty};\OO)$}{Hida's universal nearly ordinary Hecke algebra for $G$}
$$
\left[\varinjlim_{\alpha,\alpha'}\mathcal H_\OO(\alpha',\alpha)\right]\otimes_\OO
\mathcal H_\OO(K^{S(K)},G(\Adeles^{S(K)}))\to\End_\OO(\mathcal H_{?,\rm ord}^\bullet(K_{\infty,\infty};E/\OO)),
$$
where $S(K)$ denotes the minimal finite set of places of $\QQ$ containing $p$ and $\infty$\nomenclature[]{$S(K)$}{finite set of places depending on compact open $K$} and for which there is a factorization
$$
K_{\alpha',\alpha}\;=\;L_{S(K)}\times \prod_{v\not\in S(K),v\nmid \infty}G(\ZZ_v)
$$
for some compact open subgroup
$$
L_{S(K)}\subseteq \prod_{v\in S(K),v\nmid\infty}G(\QQ_v).
$$
Since $\mathcal H_\OO(K^{S(K)},G(\Adeles^{S(K)}))$ is by definition a product of spherical Hecke algebras, $\mathbf{h}_{?,\rm ord}(K_{\infty,\infty};\OO)$ is commutative. By Corollary \ref{cor:independenceofweight}, $\mathcal H_{?,\rm ord}^\bullet(K_{\infty,\infty};\lambda,E/\OO)$ is an $\mathbf{h}_{?,\rm ord}(K_{\infty,\infty};\OO)$-module, which differs from $\mathcal H_{?,\rm ord}^\bullet(K_{\infty,\infty};E/\OO)$ only in terms of the $\Lambda$-action.

\subsection{Galois representations}\label{sec:galois}

In this section, assume that $S(K)=\{p,\infty\}$, i.e.\ $K^{(p)}\;=\;G\left(\widehat{\ZZ}^{(p)}\right)$ is the product of the standard maximal compact open subgroup over all primes $\ell\neq p$ and
$$
K_{\alpha',\alpha}\;=\;K^{(p)}\times I_{\alpha',\alpha}.
$$
Then the Hecke algebra of interest is
$$
\mathcal H_\OO(K_{\alpha',\alpha},G(\Adeles^{(p\infty)})\times\Delta_G)\;=\;
\otimes_{v\nmid p\infty}'
\left(
\mathcal H_\OO(\GL_{n+1}(\OO_v),\GL_{n+1}(F_v))
\otimes
\mathcal H_\OO(\GL_{n}(\OO_v),\GL_{n}(F_v))
\right)\otimes\mathcal H_{\OO}(\alpha',\alpha).
$$
In $\mathcal H_\OO(\GL_{n}(\OO_v),\GL_{n}(F_v))$, we find the standard Hecke operators
$$
T_{v,\nu}\;=\;\GL_{n}(\OO_v)\varpi_v^{\omega_\nu}\GL_{n}(\OO_v),\quad 1\leq\nu\leq n.
$$
Consider as in \eqref{eq:Fheckepolynomial} the spherical (reciprocal) Hecke polynomial
$$
H_{F_v,n}(X):=\sum_{\nu=0}^n (-1)^\nu q_v^{\frac{\nu(\nu-1)}{2}}T_\nu X^{n-\nu}\in\mathcal H_A^n(\alpha',\alpha).
$$
Recall that $H_{F_v,n}(X)$ admits a factorization
$$
H_{F_v,n}(X)\;=\;\prod_{i=1}^n (X-\widetilde{U}_i)
$$
for $\widetilde{U}_i$ in the parabolic Hecke algebra at $v$. Define the Hecke polynomial\nomenclature[]{$H_v(X)$}{reciprocal spherical Hecke polynomial for $G$ at $v$}
$$
H_v(X)\;:=\;\prod_{i=1}^{n+1}\prod_{j=1}^n (X-\widetilde{U}_i\otimes\widetilde{U}_j)\;\in\;
\left(\mathcal H_\OO(\GL_{n+1}(\OO_v),\GL_{n+1}(F_v))
\otimes
\mathcal H_\OO(\GL_{n}(\OO_v),\GL_{n}(F_v))\right)[X].
$$
This is the spherical Hecke polynomial for $\GL_{n+1}\times\GL_n$ over $F_v$.

For $1\leq\mu\leq n+1$ and $1\leq\nu\leq n$, the operators $T_{v,\mu}\otimes 1$ and $1\otimes T_{v,\nu}$ act on
$$
  \sum_q H_{?,\rm ord}^{q}({\mathscr X}(K_{\alpha',\alpha}); \underline{L}_{\lambda,A})
$$
  for $?\in\{-,\rm c,!\}$ and $A\in\{\OO,E,E/\OO,p^{-\beta}\OO/\OO,\OO/p^\beta\OO\}$.

  The operators $T_{v,\mu}\otimes T_{v,\nu}$ for $v\not\in S(p)$ together with the image of $\mathcal H_\OO(\alpha',\alpha)$ generate the nearly ordinary Hecke algebra ${\rm h}_{?,\rm ord}(K_{\alpha',\alpha};\lambda,\OO)$ over $\OO[T(\ZZ_p/p^{\alpha'}\ZZ_p]$\nomenclature[]{${\rm h}_{?,\rm ord}(K_{\alpha,\alpha};\lambda,\OO)$}{nearly ordinary Hecke algebra of weight $\lambda$ and level $K_{\alpha,\alpha}$}. Passing to the projective limit over $\alpha',\alpha$, we obtain the universal nearly ordinary Hecke algebra ${\bf h}_{?,\rm ord}(K_{\infty,\infty};\lambda,\OO)$ over $\Lambda$\nomenclature[]{${\bf h}_{?,\rm ord}(K_{\infty,\infty};\lambda,\OO)$}{Hida's universal nearly ordinary Hecke algebra of weight $\lambda$}. Recall for $\lambda=0,$ define
  $$
    {\bf h}_{?,\rm ord}(K_{\infty,\infty};\OO)\;:=\;
    {\bf h}_{?,\rm ord}(K_{\infty,\infty};0,\OO).
  $$
  Finally, for any $q\in\ZZ,$ let ${\rm h}_{?,\rm ord}^q(K_{\alpha,\alpha};\lambda,\OO)$ denote the image of the canonical map\nomenclature[]{${\rm h}_{?,\rm ord}^q(K_{\alpha,\alpha};\lambda,\OO)$}{nearly ordinary Hecke algebra of level $K_{\alpha,\alpha}$ in cohomological degree $q$}
  $$
    {\rm h}_{?,\rm ord}(K_{\alpha,\alpha};\lambda,\OO)
    \;\to\;\End_{\OO}H_{?,\rm ord}^{q}({\mathscr X}(K_{\alpha',\alpha}); \underline{L}_{\lambda,A}).
  $$
  We adopt the same notation for univeral nearly ordinary Hecke algebras.

  Write $\lambda=\lambda_{n+1}\otimes\lambda_n$ for dominant weights $\lambda_m$ on $\res_{F/\QQ}\GL_m$. Then according to \eqref{eq:Xdecomposition}, we have a K\"unneth spectral sequence
  $$
    E_{pq}^2:=\bigoplus_{q_{n+1}+q_n=q}
    {\rm Tor}_{-p}^{\OO}
        \left(
        H_{?}^{q_{n+1}}({\mathscr X}_{n+1}(K^{n+1}_{\alpha,\alpha}); \underline{L}_{\lambda_{n+1},E/\OO})
        ,
        H_{?}^{q_{n}}({\mathscr X}_{n}(K^{n}_{\alpha,\alpha}); \underline{L}_{\lambda_n,E/\OO})
        \right)
  $$
  $$
    \Longrightarrow\quad
    H_{?}^{p+q}({\mathscr X}(K_{\alpha,\alpha}); \underline{L}_{\lambda,E/\OO}).
  $$
  Since $\OO$ is a principal ideal domain, ${\rm Tor}_\bullet$ vanishes in degrees $\geq 2$, whence we deduce a short exact sequence
  $$
    0\to
      \bigoplus_{q_{n+1}+q_n=q}
      H_{?}^{q_{n+1}}({\mathscr X}_{n+1}(K^{n+1}_{\alpha,\alpha}); \underline{L}_{\lambda_{n+1},E/\OO})\otimes
      H_{?}^{q_{n}}({\mathscr X}_{n}(K^{n}_{\alpha,\alpha}); \underline{L}_{\lambda_n,E/\OO})
        \to
        $$
        $$
     H_{?}^{p+q}({\mathscr X}(K_{\alpha,\alpha}); \underline{L}_{\lambda,E/\OO})
     \to
    \!\!\!\!\!\!\!\!\!\!
    \bigoplus_{q_{n+1}+q_n=q+1}
    \!\!\!\!\!\!\!\!\!\!
        {\rm Tor}_{1}^{\OO}
        \left(
        H_{?}^{q_{n+1}}({\mathscr X}_{n+1}(K^{n+1}_{\alpha,\alpha}); \underline{L}_{\lambda_{n+1},E/\OO})
        ,
        H_{?}^{q_{n}}({\mathscr X}_{n}(K^{n}_{\alpha,\alpha}); \underline{L}_{\lambda_n,E/\OO})
        \right)
      \to 0.
      $$
      In particular, the edge morphism of above spectral sequence provides us with a canonical monomorphism
      $$
      \bigoplus_{q_{n+1}+q_n=q}
      \varinjlim_{\alpha}
      H_{?}^{q_{n+1}}({\mathscr X}_{n+1}(K^{n+1}_{\alpha,\alpha}); \underline{L}_{\lambda_{n+1},E/\OO})\otimes
      \varinjlim_{\alpha}
      H_{?}^{q_{n}}({\mathscr X}_{n}(K^{n}_{\alpha,\alpha}); \underline{L}_{\lambda_n,E/\OO})
      \;\to\;
      \varinjlim_{\alpha}H_{?}^{q}({\mathscr X}(K_{\alpha,\alpha}); \underline{L}_{\lambda,E/\OO}),
  $$
  and likewise for ordinary cohomology.

  For any field $k$ which is either (an extension of) the residue field of $\OO$\nomenclature[]{$k$}{(extension of the) residue field of $\OO$}, or the field $E$, the K\"unneth spectral sequence degenerates and induces a canonical isomorphism
  \begin{equation}
      \bigoplus_{q_{n+1}+q_n=q}
      \varinjlim_{\alpha}
      H^{q_{n+1}}({\mathscr X}_{n+1}(K^{n+1}_{\alpha,\alpha}); \underline{L}_{\lambda_{n+1},k})\otimes
      \varinjlim_{\alpha}
      H^{q_{n}}({\mathscr X}_{n}(K^{n}_{\alpha,\alpha}); \underline{L}_{\lambda_n,k})
      \;\cong\;
      \varinjlim_{\alpha}H^{q}({\mathscr X}(K_{\alpha,\alpha}); \underline{L}_{\lambda,k}).
      \label{eq:kuennethoverfields}
  \end{equation}
  In particular, we obtain a canonical isomorphism
  \begin{equation}
    {\bf h}_{\rm ord}(K_{\infty,\infty}^{n+1};k)\otimes
    {\bf h}_{\rm ord}(K_{\infty,\infty}^{n};k)
    \;\cong\;
    {\bf h}_{\rm ord}(K_{\infty,\infty};k),
    \label{eq:hasaproduct}
  \end{equation}
  of $\Lambda$-algebras. Here ${\bf h}_{\rm ord}(K_{\infty,\infty}^{m};k)$ is the Hecke algebra generated by $T_{v,m}$, $v\not\in S(p)$, and $\prod_{v\mid p}\mathcal H_k^m(\alpha,\alpha)$ acting on\nomenclature[]{${\bf h}_{\rm ord}(K_{\infty,\infty}^{m};k)$}{universal nearly ordinary Hecke algebra for $\GL(m)$ over $k$}
  $$
      \varinjlim_{\alpha}
      \sum_{q_m}H^{q_{m}}({\mathscr X}_{m}(K^{m}_{\alpha,\alpha}); \underline{L}_{\lambda_m,k})
  $$
  over the corresponding Iwasawa algebra $\OO[[T_m(\OO_p)]]$.

   Let $\mathfrak{m}$ denote a maximal ideal in ${\bf h}_{\rm ord}(K_{\infty,\infty};\OO)$ with residue field $k$\nomenclature[]{$\mathfrak{m}$}{(non-Eisenstein) maximal ideal in ${\bf h}_{\rm ord}(K_{\infty,\infty};\OO)$}. Write $S$ for the set of finite places of $F$ containing the places above $p$ and the places which ramify in $F/\QQ$\nomenclature[]{$S$}{set of places of $F$ containing $S(p)$ and the finite places ramified in $F/\QQ$}.

  Assume that (after possibly enlarging $k$) the following condition is satisfied:
 \begin{itemize}
  \item[(i)] There exists a continuous semisimple Galois representation\nomenclature[]{$\overline{\rho}_{\mathfrak{m}}$}{residual Galois representation of dimension $(n+1)n$ corresponding to $\mathfrak{m}$}
    $$
    \overline{\rho}_{\mathfrak{m}}:\quad \Gal\left(\overline{\QQ}/F\right)\to \GL_{(n+1)n}(k),
    $$
    such that for every finite place $v\not\in S$ the image $\overline{\rho}_{\mathfrak{m}}({\rm Frob}_v)$ of the geometric Frobenius element has characteristic polynomial
    $$
      H_v(X)\;\in\;({\bf h}_{\rm ord}(K_{\infty,\infty};\OO)/{\mathfrak{m}})[X].
    $$
  \end{itemize}
   By the Chebotarev density theorem, property (i) characterizes the Galois representation in question uniquely up to isomorphism, provided it exists. According to \eqref{eq:hasaproduct}, we find (after possibly enlarging $k$ once again) maximal ideals $\mathfrak{m}_{m}$ in ${\bf h}_{\rm ord}(K_{\infty,\infty}^{m};\OO)$, $m\in\{n,n+1\}$\nomenclature[]{$\mathfrak{m}_m$}{(non-Eisenstein) maximal ideal in ${\bf h}_{\rm ord}(K_{\infty,\infty}^{m};\OO)$}, such that condition (i) amounts to
  \begin{itemize}
  \item[(i')] There exist continuous semisimple Galois representations\nomenclature[]{$\overline{\rho}_{\mathfrak{m}_m}$}{residual Galois representation of dimension $m$ corresponding to $\mathfrak{m}_m$}
    $$
      \overline{\rho}_{\mathfrak{m}_m}:\quad \Gal\left(\overline{\QQ}/F\right)\to \GL_{m}(k),
    $$
    such that for every finite place $v\not\in S$ the image $\overline{\rho}_{\mathfrak{m}_m}({\rm Frob}_v)$ of the geometric Frobenius element has characteristic polynomial
    $$
      H_{F_v,m}(X)\;\in\;({\bf h}_{\rm ord}(K_{\infty,\infty}^m;\OO)/{\mathfrak{m}_m})[X],
    $$
    for $m\in\{n,n+1\}$.
  \end{itemize}

  \begin{remark}
  Condition (i') implies
  $$
    \overline{\rho}_{\mathfrak{m}}\;=\;
    \left(\overline{\rho}_{\mathfrak{m}_{n+1}}\otimes
    \overline{\rho}_{\mathfrak{m}_{n}}\right)^{\rm ss}.
  $$
  \end{remark}
  
  \begin{remark}\label{rmk:modpcharacterization}
  Condition (i') implies that action of $\OO[[Z(\ZZ_p)]]\subseteq\Lambda$ is compatible with the determinant of $\overline{\rho}_{\mathfrak{m}}$ in the following sense. On the one hand, we know that for $v\not\in S$,
  $$
  T_{v,n+1}\otimes 1\;\equiv\;
  \chi_{\rm cyc}^{\otimes\frac{(n+1)n}{2}}\otimes
  \det\overline{\rho}_{\mathfrak{m}_{n+1}}({\rm Frob}_v)\pmod{\mathfrak{m}},
  $$
  and
  $$
  1\otimes T_{v,n}\;\equiv\;
  \chi_{\rm cyc}^{\otimes\frac{n(n-1)}{2}}\otimes
  \det\overline{\rho}_{\mathfrak{m}_{n}}({\rm Frob}_v)\pmod{\mathfrak{m}}.
  $$
  On the other hand, suitable powers of these Hecke operators away from $p$ may be considered as elements of $\Lambda$ (as diamond operators).
  \end{remark}

  \begin{definition}
  A maximal ideal $\mathfrak{m}$ in ${\bf h}_{\rm ord}(K_{\infty,\infty};\OO)$ is called {\em non-Eisenstein}, if (i') is satisfied and
  \begin{itemize}
      \item[(ii)] Each $\overline{\rho}_{\mathfrak{m}_m}$ is absolutely irreducible.
  \end{itemize}
  \end{definition}

  \begin{remark}
    For $\mathfrak{m}$ non-Eisenstein, $\overline{\rho}_{\mathfrak{m}}$ may be reducible (when considering the tensor product of two symmetric powers of the same two-dimensional Galois representation for example).
  \end{remark}

  \begin{conjecture}\label{conj:galoisrepresentations}
    For each $m\geq1$ and each maximal ideal $\mathfrak{m}_m$ in ${\bf h}_{\rm ord}(K_{\infty,\infty}^m;\OO)$, there exists a Galois representation $\overline{\rho}_{\mathfrak{m}}$ as in (i').
  \end{conjecture}

  For $F$ totally real or a CM field, results of Scholze \cite[Corollary 5.4.3]{scholze2015} imply the existence of $\overline{\rho}_{\mathfrak{m}_m}$ for all maximal ideals ${\mathfrak{m}_m}$, which also ensures the existence of $\overline{\rho}_{\mathfrak{m}}$. It is expected that the representation $\overline{\rho}_{\mathfrak{m}_m}$ lifts to a representation over the localization ${\bf h}_{\rm ord}(K_{\infty,\infty}^m;\OO)_{\mathfrak{m}_m}$, which is known by Corollary 5.4.4 in loc.\ cit.\ modulo a nilpotent ideal of bounded exponent (see also Theorem 5.13 in \citep{newtonthorne2016}).

  \begin{theorem}\label{thm:localizedcohomology}
    Assume that $F$ is totally real, CM or that Conjecture \ref{conj:galoisrepresentations} holds over $F$. Let $\mathfrak{m}$ be a non-Eisenstein maximal ideal in ${\bf h}_{\rm ord}(K_{\infty,\infty};\OO)$. Assume that the residue field $k$ of $\mathfrak{m}$ embeds into the residue field of $E$. Then for every $E$-rational dominant weight $\lambda$ of $G$, we have an identity
  \begin{equation}
      H^q_{\rm c}(\mathscr X(K_{\alpha',\alpha});\underline{L}_{\lambda,E})_{\mathfrak{m}}=
      H^q(\mathscr X(K_{\alpha',\alpha});\underline{L}_{\lambda,E})_{\mathfrak{m}}
      \label{eq:localizedinnercohomology}
  \end{equation}
  of localizations at $\mathfrak{m}$ and if this space is non-zero, then $q_0\leq q\leq q_0+l_0$. Furthermore, every absolutely irreducible constituent of \eqref{eq:localizedinnercohomology} is cuspidal.
  \end{theorem}

  \begin{proof}
    By \eqref{eq:Xdecomposition} and \eqref{eq:kuennethoverfields}, the first statement is a consequence of Theorem 4.2 in \citep{newtonthorne2016} (c.f.\ the proof of Theorem 6.23 in \citep{kharethorne2016} for a sketch of the argument in \citep{newtonthorne2016}). For regular $\lambda$ we may invoke Proposition 4.2 in \citep{lischwermer2004}, which says that every automorphic representation contributing to \eqref{eq:localizedinnercohomology} is essentially tempered at infinity (mod center), whence by \citep{voganzuckerman1984} the localized cohomology vanishes outside the cuspidal range.

    For general weights $\lambda$ observe that since we are working with a product of general linear groups, the K\"unneth formula reduces us to a single copy of $\GL(n)$. In this situation, we remark that every irreducible constituent of \eqref{eq:localizedinnercohomology} is {\em strongly inner} in the sense of \citep{harderraghuram2017}, and therefore by the classification of the residual spectrum of \citep{moeglinwaldspurger1989} combined with the classification result of Jacquet-Shalika in \citep{jacquetshalika1981,jacquetshalika1981.2} implies that every constituent of \eqref{eq:localizedinnercohomology} is in fact cuspidal. We refer the reader to section 4.3 in Harder-Raghuram \citep{harderraghuram2017} for the details of this argument and also to Proposition 4.1 in \citep{frankeschwermer1998} for an argument along the same lines.
  \end{proof}

  \begin{remark}
    By \citep{wallach1984}, every automorphic representation contributing to \eqref{eq:localizedinnercohomology} is in fact {\em cuspidal}.
  \end{remark}
  \begin{remark}
    The vanishing in degrees $<q_0$ is also implied by \citep{lischwermer2004}.
  \end{remark}

\subsection{The Control Theorem}\label{sec:controltheorem}

  At every place $v\not\in S$, the parabolic Hecke algebra $\mathcal H_\OO^{B_{n+1}(F_v)\times B_n(F_v)}(\alpha',\alpha)$ at $v$ naturally acts on the cohomology
  \begin{equation}
    H^\bullet(\mathscr X(K_{\alpha',\alpha}\cap G^{\rm{der}}(\Adeles^{(\infty)}));\underline{L}_{\lambda,A}).
    \label{eq:derivedcohomology}
  \end{equation}
  The canonical inclusion
  $$
    \mathcal H_\OO(\GL_{n+1}(\OO_v)\times\GL_n(\OO_v),\GL_{n+1}(F_v)\times\GL_n(F_v))\;\to\;\mathcal H_\OO^{B_{n+1}(F_v)\times B_n(F_v)}(\alpha',\alpha)
  $$
  of the spherical Hecke algebra at $v$ of $G$ into the parabolic Hecke algebra induces a canonical action of the spherical Hecke algebra on \eqref{eq:derivedcohomology}. Likewise, $\mathcal H_\OO(\alpha',\alpha)$ acts on \eqref{eq:derivedcohomology} as well. This remains valid for the more general coefficient systems considered below.

  We call a character $\eta:T(\QQ_p)\to E^\times$ {\em locally algebraic} if there is an $E$-rational algebraic character $\lambda:T\to\GL_1$ and a finite order character $\vartheta:T(\QQ_p)\to E^\times$ with
$$
\eta\;=\;\lambda^{w_0}\vartheta.
$$
Furthermore, $\eta$ is {\em dominant} if $\lambda$ is a dominant $E$-rational character of $T$. Write $P_{\lambda^{w_0}\vartheta}\subseteq\Lambda=\OO[[T(\ZZ_p)]]$ for the kernel of the algebra homomorphism $\lambda^{w_0}\vartheta:\Lambda\to E$ induced by $\lambda^{w_0}\vartheta$.

\begin{proposition}\label{prop:regularweightszariskidense}
The set\nomenclature[]{$\mathscr X_{\rm reg,bal}^0$}{set or regular dominant arithmetic points for which $\eta_0$ is admissible}
$$
\mathscr X_{\rm reg,bal}^0\;:=\;\{P_{\lambda^{w_0}}\;\mid\;\lambda\in X_{\CC}(T)\;\text{regular dominant, }\eta_0\;\text{admissible for }\lambda\}
$$
of regular dominant arithmetic points for which $\eta_0$ is admissible is Zariski dense in $\Spec E[[T(\ZZ_p)]]$.
\end{proposition}

\begin{proof}
  Write $\lambda=\lambda_{n+1}\otimes\lambda_n$ with regular dominant weights $\lambda_{n+1}$ and $\lambda_n$ of $\res_{F/\QQ}T_{n+1}$ and $\res_{F/\QQ}T_n$, where $\lambda_m=(\lambda_{m,\tau,i})_{\tau:F\to\CC,1\leq i\leq m}$ with $\lambda_{m,\tau,i}\in\ZZ$ and
  $$
    \lambda_{m,\tau,1}\geq\lambda_{m,\tau,2}\geq\cdots \geq\lambda_{m,\tau,m},
  $$
  regularity meaning that these inequalities are all strict. Then $\eta_0$ is admissible for $\lambda$ if and only if
  $$
    \lambda_{n+1,\tau,1}\geq-\lambda_{n,\tau,n}\geq\lambda_{n+1,\tau,2}\geq
    \cdots\geq-\lambda_{n,\tau,1}\geq\lambda_{n+1,\tau,n+1},
  $$
  for all embeddings $\tau:F\to\CC$. Therefore, considering $T$ as a maximal torus in $\res_{F/\QQ}\GL_{(n+1)n}$, the semigroup of regular dominant $\lambda$ for which $\eta_0$ is admissible is in canonical bijection with a subset of weights of $T$ which are dominant for a suitable choice of Borel in $\res_{F/\QQ}\GL_{(n+1)n}$. In fact, the subset of weights we obtain contains all regular dominant weights of this larger general group. The prime ideals corresponding to the latter set are visibly Zariski dense in $\Spec E[[T(\ZZ_p)]]$.
\end{proof}

For any $\Lambda$-module $\mathcal M,$ set\nomenclature[]{$\mathcal M[\eta]$}{generalized $\eta$-eigenspace in $\mathcal M$}
$$
\mathcal M[\eta]\;:=\;\{m\in\mathcal M\mid \forall x\in T(\ZZ_p):x\cdot m=\eta(x)\cdot m\}.
$$

\begin{theorem}\label{thm:regularcontrol}
  Let $F/\QQ$ be an arbitrary number field. Assume $p\nmid (n+1)n$. Then for any $E$-valued locally algebraic character $\eta=\lambda^{w_0}\vartheta$ of $T(\QQ_p)$ such that $\vartheta$ factors over $T(\ZZ_p/p^{\alpha'}\ZZ_p)$, $\alpha\geq\alpha'\geq0$, $\alpha\geq\alpha_0^K$, and $\lambda$ regular dominant, the canonical map
  $$
  H_{\rm ord}^{q_0}(\mathscr X(K_{\alpha',\alpha});\underline{L}_{\lambda,E/\OO})[\vartheta]
  \;\to\;
  \mathcal H_{\rm ord}^{q_0}(K_{\infty,\infty}; E/\OO)[\lambda^{w_0}\vartheta]
  $$
  has finite kernel and finite cokernel.

  Assume $F/\QQ$ is totally real, CM or that Conjecture \ref{conj:galoisrepresentations} applies over $F$. Let $\mathfrak{m}$ denote a non-Eisenstein maximal ideal in ${\bf h}_{\rm ord}(K_{\infty,\infty};\OO)$. Assume that the residue field $k$ of $\mathfrak{m}$ embeds into the residue field of $E$. Then for every $E$-rational dominant weight $\lambda$ of $G$, the canonical map
  $$
  H_{\rm c,ord}^{q_0}(\mathscr X(K_{\alpha',\alpha});\underline{L}_{\lambda,E/\OO})_{\mathfrak{m}}[\vartheta]
  \;\to\;
  \mathcal H_{\rm c,ord}^{q_0}(K_{\infty,\infty}; E/\OO)_{\mathfrak{m}}[\lambda^{w_0}\vartheta]
  $$
  has finite kernel and finite cokernel.
\end{theorem}

\begin{proof}
  Given the general vanishing results of Li-Schwermer in \citep{lischwermer2004} (for the first statement) and the vanishing statement in Theorem \ref{thm:localizedcohomology} (for the second statement), the proof proceeds along the lines of Hida's proof of Theorem 6.2 in \citep{hida1995,hida1998}, adapted to split $\GL(n)$, and working with $G$ instead of a single copy of the general linear group.

  To be more specific, recall that $U\subseteq G$ denotes the unipotent radical of the standard upper triangular Borel. Put\nomenclature[]{$I_{0,\alpha}^\circ$}{intersection of $I_{0,\alpha}$ with $G^{\rm{der}}$}
  $$
    I_{0,\alpha}^\circ\;:=\;I_{0,\alpha}\cap G^{\rm{der}}(\ZZ_p),
  $$
  and furnish
  $$
    Y_\alpha\;:=\;I_{0,\alpha}^\circ/U(\ZZ_p),
  $$
  with the right action of the semigroup
  $$
    I_{\alpha',\alpha}\Delta_G I_{\alpha',\alpha}\;=\;I_{\alpha',\alpha}^\circ\Delta_G I_{\alpha',\alpha}^\circ
  $$
  defined by Hida (cf.\ section 3 in \citep{hida1995} and p.\ 682 of \citep{hida1998}, taking into account that Hida considers right actions of Hecke operators where we consider left actions, which results in the opposite dominance condition in loc.\ cit.).

  Following Hida, define for $A\in\{\OO,E,E/\OO\},$
  $$
    {\mathcal C}_{\alpha}(A)\;:=\;\{\phi:Y_\alpha\to A\;\mid\;\phi\;\text{continuous}\}\;=\;
    {\rm ind}_{(B\cap G^{\rm{der}})(\ZZ_p)}^{I_{0,\alpha}\cap G^{\rm{der}}(\ZZ_p)}A,
  $$
  which carries an action of $I_{\alpha',\alpha}\Delta_G I_{\alpha',\alpha}$ by right translation.

  Hida modified the action of the semigroup $I_{\alpha,\alpha}\Delta_G I_{\alpha,\alpha}$ on $L_{\lambda,E/\OO}$ by twisting the action of $T(\ZZ_p)$ by the character $\vartheta$ (cf.\ page 684 in \citep{hida1998}). The resulting representation is denoted $L_{\lambda\otimes\vartheta,E/\OO}.$ Put\nomenclature[]{$I_{\alpha',\alpha}^\circ$}{intersection of $I_{\alpha',\alpha}$ with $G^{\rm{der}}$}
  $$
    I_{\alpha',\alpha}^\circ\;:=\;I_{\alpha',\alpha}\cap G^{\rm{der}}(\OO_p)
  $$
  and\nomenclature[]{$K_{\bullet,\bullet}^\circ$}{intersection of $K_{\bullet,\bullet}$ with $G^{\rm{der}}$}
  $$
    K_{\bullet,\bullet}^\circ\;:=\;K_{\bullet,\bullet}\cap G^{\rm{der}}(\Adeles^{(\infty)}).
  $$
  We observe first that Proposition 4.1 and Theorem 4.1 in \citep{hida1995} are valid for $\GL(n)$ both for cohomology with and without compact supports (without the need of localization). These results also extend to $G$ without modification in both cases. Hence Proposition 5.1 of \citep{hida1995} remains valid for $G$ in both cases as well.
    
  Given this, Theorem 5.1 of \citep{hida1995} and its proof remain valid without modification for nearly ordinary cohomology of $G$ with and without compact supports. This shows that there is a canonical isomorphism of $I_{\alpha',\alpha}^\circ\Delta_G I_{\alpha',\alpha}^\circ\times G(\Adeles^S)$-modules
  \begin{equation}
    \iota_{\lambda}:\quad
      \mathcal H^q_{?,\rm ord}(\mathscr X(K_{\infty,\infty}^\circ);\underline{L}_{\lambda\otimes\vartheta,E/\OO})
    \;\cong\;
      H^q_{?,\rm ord}(\mathscr X(K_{0,1}^\circ);\underline{\mathscr C}_1(E/\OO))
    \label{eq:hidasuniversaliso}
  \end{equation}
  for all degrees $q\in\ZZ$, satisfying
  $$
    \iota_\lambda(\lambda^{w_0}(t)\langle t\rangle (-))\;=\;\langle t\rangle\iota_\lambda(-)
  $$
  for any $t\in T(\ZZ_p)$ and $?\in\{-,\rm c\}$ (cf.\ also (6.8) in \citep{hida1998}). This isomorphism is compatible with the isomorphisms constructed in Theorem \ref{thm:independenceofweight}.

  By \citep{lischwermer2004}, regularity of $\lambda$ implies vanishing
  $$
  H^q(\mathscr X(K_{\alpha',\alpha}^\circ);\underline{L}_{\lambda\otimes\vartheta,\CC})\;=\;0,\quad q<q_0,
  $$
  for all $\alpha\geq\alpha'\geq 0,$ $\alpha\geq\alpha_0^K$. In the case of compactly supported cohomology localized at the non-Eisenstein maximal ideal ${\mathfrak{m}}$ we appeal to Theorem \ref{thm:localizedcohomology} to deduce the corresponding vanishing statement in degrees $q<q_0$ for general dominant $\lambda$.

  Therefore, Theorem 5.2 and its proof show with with Lemma 5.1 of \citep{hida1995} that for $\alpha\geq\alpha'$, $\alpha\geq\alpha_0^K$ such that $\vartheta$ factors over $T(\ZZ_p/p^{\alpha'}\ZZ_p)$, \eqref{eq:hidasuniversaliso} induces an canonical map
  $$
    \iota_{\lambda}^\vartheta:\quad 
     H^{q_0}_{\rm ord}(K_{\alpha',\alpha}^\circ;L_{\lambda\otimes\vartheta,E/\OO})
    \;\to\;
      H^{q_0}_{\rm ord}(\mathscr X(K_{0,1}^\circ);\underline{\mathscr C}_1(E/\OO))[\lambda^{w_0}\vartheta]
  $$
  with finite kernel and finite cokernel provided $\lambda$ is regular, see also (6.9) in \citep{hida1998}.

  Again along the same lines, the canonical map
  $$
    \iota_{\lambda}^\vartheta:\quad 
     H^{q_0}_{\rm c,ord}(K_{\alpha',\alpha}^\circ;L_{\lambda\otimes\vartheta,E/\OO})_{\mathfrak{m}}
    \;\to\;
    H^{q_0}_{\rm c,ord}(\mathscr X(K_{0,1}^\circ);\underline{\mathscr C}_1(E/\OO))_{\mathfrak{m}}[\lambda^{w_0}\vartheta]
  $$
  has finite kernel and finite cokernel for general dominant $\lambda$.

  So far all arguments are valid without any assumption on $p$.

  The rest of the argument then proceeds as the proof of Theorem 6.1 in section 6.3 of \citep{hida1998} without modifications.
\end{proof}

\begin{corollary}\label{cor:Hordcontrol}
  Under the hypotheses of Theorem \ref{thm:regularcontrol}, consider the algebra homomorphism
  $$
  \left(\lambda\vartheta:\quad\Lambda\to\OO\right)\;\in\;\Spec(\Lambda).
  $$
  If $1+p^{\alpha'}\OO$ lies in the kernel of $\vartheta,$ there is a canonical map
  $$
    \mathcal H_{\rm c,ord}^{q_0+l_0}(K_{\infty,\infty}; \OO)\otimes_{\Lambda,\lambda\vartheta}\OO
    \to
    H_{\rm c,ord}^{q_0+l_0}(\mathscr X(K_{\alpha',\alpha}); \underline{L}_{\lambda,\OO}^\circ)\otimes_{\OO[T(\ZZ_p/p^{\alpha'}\ZZ_p)],\vartheta}\OO
  $$
  with finite kernel and cokernel whenever $\lambda$ is regular dominant. Pulling back $\lambda\vartheta$ to $\Lambda^\circ$, we obtain for general (not necessarily regular) dominant $\lambda$, a canonical map
  $$
    \mathcal H_{\rm ord}^{q_0+l_0}(K_{\infty,\infty}; \OO)_{\mathfrak{m}}\otimes_{\Lambda^\circ,\lambda\vartheta}\OO
    \to
    H_{\rm ord}^{q_0+l_0}(\mathscr X(K_{\alpha',\alpha}); \underline{L}_{\lambda,\OO}^\circ)_{\mathfrak{m}}\otimes_{\Lambda^\circ,\vartheta}\OO
  $$
  with finite kernel and cokernel.
\end{corollary}

\begin{proof}
  It suffices to observe that by Poincar\'e-Pontryagin duality, the Pontryagin dual of $\mathcal H_{\rm ord}^{q_0}(K_{\infty,\infty}; E/\OO)$ is $\mathcal H_{\rm c,ord}^{q_0+l_0}(K_{\infty,\infty}; \OO),$ and similarly in the second case (the condition of being non-Eisenstein is stable under taking contragredients).
\end{proof}

\begin{corollary}\label{cor:hordcontrol}
  Write $P_{\lambda\vartheta}$ for the kernel of $\lambda\vartheta\in\Spec(\Lambda)$ and $P_\vartheta$ for that of $\vartheta\in\Spec(\OO[T(\ZZ_p/p^{\alpha'}\ZZ_p)])$\nomenclature[]{$P_{\lambda\vartheta}$}{kernel of $\lambda\vartheta$, a prime ideal in $\Lambda$}\nomenclature[]{$P_\vartheta$}{kernel of $\vartheta$, prime ideal in $\OO[T(\ZZ_p/p^{\alpha'}\ZZ_p)]$}. Then we have a canonical isogeny
  $$
  {\bf h}_{\rm c,ord}^{q_0+l_0}(K_{\infty,\infty};\OO)/P_{\lambda\vartheta}{\bf h}_{\rm c,ord}^{q_0+l_0}(K_{\infty,\infty};\OO)\;\to\;
  {\rm h}_{\rm c,ord}^{q_0+l_0}(K_{\alpha',\alpha};\lambda,\OO)/P_{\vartheta}{\rm h}_{\rm c,ord}^{q_0+l_0}(K_{\alpha',\alpha};\lambda,\OO)
  $$
  for regular dominant $\lambda$ and pulling back $\lambda\vartheta$ to $\Lambda^\circ$, we have a canonical isogeny
  $$
  {\bf h}_{\rm ord}^{q_0+l_0}(K_{\infty,\infty};\OO)_{\mathfrak{m}}/P_{\lambda\vartheta}^\circ{\bf h}_{\rm ord}^{q_0+l_0}(K_{\infty,\infty};\OO)_{\mathfrak{m}}\;\to\;
  {\rm h}_{\rm ord}^{q_0+l_0}(K_{\alpha',\alpha};\lambda,\OO)_{\mathfrak{m}}/P_{\vartheta}^\circ{\rm h}_{\rm ord}^{q_0+l_0}(K_{\alpha',\alpha};\lambda,\OO)_{\mathfrak{m}}
  $$
  with $P_{\lambda\vartheta}^\circ=P_{\lambda\vartheta}\cap\Lambda_0$ (likewise for $P_\vartheta^\circ$) for general dominant $\lambda$.
\end{corollary}

\begin{corollary}\label{cor:hordfiniteness}
  The universal nearly ordinary Hecke algebra ${\bf h}_{\rm c,ord}^{q_0+l_0}(K_{\infty,\infty};\OO)$ is finite over $\Lambda$. Likewise, ${\bf h}_{\rm ord}^{q_0+l_0}(K_{\infty,\infty};\OO)_{\mathfrak{m}}$ is finite over $\Lambda^\circ$.
\end{corollary}

\begin{corollary}
  Assume $n>2$ or $F$ not totally real and $p\nmid(n+1)n$. Then $\mathcal H_{\rm ord}^{q_0}(K_{\infty,\infty}; E/\OO)$ is a cotorsion $\Lambda$-module, i.e.\ its Pontryagin dual
  $\mathcal H_{\rm c,ord}^{q_0+l_0}(K_{\infty,\infty}; \OO)$ is a torsion $\Lambda$-module. The same applies in the second case.
\end{corollary}

\begin{proof}
  The claim follows from the existence of non-arithmetic regular dominant weights $\lambda$ as observed on p.\ 690 in \citep{hida1998}.
\end{proof}

The Krull dimension of ${\bf h}^{\rm ord}(K_{\infty,\infty};\OO)$ is closely related to the Leopoldt Conjecture, cf.\ Conjecture 1.1 in \citep{hida1998}, see also \citep{kharethorne2016}.

The Hecke module $\mathcal H_{\rm c,ord}^{q_0+l_0}(K_{\infty,\infty}; \OO)$ is always a faithful ${\bf h}_{\rm c,ord}^{q_0+l_0}(K_{\infty,\infty};\OO)$-module by definition. The freeness of universal nearly ordinary cohomology over the universal Hecke algebra turns out to be related to the Leopoldt Conjecture, cf.\ Theorem 4.9 in \citep{hansenthorne2017}.

In fact, Theorems \ref{thm:localizedcohomology} and \ref{thm:regularcontrol} suggest that the same should apply to ${\bf h}_{\rm ord}^{q_0+l_0}(K_{\infty,\infty};\OO)_{\mathfrak{m}}$.

\section{Cohomological construction of $p$-adic measures}\label{sec:measures}

We begin by recalling the construction of abelian $p$-adic $L$-functions from \citep{januszewski2011,januszewski2015,januszewski2016}.

\subsection{The modular symbol}

Following the formalism from Sections 5.1 and 6.4 in \citep{januszewski2016} we define for any finite ad\`ele $g\in G(\Adeles^{(\infty)})$ and any $\OO$-submodule $L_\OO$ of $L_{\lambda,E}$ the translated lattice\nomenclature[]{$gL_\OO$}{$g$-translation of lattice $L_\OO\subseteq L_{\lambda,E}$}
$$
gL_\OO\;:=\;L_{\lambda,E}\cap g\cdot (L_\OO\otimes_\ZZ\widehat{\ZZ}),
$$
  where the intersection takes place in $L_{\lambda,E\otimes_\QQ\widehat{\Adeles}}$. We have an associated sheaf $\underline{gL}_\OO$ on $\mathscr X(K)$ and as in loc.\ cit.\ we have a canonical morphism\nomenclature[]{$T_g$}{canonical isomorphism $t_g^*\underline{L}_\OO\to\underline{gL}_\OO$}
    $$
    T_g:\quad t_g^*\underline{L}_\OO\to\underline{gL}_\OO,
    $$
    of sheaves on $\mathscr X(gKg^{-1})$. This morphism allows us to define a normalized pull back operator $t_g^\lambda$, sending sections of the sheaf $\underline{L}_\OO$ over an open $U\subseteq\mathscr X(K)$ to sections of $\underline{gL}_\OO$ over $Ug^{-1}\subseteq\mathscr X(gKg^{-1})$.

    We remark that for $g_1,g_2\in G(\Adeles^{(\infty)})$ there is an identity $(g_1g_2)L_\OO=g_1(g_2L_\OO)$, and the construction of the translated lattice $gL_\OO$ is functorial in $L_\OO$. Furthermore, the translated lattice always comes with a canonical map $gL_\OO\to L_{\lambda,E}$ and we have $gL_{\lambda,E}=L_{\lambda,E}$ for all $g\in G(\Adeles^{(\infty)})$.

    As before, let $\eta_j:L_{\lambda,E}\to E_{(j)}$ denote a non-zero $H$-equivariant functional. We fix once and for all an isomorphism $E_{(j)}\cong E$, which induces isomorphisms $A_{(j)}\cong A$. This allows us to identify $A_{(j)}$ and $A$ in the sequel. From this identification and the non-vanishing statement \eqref{eq:etajg0v0nonvanishing} in Proposition \ref{prop:generalmaninlemma}, we deduce by restriction a $p$-adically normalized $p$-adically optimal functional\nomenclature[]{$\eta_{j,A}$}{normalization of $\eta_j$}
$$
\eta_{j,A}:\quad L_{\lambda,A}^{1,0}\to A_{(j)},
$$
for $A=\OO$ and $A=E,$ given by
$$
v\;\mapsto\;\frac{\eta_{j}(v)}{\eta_{j}(g_0v_0)}.
$$
Then $\eta_{j,A}$ is independent of the choice of $\eta_j$ and also independent of the identification $A_{(j)}=A$.

The codomain of $\eta_{j,A}$ gives rise to a sheaf $\underline{A}_{(j)}$ on the locally symmetric space\nomenclature[]{$\mathscr Y(L)$}{locally symmetric space for $H$ of level $L$}
$$
{\mathscr Y}(L)\;:=\;H(\QQ)\backslash H(\Adeles)/L\cdot {K_\infty'}^0
$$
where $L\subseteq H(\Adeles^{(\infty)})$ is any compact open subgroup and\nomenclature[]{$K_\infty'$}{standard maximal compact in $H(\RR)$, equals the intersection of $H(\RR)$ with $\widetilde{K}_\infty$}
$$
K_\infty'\::=\;H(\RR)\cap \widetilde{K}_\infty
$$
happens to be a standard maximal compact subgroup. The numerical coincidence we exploit is
$$
  \dim{\mathscr Y}\;:=\;\dim{\mathscr Y}(L)\;=\;q_0,
$$
which over $\QQ$ was first observed in \citep{kazhdanmazurschmidt2000}, and over general base fields $F$ in \citep{januszewski2011}. Then
$$
  \dim{\mathscr X}-\dim{\mathscr Y}\;=\;q_0+l_0
$$
is the top degree for $G$.
  
By strong approximation for $\SL(n)$, the connected components of $\mathscr Y(L)$ are parametrized by elements in the class group\nomenclature[]{$C(L)$}{class group of level $\det L$}
$$
C(L)\;:=\;F^\times\backslash\Adeles_F^\times/\det(L) F_\infty^0.
$$
We write $\mathscr Y(L)[x]$ for the component mapping to $x\in C(\det(L))$ under the determinant.

We fix once and for all a system of fundamental classes as in \cite[Section 5.3]{januszewski2016}. Then, for each $L$ neat and each $x\in C(L)$, we have compatible isomorphisms
$$
\int_{\mathscr Y(L)[x]}:\quad H_{\rm c}^{\dim\mathscr Y}(\mathscr Y(L)_x;\underline{A}_{(j)})\;\to\;
A_{(j)}.
$$
Whenever $L\subseteq K$, the inclusion $H\to G$ induces a proper map\nomenclature[]{$i$}{proper map $\mathscr Y(L)\to\mathscr X(K)$}
$$
i:\quad\mathscr Y(L)\to\mathscr X(K).
$$

We need the following generalization of Proposition 3.4 in \citep{schmidt2001}.

\begin{proposition}\label{prop:leveladdition}
For any $\beta\geq\alpha>0$, the compact open subgroup\nomenclature[]{$\mathfrak I^{n}_{\beta}$}{intersection of $g_\beta$-conjugate of $I_\alpha$ with $H$}
  $$
  \mathfrak I^{n}_{\beta}
  \;:=\;H(\QQ_p)\cap g_\beta I_\alpha g_\beta^{-1}
  $$
  of $H(\ZZ_p)$ is independent of $\alpha$. It satisfies
  $$
  (H(\ZZ_p):\mathfrak I^{n}_{\beta})\;=\;
  \prod_{v\mid p}\prod_{\mu=1}^{n}\left(1-q_v^{-\mu}\right)^{-1}
  \cdot
  p^{\beta\frac{(n+2)(n+1)n+(n+1)n(n-1)}{6}},
  $$
  and
  $$
  \det \mathfrak I^{n}_{\beta}\;=\;1+p^{\beta}\OO_p.
  $$
\end{proposition}

\begin{proof}
  Set
  $$
    I^{n,1}_{\alpha,p}\;:=\;\{r\in\GL_{n}(\OO_p)\;\mid\;r\pmod{p^\alpha}\in U_{n}(\OO_p)\}.
  $$
  We denote the opposite group constructed with $U_{n}^-$ by $I^{n,1-}_{\alpha,p}$ and write $\sim$ for an identity of subgroups of $H(\ZZ_p)$ up to conjugation. Then
  \begin{align*}
    g_\beta I_\alpha g_\beta^{-1}\cap H(\QQ_p)
    &\;=\;
    h_nt_p^{\beta} I^{n+1,1}_{\alpha,p}t_p^{-\beta} h_n^{-1}\cap t_p^{\beta}I^{n,1}_{\alpha,p}t_p^{-\beta}\\
    &\;\sim\;
    w_nh_nt_p^{\beta} I^{n+1,1}_{\alpha,p}t_p^{-\beta} (w_nh_n)^{-1}\cap w_nt_p^{\beta}I^{n,1}_{\alpha,p}(w_nt)^{-\beta}\\
    &\;=\;
    w_nh_nt_p^{\beta} I^{n+1,1}_{\alpha,p}t_p^{-\beta} (w_nh_n)^{-1}\cap w_nt_p^{\beta}w_nI^{n,1-}_{\alpha,p}(w_ntw_n)^{-\beta}\\
    &\;=\;
    w_nh_nt_p^{\beta} I^{n+1,1}_{\alpha,p}t_p^{-\beta} (w_nh_n)^{-1}\cap t_p^{-\beta}I^{n,1-}_{\alpha,p}t_p^{\beta}.
  \end{align*}
  Now
  $$
  w_nh_n\;=\;
  \begin{pmatrix}
  &&&1\\
  &{\bf1}_n&&\vdots\\
  &&&\vdots\\
  0&\hdots&0&1
  \end{pmatrix},
  $$
  and
  $$
  (w_nh_n)^{-1}\;=\;
  \begin{pmatrix}
  &&&-1\\
  &{\bf1}_n&&\vdots\\
  &&&-1\\
  0&\hdots&0&1
  \end{pmatrix}.
  $$
  For $r\in I^{n,1}_{\alpha,p}$ we have
  $$
  \left(w_nh_nt_p^{\beta}rt_p^{-\beta}(w_nh_n)^{-1}\right)_{ij}
  \;=\\
  $$
  $$
  \quad\quad
  \begin{cases}
    p^{\beta(j-i)}r_{ij}
    + p^{\beta(n+1-i)}r_{in+1}
      &1\leq i,j\leq n,\\
    p^{\beta(j-(n+1))}r_{n+1j}
      &i=n+1,1\leq j\leq n,\\
    p^{\beta(n+1-i)}r_{in+1}
    + r_{n+1n+1}
    -\sum_{j=1}^n\left(p^{\beta(j-(n+1))}r_{n+1j}+p^{\beta(j-i)}r_{ij}\right)
      &1\leq i\leq n,j=n+1,\\
    r_{n+1n+1}
    -\sum_{j=1}^np^{\beta(j-(n+1))}r_{n+1j}
    &i=j=n+1.
  \end{cases}
  $$
  The condition that this be an element of $H(\QQ_p)$ is equivalent to the conditions
  \begin{align}
    r_{n+1j}&=0, &&1\leq j\leq n,\nonumber\\
    r_{n+1n+1}
    &=1,\nonumber\\
    p^{\beta(n+1-i)}r_{in+1}+1
    -\sum_{j=1}^np^{\beta(j-i)}r_{ij}&=0, &&1\leq i\leq n.\label{eq:lastcolumn}
  \end{align}
  Therefore, $w_nh_nt_p^{\beta}rt_p^{-\beta}(w_nh_n)^{-1}$ lies in $t_p^{-\beta}I^{n,1-}_{\alpha,p}t_p^{\beta}$ if and only if \eqref{eq:lastcolumn} is satisfied and
  \begin{align}
    r_{ij}
    + p^{\beta(n+1-j)}r_{in+1}
    &\in p^{\alpha+\beta(2i-2j)}\OO, &&1\leq i<j\leq n,\label{eq:upper}\\
    r_{ii} + p^{\beta(n+1-i)}r_{in+1}
    &\in 1+p^{\alpha}\OO, &&1\leq i\leq n,\label{eq:diag}\\
    r_{ij}
    + p^{\beta(n+1-j)}r_{in+1}&\in p^{\beta(2i-2j)}\OO, &&1\leq j<i\leq n.\label{eq:bottom}
  \end{align}
  Conditions \eqref{eq:upper} and \eqref{eq:diag} are automatic because $\beta\geq\alpha$.

  Condition \eqref{eq:lastcolumn} is equivalent to
  $$
    r_{ii}\;=\;1+\sum_{j=1}^{i-1}p^{\beta(j-i)}r_{ij}
    +\sum_{j=i+1}^np^{\beta(j-i)}r_{ij}+p^{\beta(n+1-i)}r_{in+1},
  $$
  which in turn is equivalent to
  \begin{equation}
    r_{ii}\;=\;1+\sum_{j=1}^{i-1}p^{\beta(j-i)}(r_{ij}+p^{\beta(n+1-j)}r_{in+1})
    +\sum_{j=i+1}^np^{\beta(j-i)}r_{ij}-(i-2)\cdot p^{\beta(n+1-i)}r_{in+1},
  \label{eq:lastcolumn2}
  \end{equation}
  for $1\leq i\leq n$. The last two summands in \eqref{eq:lastcolumn2} lie in $p^{\beta}\OO$, and by \eqref{eq:bottom}, the summands of the frist sum on the right hand side lies in $p^{\beta(i-j)}\OO\subseteq p^\beta\OO$. This readily implies
  \begin{equation}
    r_{ii}\;\in\;1+p^{\beta}\OO,\quad 1\leq i\leq n.
  \label{eq:diagonalr}
  \end{equation}
  Reversing this argument shows that the diagonal  $(r_{ii})_{1\leq i\leq n}$ may assume any value in $(1+p^{\beta}\OO)^n$. With our previous computation, this shows
  \begin{align*}
    \det\left(w_nh_nt_p^{\beta}rt_p^{-\beta}(w_nh_n)^{-1}\right)_{1\leq i,j\leq n}
    \;=\;&
    \det\left(r_{ij} + p^{\beta(n+1-j)}r_{in+1}\right)_{1\leq i,j\leq n}\\
    \equiv\;&\prod_{i=1}^nr_{ii}\quad\pmod{p^{\beta}\OO}\\
    \equiv\;&1\quad\quad\quad\pmod{p^{\beta}\OO},
  \end{align*}
  and the determinant maps $\mathfrak I^{n}_{\beta}$ surjectively onto $1+p^{\beta}\OO$ as claimed.

  The same computation shows that
  $$
    w_nh_nt_p^{\beta} I^{n+1,1}_{\alpha,p}t_p^{-\beta} (w_nh_n)^{-1}\cap t_p^{-\beta}I^{n,1-}_{\alpha,p}t_p^{\beta}
    \;=\;
    I^{n}_{\beta,p} \cap t_p^{-\beta}I^{n,-}_{\beta,p}t_p^{\beta},
  $$
  which consists of matrices $s\in\GL_n(\OO)$ with entries
  $$
    s_{ij}\;\in\;
    \begin{cases}
      p^{\beta(j-i)}\OO,&i<j,\\
      1+p^{\beta}\OO,&i=j,\\
      p^{\beta(i-j)}\OO,&i>j.
    \end{cases}
  $$
  With this explicit description of the intersection the computation of the index is straightforward.  
\end{proof}

Put\nomenclature[]{$C(p^\beta)$}{ray class group of level $p^\beta$}
$$
C(p^\beta)\;:=\;F^\times\backslash{}\Adeles_F^\times/(1+p^\beta\OO)\cdot\det\left(K^{(p)}\cap H(\Adeles^{(p\infty)})\right)\;=\;
C(g_\beta K_{\alpha',\alpha} g_\beta^{-1}\cap H(\Adeles^{(\infty)})),
$$
where the second identity is a consequence of Proposition \ref{prop:leveladdition}.

For any $\beta>0$ and $x\in C(p^\beta)$ we consider the modular symbol\nomenclature[]{$\mathscr P_{A,x,\beta}^{\lambda,j}$}{topological modular symbol}
$$
\mathscr P_{A,x,\beta}^{\lambda,j}:\quad
H_{\rm c,ord}^{\dim\mathscr Y}({\mathscr X}(K_{\alpha',\alpha}); \underline{L_{\lambda,A}})
\to A_{(j)},
$$
explicitly defined as
$$
\phi\;\mapsto\;
\int_{\mathscr Y(g_\beta K_{\alpha',\alpha} g_\beta^{-1}\cap H(\Adeles^{(\infty)}))[x]}
\eta_{j,A} i^*\left[(-\lambda^{w_0})(t_p^\beta)\cdot t^{\lambda}_{g_{\beta}}\right](U_p^{-\beta}\phi).
$$
By \eqref{eq:integrallattice}, we know that $L_{\lambda,A}^{x,\beta}\subseteq L_{\lambda,A}^{1,0}$ and hence
\begin{equation}
  (-\lambda^{w_0})(t_p^\beta)\cdot t^{\lambda}_{g_{\beta}}(U_p^{-\beta}\phi)\;\in\; L_{\lambda,A}^{x,\beta}\;\subseteq\;L_{\lambda,A}^{1,0},
  \label{eq:latticecontainment}
\end{equation}
whence $\mathscr P_{A,x,\beta}^{\lambda,j}$ is indeed well defined.

Then for any cohomology class $\phi$ as above, we obtain an element\nomenclature[]{$\mu_{A,\beta}^{\lambda,j}(\phi)$}{$A$-valued distribution on $C(p^\beta)$}\nomenclature[]{$A_{(j)}[C(p^\beta)]$}{the module $A_{(j)}\otimes_\OO \OO[C(p^\beta)]$}
$$
\mu_{A,\beta}^{\lambda,j}(\phi)\;:=\;
\sum_{x\in C(p^\beta)}\mathscr P_{A,x,\beta}^{\lambda,j}(\phi)\cdot x
\;\in\;A_{(j)}[C(p^\beta)]\;:=\;A_{(j)}\otimes_\OO \OO[C(p^\beta)].
$$
Here the right hand side denotes the tensor product of $A_{(j)}$ with the group ring of the finite ray class group $C(p^\beta)$ over $A$.

\subsection{The distribution relation}

To establish the distribution relation for $\mu_{A,\beta}^{\lambda,j}$ we follow the argument in section 6.6 of \citep{januszewski2015}.

The following generalizes Lemma 6.5 in \citep{januszewski2015} and Lemma 6.1 in \citep{januszewski2016}.

\begin{lemma}\label{lem:distributionmatrices}
  Let $u\in U(\OO_p)$. Then for every $\beta>0$:
  \begin{itemize}
    \item[(i)]
      There exists $k_{u}=(k_{u}',k_{u}'')\in I_{\alpha,\alpha}$ satisfying
      \begin{equation}
        h t_p^\beta\cdot u t_p\;=\;\Delta(k_{u}'')^{-1}\cdot h t_p^{\beta+1} \cdot k_u.
        \label{eq:distributionmatrixrelation}
      \end{equation}
    \item[(ii)]
      For every $k_u=(k_{u}',k_{u}'')\in I_{\alpha,\alpha}$ satisfying \eqref{eq:distributionmatrixrelation} the residue class of the determinant
      $$
        \det k_u'\;\equiv\;\det k_u''\pmod{p^{\beta+1}}
      $$
      is uniquely determined by $u\in U(\OO_p)/t_pU(\OO_p)t_p^{-1}$ and lies in $1+p^\beta\OO_p$.
    \item[(iii)]
      The map
      $$
        U(\OO_p)/t_pU(\OO_p)t_p^{-1}\;\to\;(1+p^\beta\OO_p)/(1+p^{\beta+1}\OO_p),
      $$
      $$
        u\;\mapsto\;\det k_u',
      $$
      is a surjective group homomorphism.
  \end{itemize}
\end{lemma}

\begin{proof}
  The proof proceeds as the proof of Lemma 6.5 in \citep{januszewski2015} with the following additional observations:
  \begin{align*}
    h t_p^\beta\cdot u t_p
    \;=\;& (h_nj_n(t_{-1}),{\bf1}_n)\cdot t_p^{\beta} u t_p\\
    \;=\;& (h_n,{\bf1}_n)\cdot t_p^{\beta}(j_n(t_{-1}),{\bf1}_n)u t_p\\
    \;=\;& (h_n,{\bf1}_n)\cdot t_p^{\beta}u^{(j_n(t_{-1}),{\bf1}_n)} t_p \cdot (j_n(t_{-1}),{\bf1}_n),
  \end{align*}
  and likewise
  $$
    h t_p^{\beta+1}\;=\; (h_n,{\bf1}_n)\cdot t_p^{\beta+1}\cdot (j_n(t_{-1}),{\bf1}_n),
  $$
  Therefore, the statement reduces to the same statement with $(h_n,{\bf1}_n)$ replacing $h$, which is treated in loc.\ cit.\ for the compact open $I_{0,\alpha}$. Hence it suffices to remark that the elements $k_{u,w}$ and $k_{u,w}'$ constructed in said proof lie in $I^{n+1}_{\alpha,\alpha}$ and $I^n_{\alpha,\alpha}$ respectively.
\end{proof}

For any $\beta\geq\beta'>0$ the a canonical projection
$$
C(p^\beta)\;\to\;C(p^{\beta'})
$$
induces an $\OO$-linear epimorphism\nomenclature[]{$\res_{\beta'}^\beta$}{restriction map $A_{(j)}[C(p^\beta)]\to A_{(j)}[C(p^{\beta'})]$}
$$
\res_{\beta'}^{\beta}:\quad
A_{(j)}[C(p^\beta)]\to
A_{(j)}[C(p^{\beta'})].
$$

\begin{proposition}\label{prop:distribution}
  For any cohomology class $\phi$ and any $\beta\geq\beta'>0$ we have the distribution relation
  $$
    \res_{\beta}^{\beta'}
    \left(\mu_{A,\beta'}^{\lambda,j}(\phi)\right)\;=\;
    \mu_{A,\beta}^{\lambda,j}(\phi).
  $$
\end{proposition}

\begin{proof}
  It suffices to treat the case $\beta=\beta'+1$. In this case, the proof proceeds as the proof of Theorem 6.1 in \citep{januszewski2016}, taking into account that equation (18) in loc.\ cit.\ remains valid by Proposition \ref{prop:leveladdition}, and Lemma \ref{lem:distributionmatrices} and replaces Lemma 6.1 there.
\end{proof}

By Proposition \ref{prop:distribution}, we have a projective system $(\mu_{A,\beta}^{\lambda,j}(\phi))_{\beta}$. Put\nomenclature[]{$C_F(p^\infty)$}{ray class group of level $p^\infty$ of $F$}
$$
C_F(p^{\infty})\;:=\;\varprojlim_\beta C_F(p^\beta)
$$
and\nomenclature[]{$\mu^{\lambda,j}_A(\phi)$}{$A$-valued distribution on $C(p^\infty)$}
$$
\mu^{\lambda,j}_A(\phi)\;:=\;
\varprojlim\limits_{\beta}
\mu_{A,\beta}^{\lambda,j}(\phi).
$$
Thus we obtain an $\OO$-linear map\nomenclature[]{$\mu^{\lambda,j}_A$}{torsion measure as a map $H_{\rm c,ord}^{\dim\mathscr Y}({\mathscr X}(K_{\alpha',\alpha}); \underline{L}_{\lambda,A})\to
A_{(j)}[[C_F(p^\infty)]]$}
$$
\mu^{\lambda,j}_A:\quad
H_{\rm c,ord}^{\dim\mathscr Y}({\mathscr X}(K_{\alpha',\alpha}); \underline{L}_{\lambda,A})\to
A_{(j)}[[C_F(p^\infty)]].
$$

\subsection{$p$-adic character varieties}

Recall that
$$
C_F(p^\infty)\;=\;
\varprojlim_\beta F^\times\backslash\Adeles_F^\times/F_\infty^0(\widehat{\OO}_F^{(p)})(1+p^\beta F_p).
$$
We have
$$
C_F(p^\infty)\;=\;\Delta\times\ZZ_p^{r_F},
$$
for a finite group $\Delta$ and an integers $r_F>0$.

The norm map $N_{F/\QQ}:F\to\QQ$ induces a morphism\nomenclature[]{$N_{F/\QQ}$}{norm in $F/\QQ$, extended to $C_F(p^\infty)\to C_\QQ(p^\infty)$}
$$
N_{F/\QQ}:\quad C_F(p^\infty)\to C_\QQ(p^\infty)=\ZZ_p^\times
$$
with image of finite index. The decomposition\nomenclature[]{$\mu_{p-1}$}{$p-1$-th roots of unity in $\ZZ_p$}
\begin{equation}
  C_\QQ(p^\infty)\;=\;\mu_{p-1}\times (1+p\ZZ_p)
  \label{eq:Zpdecomposition}
\end{equation}
gives rise to two distinguished characters of $C_F(p^\infty)$ as follows. Write $\pi_i,$ $i\in\{1,2\},$ for the projection onto the $i$-th factor of the right hand side in \eqref{eq:Zpdecomposition}. Set\nomenclature[]{$\omega_F$}{relative Teichm\"uller character}
$$
\omega_F:=\pi_1\circ N_{F/\QQ}:\quad C_F(p^\infty)\to\QQ[\mu_{p-1}],
$$
where $\mu_{p-1}$ denotes the group of $(p-1)$-st roots of unity and the group $\mu_2=\{\pm1\}$ for $p=2$. Put\nomenclature[]{$\langle-\rangle_F$}{torsion free component of cyclotomic character}
$$
\langle-\rangle_F:=\pi_2\circ N_{F/\QQ}:\quad C_F(p^\infty)\to\CC_p^\times,
$$
where $\CC_p$ denotes the completion of an algebraic closure of $\QQ_p$. Here we implicitly embedded $1+p\ZZ_p$ into $\CC_p^\times$ canonically.

Let\nomenclature[]{$\mathscr X_F$}{rigid analytic space of $p$-adic characters on $C_F(p^\infty)$}
$$
\mathscr X_F\;:=\;\Hom_{\rm cts}(C_F(p^\infty),\CC_p^\times)
$$
denote the rigid analytic variety of continuous $p$-adic characters of $C_F(p^\infty)$.

The space $\mathscr X_F$ is an equidimensional rigid analytic variety of dimension $r_F$ with $\#\Delta$ irreducible components:
$$
\mathscr X_F\;=\;\bigsqcup_{\chi_0\in\widehat{\Delta}}
\chi_0\cdot (1+{\mathfrak{m}}_{\CC_p})^{r_F}
$$
Inside this space, we have the Zariski dense set\nomenclature[]{$\mathscr X_F^0$}{finite order characters in $\mathscr X_F$}
$$
\mathscr X_F^0\;:=\;\{\chi\in \mathscr X_F\mid\chi\;\text{of finite order}\}
$$
of finite order characters.

Since $r_\QQ=1,$ there is canonical projection
$$
C_\QQ(p^\infty)\;=\;\Delta\times\ZZ_p\to\ZZ_p\;:=\;C_\QQ^{\rm cyc}(p^\infty).
$$
Correspondingly, the {\em cyclotomic line} is defined as\nomenclature[]{$\mathscr X_\QQ^{\rm cyc}$}{cyclotomic line over $\QQ$}
$$
\mathscr X_\QQ^{\rm cyc}\;:=\;
\Hom_{\rm cts}\left(C_\QQ^{\rm cyc}(p^\infty),\CC_p^\times\right)
\;\subseteq\;\mathscr X_\QQ.
$$
Finite order characters of $C_\QQ^{\rm cyc}(p^\infty)$ are precisely the Dirichlet characters of $p$-power order which are unramified outside $p\infty$. These are Zariski dense in $\mathscr X_\QQ^{\rm cyc}$.

The norm induces a commutative diagram
$$
\begin{CD}
\mathscr X_\QQ@>N_{F/\QQ}^*>>\mathscr X_F\\
@AAA @AAA\\
\mathscr X_\QQ^0@>N_{F/\QQ}^*>>\mathscr X_F^0
\end{CD}
$$
We call the image\nomenclature[]{$\mathscr X^{\rm cyc}_F$}{cyclotomic line over $F$}
$$
\mathscr X^{\rm cyc}_F:=N_{F/\QQ}^*(\mathscr X^{\rm cyc}_\QQ)\;\subseteq\;\mathscr X_F
$$
the cyclotomic line over $F$. It has a canonical (topological) generator $\langle\cdot\rangle_F$, the cyclotomic character of $F$. All other characters in the cyclotomic line are of the form $\langle\cdot\rangle_F^s$ for some $s\in\ZZ_p$.

Remark that the set
$$
N_{F/\QQ}^*(\mathscr X_\QQ^{\rm cyc,0})\;\subseteq\;\mathscr X_F^0
$$
of norm-inflated Dirichlet characters of $p$-power order is Zariski dense in $\mathscr X_F^{\rm cyc}$.

\subsection{$p$-adic Tate twists}

Recall that we identified $A_{(j)}$ and $A$. This allows us to identify the modules $A_{(j)}$ for varying $j$.

\begin{theorem}\label{thm:tatetwists}
  Assume that two non-zero $H$-linear functionals
  $$
  \eta_{j_i}:\quad L_{\lambda,E}\to E_{(j_i)},\quad i\in\{1,2\},
  $$
  are given. Then we have for every
  $$
  \phi\in H_{\rm c,ord}^{\dim\mathscr Y}({\mathscr X}(K_{\alpha',\alpha}); \underline{L}_{\lambda,\OO})
  $$
  an identity of measures
  $$
  \omega_F^{j_2}(x)\langle x\rangle_F^{j_2}\mu^{\lambda,j_1}_{\OO}(\phi)(x)\;=\;
  \omega_F^{j_1}(x)\langle x\rangle_F^{j_1}d\mu^{\lambda,j_2}_{\OO}(\phi)(x),
  $$
  on $C_F(p^\infty)$.
\end{theorem}

\begin{proof}
  By construction, we have for all $\beta>0$ and all $x\in C_F(p^\beta)$,
  $$
  \mu^{\lambda,j_i}_{\OO,\beta}(\phi)(x)\;=\;\eta_{j_1,\OO}(v)\;=\;\frac{\eta_{j_i}(v)}{\eta_{j_i}(g_0v_0)}\;\in\;\OO_{(j_i)}
  $$
  for some vector $v\in L_{\lambda,\OO}^{x,\beta}$ independent of $i\in\{1,2\}$. Two applications of the congruence \eqref{eq:manincongruence} in Proposition \ref{prop:generalmaninlemma} show:
\begin{align*}
  N_{F/\QQ}(x)^{j_2}\cdot\mu^{\lambda,j_1}_{\OO,\beta}(\phi)(x)
  \;=\;&N_{F/\QQ}(x)^{j_2}\cdot \eta_{j_1,\OO}(v)\\
  \equiv\;&N_{F/\QQ}(x)^{j_1+j_2}\cdot\Omega_p^{x,\beta}\quad\quad\quad\;\pmod{p^\beta\OO}\\
  \equiv\;&N_{F/\QQ}(x)^{j_1}\cdot \eta_{j_2,\OO}(v)\,\quad\quad\quad\pmod{p^\beta\OO}\\
  =\;&N_{F/\QQ}(x)^{j_1}\cdot\mu^{\lambda,j_2}_{\OO,\beta}(\phi)(x).
\end{align*}
This proves the claim.
\end{proof}

\subsection{Modular symbols in $p$-adic families}\label{sec:modularsymbolsinfamilies}

We restrict our attention to the nearly ordinary case. For any $\alpha\geq 0,$ the normalized projections $\eta_{j,K}$ and $\eta_{j,\OO}$ induce a canonical projection
$$
\eta_{j,p^{-\alpha}\OO/\OO}:\quad L_{\lambda,p^{-\alpha}\OO/\OO}\to p^{-\alpha}\OO/\OO.
$$

Given $\beta\geq\alpha\geq\alpha_0^K$ and $x\in C(p^\beta)$, relation \eqref{eq:latticecontainment} shows that we may consider the modular symbol $\mathscr P_{A,x,\beta}^{\lambda,j}$ for the case $A=p^{-\alpha}\OO/\OO$, i.e.\ we have
$$
\mathscr P_{\alpha,x,\beta}^{\lambda,j}:\quad
H_{\rm c,ord}^{\dim\mathscr Y}({\mathscr X}(K_{\alpha,\alpha}); \underline{L}_{\lambda,p^{-\alpha}\OO/\OO})
\to (p^{-\alpha}\OO/\OO)_{(j)},
$$
given by
$$
\phi\;\mapsto\;
\int_{\mathscr Y(g_\beta K_{\alpha,\alpha} g_\beta^{-1}\cap H(\Adeles^{(\infty)})[x]}
\eta_{j,p^{-\alpha}\OO/\OO} i^*\left[(-\lambda^{w_0})(t_p^\beta)\cdot t^{\lambda}_{g_{\beta}}\right](U_p^{-\beta}\phi).
$$
We emphasize that here both the $\OO$-module $p^{-\alpha}\OO/\OO$ and the level $K_{\alpha,\alpha}$ depend on $\alpha$.

As before, we obtain elements\nomenclature[]{$\mu_{\alpha,\beta}^{\lambda,j}$}{torsion valued distribution on $C(p^\beta)$}
$$
\mu_{\alpha,\beta}^{\lambda,j}(\phi)\;:=\;\sum_{x\in C(p^\beta)}\mathscr P_{\alpha,x,\beta}^{\lambda,j}(\phi)\cdot x\;\in\;(p^{-\alpha}\OO/\OO)_{(j)}[C(p^\beta)],
$$
which in light of Proposition \ref{prop:distribution} satisfy for any $\beta\geq\beta'\geq\alpha>\alpha_0^K$ the distribution relation,
\begin{equation}
\res_{\beta'}^{\beta}
\left(
\mu_{\alpha,\beta}^{\lambda,j}(\phi)
\right)\;=\;
\mu_{\alpha,\beta'}^{\lambda,j}(\phi).
\label{eq:alphabetadistributionrelation}
\end{equation}
Therefore, we obtain for each $\alpha>\alpha_0^K$ an $\OO$-linear map
$$
\mu^{\lambda,j}_{\alpha}:\quad
H_{\rm c,ord}^{\dim\mathscr Y}({\mathscr X}(K_{\alpha',\alpha}); \underline{L}_{\lambda,p^{-\alpha}\OO/\OO})\to
(p^{-\alpha}\OO/\OO)_{(j)}[[C(p^\infty)]].
$$
By construction, we have for all $\alpha\geq\alpha'\geq\alpha_0^K$, a commutative square
$$
\begin{CD}
  H_{\rm c,ord}^{\dim\mathscr Y}({\mathscr X}(K_{\alpha',\alpha'}); \underline{L}_{\lambda,p^{-\alpha'}\OO/\OO})
  @>{\mu^{\lambda,j}_{\alpha'}}>>
  (p^{-\alpha'}\OO/\OO)_{(j)}[[C(p^\infty)]]\\
  @VVV @VVV\\
  H_{\rm c,ord}^{\dim\mathscr Y}({\mathscr X}(K_{\alpha,\alpha}); \underline{L}_{\lambda,p^{-\alpha}\OO/\OO})
  @>{\mu^{\lambda,j}_{\alpha}}>>
(p^{-\alpha}\OO/\OO)_{(j)}[[C(p^\infty)]].
\end{CD}
$$
This allows us to pass to the direct limit to obtain a map\nomenclature[]{$\mu^{\lambda,j}$}{direct limit of $\mu^{\lambda,j}_\alpha$ for $\alpha\to\infty$}
$$
\mu^{\lambda,j}:\quad
\mathcal H_{\rm c,ord}^{\dim\mathscr Y}(K_{\infty,\infty};\underline{L}_{\lambda,E/\OO})\;\to\;
(E/\OO)_{(j)}[[C(p^\infty)]].
$$
\begin{theorem}[Independence of weight]\label{thm:muindependenceofweight}
  For any $\lambda$ for which $\eta_0$ is admissible we have a commuting square
  $$
\begin{CD}
  \mathcal H_{\rm c,ord}^{\dim\mathscr Y}(K_{\infty,\infty}; \underline{L}_{\lambda,K/\OO})
  @>{\mu^{\lambda,0}}>>
  (K/\OO)_{(0)}[[C(p^\infty)]]\\
  @V{\pi_\lambda}VV @|\\
  \mathcal H_{\rm c,ord}^{\dim\mathscr Y}(K_{\infty,\infty}; \underline{K/\OO})
  @>{\mu^{0,0}}>>
(K/\OO)_{(0)}[[C(p^\infty)]],
\end{CD}
$$
where the map $\pi_\lambda$ is the isomorphism from Corollary \ref{cor:independenceofweight}.
\end{theorem}

\begin{proof}
  By construction of $\mu^{\lambda,0}$ and $\mu^{0,0}$ as the inductive limits of the maps
  $\mu^{\lambda,0}_\alpha$ and $\mu^{0,0}_\alpha$, the claim is equivalent to the commutativity of the diagram
$$
\begin{CD}
  H_{\rm c,ord}^{\dim\mathscr Y}({\mathscr X}(K_{\alpha,\alpha}); \underline{L}_{\lambda,p^{-\alpha}\OO/\OO})
  @>{\mu^{\lambda,0}_{\alpha}}>>
  (p^{-\alpha}\OO/\OO)_{(0)}[[C(p^\infty)]]\\
  @V{\pi_\lambda}VV @VVV\\
  H_{\rm c,ord}^{\dim\mathscr Y}({\mathscr X}(K_{\alpha,\alpha}); \underline{p^{-\alpha}\OO/\OO})
  @>{\mu^{0,0}_{\alpha}}>>
(p^{-\alpha}\OO/\OO)_{(0)}[[C(p^\infty)]]
\end{CD}
$$
for each $\alpha\geq\alpha_0^K$.

On the one hand, this reduces us to the commutativity of
$$
\begin{CD}
  H_{\rm c,ord}^{\dim\mathscr Y}({\mathscr X}(K_{\alpha,\alpha}); \underline{L}_{\lambda,p^{-\alpha}\OO/\OO})
  @>{\mu^{\lambda,0}_{\alpha,\beta}}>>
  (p^{-\alpha}\OO/\OO)_{(0)}[C(p^\beta)]\\
  @V{\pi_\lambda}VV @VVV\\
  H_{\rm c,ord}^{\dim\mathscr Y}({\mathscr X}(K_{\alpha',\alpha}); \underline{p^{-\alpha}\OO/\OO})
  @>{\mu^{0,0}_{\alpha,\beta}}>>
(p^{-\alpha}\OO/\OO)_{(0)}[C(p^\beta)]
\end{CD}
$$
for all $\beta\geq\alpha\geq\alpha_0^K$. And on the other hand, the commutativity of this diagram is by Theorem \ref{thm:independenceofweight} equivalent to the identity
\begin{equation}
  \mathscr P_{\alpha,x,\beta}^{\lambda,0}(\iota_\lambda(\phi))\;=\;
  \mathscr P_{\alpha,x,\beta}^{0,0}(\phi),
  \label{eq:modularsymbolchangeofweight}
\end{equation}
for all $\beta\geq\alpha\geq\alpha_0^K$ and $\phi$.

To this point, we observe first, that the congruence \eqref{eq:manincongruence} and the relation \eqref{eq:specialomegaone} in Proposition \ref{prop:generalmaninlemma}, imply the identity
$$
\eta_{p^{-\alpha}\OO/\OO,0}((-\lambda^{w_0})(t_\beta)\cdot
g_\beta (p^{-\alpha}v_0))\;=\;\eta_{p^{-\alpha}\OO/\OO,0}(p^{-\alpha}g_0v_0)\;\in\;p^{-\alpha}\OO/\OO.
$$
By the definition of $i_\lambda:p^{-\alpha}\OO/\OO\to L_{\lambda,p^{-\alpha}\OO/\OO}$, this implies that the following diagram of sheaves on $\mathscr Y(g_\beta K_{\alpha',\alpha} g_\beta^{-1}\cap H(\Adeles^{(\infty)})$
\begin{equation}
\begin{CD}
  i^*\left[(-\lambda^{w_0})(t_p^\beta)\cdot t_{g_\beta}^{\lambda}\right]\underline{L}_{\lambda,p^{-\alpha}\OO/\OO}
  @>{\eta_{0,p^{-\alpha}\OO/\OO}}>>
  \underline{p^{-\alpha}\OO/\OO}\\
  @A{i_\lambda}AA @|\\
  i^*t_{g_\beta}^{0}\underline{p^{-\alpha}\OO/\OO}
  @>{\eta_{0,p^{-\alpha}\OO/\OO}}>>
  \underline{p^{-\alpha}\OO/\OO}
\end{CD}
\label{eq:weightchangediagram}
\end{equation}
commutes. Using the commutativy of \eqref{eq:weightchangediagram} and the Hecke equivariance of $\iota_\lambda$, we obtain
\begin{align*}
  &\eta_{0,p^{-\alpha}\OO/\OO}\left(i^*\left[(-\lambda^{w_0})(t_p^\beta)\cdot t_{g_\beta}^{\lambda}\right](U_p^{-\beta}\iota_\lambda(\phi))\right)
  \\=\;&
  \eta_{0,p^{-\alpha}\OO/\OO}\left(i^*\left[(-\lambda^{w_0})(t_p^\beta)\cdot t_{g_\beta}^{\lambda}\right](\iota_\lambda(U_p^{-\beta}\phi))\right)
  \\=\;&
   \eta_{0,p^{-\alpha}\OO/\OO}\left(i^*t_{g_\beta}^{0}(U_p^{-\beta}\phi)\right),
\end{align*}
which proves \eqref{eq:modularsymbolchangeofweight}.
\end{proof}

We fix for all $\alpha\geq\alpha_0^K$ compatibly isomorphisms
\begin{equation}
  H_{\rm c}^{\dim\mathscr X}(K_{\alpha,\alpha};\underline{\QQ_p/\ZZ_p})\;\cong\;\QQ_p/\ZZ_p.
  \label{eq:dualityiso}
\end{equation}
Such a compatible choice is unique up to scalars in $\ZZ_p^\times$ and we are free to make it compatible with our choice of Haar measure on $G(\Adeles)$. By Poincar\'e-Pontryagin duality, our choice of \eqref{eq:dualityiso} allows us to consider $\mu^{\lambda,0}$ as an element of
$$
\Hom_\OO\left(\mathcal H_{\rm c,ord}^{\dim\mathscr Y}(K_{\infty,\infty};L_{\lambda,E/\OO}),E/\OO\right)\widehat{\otimes}_\OO\OO[[C(p^\infty)]]\;=\;
\mathcal H_{\rm ord}^{\dim\mathscr X-\dim\mathscr Y}(K_{\infty,\infty};L_{\lambda^\vee,\OO})[[C(p^\infty)]],
$$
where as before (again via transfer maps)
$$
\mathcal H_{\rm ord}^{\dim\mathscr X-\dim\mathscr Y}(K_{\infty,\infty};L_{\lambda^\vee,\OO})\;=\;
\varprojlim_{\alpha,\alpha'}
H_{\rm ord}^{\dim\mathscr X-\dim\mathscr Y}(\mathscr X(K_{\alpha',\alpha});L_{\lambda^\vee,\OO/\alpha'\OO}^\circ).
$$
Since the interpretation of $\mu^{\lambda,0}$ as an element of this space depends on the choice of isomorphism in \eqref{eq:dualityiso}, we introduce an ambiguity in the form of a $p$-adic period in $\ZZ_p^\times$.

We remark that while $\dim\mathscr Y=q_0$ is the bottom degree, $\dim\mathscr X-\dim\mathscr Y=q_0+l_0$ is the top degree for $G$.

There is inherent redundancy in this construction: The ambient space containing $\mu^{\lambda,0}$ admits two canonical $\OO[[C(p^\infty)]]$-module structures. To understand their interrelation, consider the map\nomenclature[]{$L_p^{\lambda,j}$}{integration of $\mu^{\lambda,j}$ over trivial character}
$$
L_p^{\lambda,j}:\quad
\mathcal H_{\rm c, ord}^{\dim\mathscr Y}(K_{\infty,\infty};L_{\lambda,E/\OO})\;\to\;
(E/\OO)_{(j)},
$$
$$
\phi\;\mapsto\;\int_{C(p^\infty)}{\bf1} d\mu^{\lambda,j}(\phi)
$$
Then by the distribution property, for $\beta\gg 0$,
$$
L_p^{\lambda,j}(\phi)\;=\;\sum_{x\in C(p^\beta)}\varinjlim_{\alpha}\mathscr P_{\alpha,x,\beta}^{\lambda,j}(\phi)\;\in\;(E/\OO)_{(j)}.
$$
By definition, we may identify $\mu^{\lambda,j}$ and $\mu^{\lambda+(j),0}$, whence also $L_p^{\lambda,j}$ and $L_p^{\lambda+(j),0}$. Therefore, we may and do assume $j=0$ in the sequel. As before,
$$
\Hom_\OO\left(\mathcal H_{\rm c, ord}^{\dim\mathscr Y}(K_{\infty,\infty};L_{\lambda,E/\OO}),E/\OO\right)
\;=\;
\mathcal H_{\rm ord}^{\dim\mathscr X-\dim\mathscr Y}(K_{\infty,\infty};L_{\lambda^\vee,\OO}),
$$
whence if $\eta_0$ is admissible for $\lambda$,
$$
L_p^{\lambda,0}\;\in\;\mathcal H_{\rm ord}^{\dim\mathscr X-\dim\mathscr Y}(K_{\infty,\infty};L_{\lambda^\vee,\OO}).
$$
Theorem \ref{thm:muindependenceofweight} implies

\begin{corollary}[Independence of weight]\label{cor:Lpindependenceofweight}
  For any $\lambda$ for which $\eta_0$ is admissible,
  $$
  \pi_{\lambda^\vee}^\circ(L_p^{\lambda,0})\;=\;L_p^{0,0}.
  $$
\end{corollary}

Consider the diagram
\begin{equation}
\begin{CD}
    H_{\rm c, ord}^{\dim\mathscr Y}(\mathscr X(K_{\alpha,\alpha});\underline{L}_{\lambda,E/\OO})@>\mu_{E/\OO}^{\lambda,0}>>E/\OO[[C_F(p^\infty)]]\\
    @VVV @VV\mu\mapsto\int_{C_F(p^\infty)}d\mu{}V\\
    \mathcal H_{\rm c, ord}^{\dim\mathscr Y}(K_{\infty,\infty};\underline{L}_{\lambda,E/\OO})@>L_p^{\lambda,0}>>E/\OO
\end{CD}
\label{eq:specializationdiagram}
\end{equation}
which commutes by construction. We consider the composition
$$
  \phi\;\mapsto\;\int_{C_F(p^\infty)}d\mu_{E/\OO}^{\lambda,0}(\phi)
$$
as an element $L_{p,\alpha}^{\lambda,0}$ of\nomenclature[]{$L_{p,\alpha}^{\lambda,0}$}{integration of $\mu_{E/\OO}^{\lambda,0}$ against trivial character as an element of $H_{\rm ord}^{q_0+l_0}(\mathscr X(K_{\alpha,\alpha});\underline{L}_{\lambda,\OO})$}
$$
  H_{\rm ord}^{q_0+l_0}(\mathscr X(K_{\alpha,\alpha});\underline{L}_{\lambda^\vee,\OO}^\circ)\;=\;
  \Hom_\OO(H_{\rm c, ord}^{\dim\mathscr Y}(\mathscr X(K_{\alpha,\alpha});\underline{L}_{\lambda,E/\OO}),E/\OO).
$$
  Then by the commutativity of \eqref{eq:specializationdiagram}, we obtain
  \begin{lemma}\label{lem:controlforLp}
    The canonical map
    $$
      \mathcal H_{\rm ord}^{q_0+l_0}(K_{\infty,\infty};\underline{L}_{\lambda^\vee,\OO})\;\to\;
      H_{\rm ord}^{q_0+l_0}(\mathscr X(K_{\alpha,\alpha});\underline{L}_{\lambda^\vee,\OO}^\circ)
    $$
    maps $L_p^{\lambda,0}$ to $L_{p,\alpha}^{\lambda,0}$.
  \end{lemma}

\begin{remark}\label{rmk:reconstruction}
  We may reconstruct the full measure
  $$
    \mu_{E/\OO}^{\lambda,0}\;\in\;
    H_{\rm ord}^{q_0+l_0}(\mathscr X(K_{\alpha,\alpha});\OO)[[C_F(p^\infty)]]
  $$
  from $L_p^{0,0}$ as follows. Let $\chi$ be an $\OO$-valued character with conductor dividing $p^\beta$. After possibly passing to a larger $\alpha$, we may assume $\alpha\geq\beta$ and consider $\chi$ as a function on $C(\det(K_{\alpha,\alpha}))$, which pulls back to a locally constant function on $\mathscr X(K_{\alpha,\alpha})$. Therefore, $\chi$ gives rise to a degree zero cohomology class\nomenclature[]{$[\chi]$}{degree zero cohomology class attached to finite order character $\chi$}
  $$
    [\chi]\;\in\;H^0(\mathscr X(K_{\alpha,\alpha});\OO).
  $$
  Then the evaluation of $\mu_{E/\OO}(\phi)$ at the character $\chi$ for
  $$
    \phi\;\in\;H_{\rm c,ord}^{q_0}(\mathscr X(K_{\alpha,\alpha});\underline{L}_{\lambda,E/\OO})
  $$
  is given by
  \begin{align*}
    L_p^{0,0}(\pi_\lambda(\phi\cup[\chi]))
    \;=\;&L_p^{\lambda,0}(\phi\cup[\chi])\\
    =\;&L_{p,\alpha}^{\lambda,0}(\phi\cup[\chi])\\
    =\;&\int_{C_F(p^{\infty})}d\mu_{E/\OO}^{\lambda,0}(\phi\cup[\chi])\\
    =\;&\int_{C_F(p^{\infty})}\chi d\mu_{E/\OO}^{\lambda,0}(\phi).
  \end{align*}
\end{remark}
  
  By Theorem \ref{thm:localizedcohomology}, together with Corollary \ref{cor:Lpindependenceofweight} and the Control Theorem in its dual form in Corollary \ref{cor:Hordcontrol} we obtain
  
\begin{theorem}\label{thm:controlforlocalizedLp}
  Assume $p\nmid (n+1)n$ and that $F$ is either totally real, CM, or that conjecture \ref{conj:galoisrepresentations} holds for $F$. Fix a non-Eisenstein maximal ideal $\mathfrak{m}$ in ${\bf h}_{\rm ord}(K_{\infty,\infty},\OO)$ and write $\mathfrak{m}^\vee$ for the contragredient maximal ideal with contragredient residual Galois representation\nomenclature[]{$\mathfrak{m}^\vee$}{maximal ideal for contragredient residual representation}. Then for every locally algebraic $\OO$-valued character $\lambda^\vee\vartheta^{-1}$ of $T(\ZZ_p)$ with dominant $\lambda$ for which $\eta_0$ is admissible, and $\alpha\geq\alpha_0^K$ satisfying $\mathfrak{f}_\vartheta\mid p^\alpha$, the canonical map
  $$
    \mathcal H_{\rm ord}^{q_0+l_0}(K_{\infty,\infty};\OO)_{\mathfrak{m}^\vee}\otimes_{\Lambda}\Lambda^\circ/P_{\lambda^\vee\vartheta^{-1}}^\circ\;\to\;
    H_{\rm ord}^{q_0+l_0}(\mathscr X(K_{\alpha,\alpha});\underline{L}_{\lambda^\vee,\OO}^\circ)_{\mathfrak{m}^\vee}\otimes_{\Lambda^\circ,\vartheta^{-1}} \OO
  $$
  maps $L_p^{0,0}$ mod $P_{\lambda^\vee\vartheta^{-1}}^\circ$ to $L_{p,\alpha}^{\lambda,0}$ mod $P_{\vartheta^{-1}}^\circ$.
\end{theorem}

\section{$p$-adic $L$-functions}\label{sec:padicL}

\subsection{Abelian $p$-adic $L$-functions for automorphic representations}\label{sec:abelianpadicL}

Recall that $F/\QQ$ denotes a number field and $G=\res_{F/\QQ}\GL(n+1)\times\GL(n)$ as before. For any regular algebraic cuspidal automorphic representation $\Pi\widehat{\otimes}\Sigma$ of $G(\Adeles)$ of cohomological weight $\lambda$, the action of the finite Hecke algebra on $\Pi\widehat{\otimes}\Sigma$ is defined over the field of rationality $\QQ(\Pi,\Sigma)/\QQ$ of $\Pi$ and $\Sigma,$ which is a number field by the work of Clozel \citep{clozel1990} (cf.\ \citep{januszewski2017} for a globalization of this result)\nomenclature[]{$\QQ(\Pi,\Sigma)$}{field of rationality of $\Pi\widehat{\otimes}\Sigma$}. We fix embeddings $\QQ(\Pi,\Sigma)\to\CC$ and $\QQ(\Pi,\Sigma)\to E$ where $E/\QQ_p$ is sufficiently large. This allows us to refer to $p$-adic absolute values of eigenvalues of Hecke operators acting on $\Pi\widehat{\otimes}\Sigma$.

We call $\Pi\widehat{\otimes}\Sigma$ {\em nearly ordinary} at a rational prime $p$ if for some $\alpha\gg 0$ there is a $\phi\in\Pi\widehat{\otimes}\Sigma$ which is an eigenvector of $\mathcal H_{\QQ(\Pi)}(\alpha,\alpha)$ with eigenvalue $\vartheta:T(\QQ_p)\to\QQ(\Pi)^\times$ satisfying
\begin{equation}
  |\lambda^\vee(t_p)\vartheta(t_p)|_p\;=\;1,
  \label{eq:ordinaryeigenvalue}
\end{equation}
for the normalized absolute value $|\cdot|_p$ on $E$. This is the same to say that both $\Pi$ and $\Sigma$ are nearly ordinary at $p$ (for the standard Borel subgroups $B_{n+1}$ and $B_n$) in the sense of \citep{hida1998}. We will see in the course of the proof of Theorem \ref{thm:interpolation} below that in the nearly ordinary case, $\vartheta$ is then uniquely determined by $\Pi\widehat{\otimes}\Sigma.$

  We fix embeddings\nomenclature[]{$i_\infty$}{fixed embedding $\QQ(\Pi,\Sigma)\to\CC$}
  $$
    i_\infty:\quad \QQ(\Pi,\Sigma)\to\CC,
  $$
  and\nomenclature[]{$i_p$}{fixed embedding $\QQ(\Pi,\Sigma)\to E$}
  $$
    i_p:\quad \QQ(\Pi,\Sigma)\to E.
  $$
  Once appropriately normalized, the special values of the $L$-function $L(s,\Pi\widehat{\otimes}\Sigma)$ lie in $\QQ(\Pi,\Sigma)$ and hence also in $E$ provided that $\lambda$ is balanced in the sense of \eqref{eq:admissiblelambdacondition}, cf.\ \citep{raghuram2015,januszewski2018}. In the following we implicitly assume that the a priori complex valued $p$-adic Nebentypus $\vartheta$ of $\Pi\widehat{\otimes}\Sigma$ also takes values in $E$.

\begin{theorem}\label{thm:interpolation}
  Let $\Pi\widehat{\otimes}\Sigma$ be an irreducible regular algebraic cuspidal automorphic representation of $G(\Adeles)$ of cohomological weight $\lambda$. Assume the following:
  \begin{itemize}
    \item[(i)] $\lambda$ is balanced.
    \item[(ii)] $\Pi\widehat{\otimes}\Sigma$ is nearly ordinary at a prime $p$ and $\vartheta:T(\QQ_p)\to\CC^\times$ the corresponding Nebentypus.
  \end{itemize}
  Then there are complex periods $\Omega_{\pm,j}\in\CC^\times$, indexed by the characters of $\pi_0(F_\infty^\times)$ and $j\in\ZZ$ for which \eqref{eq:jfunctionals} is non-zero, and a unique $p$-adic measure $\mu_{\Pi\widehat{\otimes}\Sigma}\in\OO[[C_F(p^\infty)]]$ with the following property.\nomenclature[]{$\mu_{\Pi\widehat{\otimes}\Sigma}$}{$\OO$-valued measure on $C_F(p^\infty)$ attached to $\Pi\widehat{\otimes}\Sigma$}
  For every $s_0=\frac{1}{2}+j$ critical for $L(s,\Pi\widehat{\otimes}\Sigma)$, 
  for all finite order Hecke characters $\chi$ of $F$ unramified outside $p\infty$ and such that $\chi_p\vartheta$ has fully supported constant conductor $\mathfrak{f}_{\chi\vartheta}$, 
  \begin{align*}
    &
    \int\limits_{C_F(p^\infty)}\chi(x)\omega_F^j(x)\langle x\rangle_F^jd\mu_{\Pi\widehat{\otimes}\Sigma}(x)\;=\;\\
    &
    \absNorm(\mathfrak{f}_{\chi\vartheta})^{j\frac{(n+1)n}{2}-\frac{(n+1)n(n-1)}{6}}\cdot
    \prod_{\mu=1}^n\prod_{\nu=1}^\mu
      G(\chi\vartheta_{\mu,\nu})
    \cdot
    \frac{L(s,\Pi\widehat{\otimes}\Sigma\otimes\chi)}{\Omega_{(-1)^j\sgn\chi,j}}.
  \end{align*}
  Furthermore, $\mu_{\Pi\widehat{\otimes}\Sigma}$ is determined uniquely by the interpolation property for a single critical $s_0=\frac{1}{2}+j$.
\end{theorem}

Previously, Schwab settled the \lq{}Manin congruences\rq{} for $n=2,$ $\lambda_3=(2,1,0)$ and $\lambda_2=(1,0)$ in her Diploma thesis \citep{schwab2015}, while the case $n=1$ was treated by Namikawa in \citep{namikawa2016}.

The automorphic representations $\Pi\widehat{\otimes}\Sigma$ of $G$ we consider are topological tensor products of regular algebraic cuspidal automorphic representations $\Pi$ and $\Sigma$ of $\GL(n+1)$ and $\GL(n)$ respectively. Motivically, the $L$-functions we study therefore conjecturally are $L$-functions of tensor products of irreducible pure motives of dimensions $n+1$ and $n$. In this context, condition (i) (i.\,e.\ condition \eqref{eq:admissiblelambdacondition}) translates into a certain condition on the Hodge-Tate weights of those motives and in light of Deligne's Conjecture on special values of motivic $L$-functions and its integral refinements, the complex periods and other invariants in the interpolation formulae of the $p$-adic $L$-functions we construct should depend on the relative position of these Hodge-Tate weights in general (cf.\ \citep{coatesperrinriou1989,coates1989}). Therefore, it is not surprising that we need to restrict our treatment to a certain class of relative positions when working with a single modular symbol in a fixed cohomological degree, which applies to the case at hand.

\begin{proof}
  As explained in section \ref{sec:birchglobal}, we may choose at every finite place $v\nmid p$ of $F$ a good test vector
  $$
    W_v^0\;\in\;\mathscr W(\Pi_v\otimes\Sigma_v,\psi_v\otimes\psi_v^{-1})
  $$
  which is the product of two normalized spherical Whittaker functions whenever $\Pi_v$ and $\Sigma_v$ are unramified. At $p$ we choose an eigenvector
  $$
    W_p\;\in\;\mathscr W(\Pi_p\otimes\Sigma_p,\psi_p\otimes\psi_p^{-1})
  $$
  for $\mathcal H_\CC(\alpha,\alpha)$ with eigenvalue $\vartheta$. From the proof of Proposition 6.4 in \citep{hida1998} we know that at each $v\mid p$ the representations $\Pi_v$ and $\Sigma_v$ are both subquotients of a principal series representation as considered in Proposition \ref{prop:nonvanishingofeigenvectors}, and that $\vartheta$ is uniquely determined by $\Pi_p$ and $\Sigma_p$ and the ordinarity condition. Furthermore, $W_p$ lies in a unique line and
  $$
    W_p({\bf 1})\;\neq\;0.
  $$
  Recall that $\lieg$ and $\liek$ denote the complexified Lie algebras of $G(\RR)$ and the standard maximal compact $K_\infty\subseteq G(\RR)$ respectively. We set $\liegk:=\liez+\liek$ where $\liez$ is the complexified Lie algebra of the center $Z(\RR)\subseteq G(\RR)$. Then for every character $\varepsilon:\pi_0(G(\RR))\to\CC^\times,$ the $\varepsilon$-eigenspace
  \begin{equation}
    H^{\dim\mathscr Y}(\lieg,\liegk; \mathscr W(\Pi_\infty\widehat{\otimes}\Sigma_\infty,\psi_\infty\otimes\psi_\infty^{-1})\otimes L_{\lambda,\CC})_\varepsilon\;=\;\left(\bigwedge^{\dim\mathscr Y}(\lieg/\liegk)^*\otimes\Pi_\infty\widehat{\otimes}\Sigma_\infty\otimes L_{\lambda,\CC}\right)^{K_\infty^0}_\varepsilon
    \label{eq:epsilonspaceingkcohomology}
  \end{equation}
  in cohomology is at most one-dimensional. Furthermore, it is non-trivial if and only if the restriction of $\varepsilon$ to $\pi_0(\GL_m(F\otimes\RR))$ where $m\in\{n+1,n\}$ is odd, agrees with the restriction of the product of the central characters $\omega_{\Pi_\infty\widehat{\otimes}\Sigma_\infty}\omega_{L_{\lambda,\CC}}$ restricted to the subgroup $\{\pm{\bf1}_m\}\subseteq \GL_m(F\otimes\RR)$.

  For each such $\varepsilon$ we pick a generator $\varphi_{\varepsilon,\infty}$ in \eqref{eq:epsilonspaceingkcohomology} and fix a non-zero
  $$
  \eta_j\in\Hom_H(V_{\lambda,\CC},\CC_{(j)}),
  $$
  which exists by \citep{kastenschmidt2013}, Theorem 2.3 and \citep{raghuram2015}, Theorem 2.21. Then each $\varphi_{\varepsilon,\infty}$ projects under the map
  $$
  \res^{G}_{H}\otimes{\bf1}\otimes\eta_{j,\CC}:\quad
  \bigwedge^{\dim\mathscr Y}(\lieg/\liegk)^*\otimes\mathscr W(\Pi_\infty\widehat{\otimes}\Sigma_\infty,\psi_\infty\otimes\psi_\infty^{-1})\otimes L_{\lambda,\CC}\to
  $$
  $$
  \quad\quad\quad  \quad\quad\quad  \quad\quad\quad  \quad\quad\quad  \quad\quad\quad
  \bigwedge^{\dim\mathscr Y}(\lieh/(\liegk\cap\lieh))^*\otimes\mathscr W(\Pi_\infty\widehat{\otimes}\Sigma_\infty,\psi_\infty\otimes\psi_\infty^{-1})\otimes L_{\lambda,\CC}
  $$
  to a vector of the form
  $$
  \res^{G}_{H}\otimes{\bf1}\otimes\eta_{j,\CC}(\varphi_{\varepsilon,\infty})
  \;=\;
  \omega_{\infty}
    \otimes W^{\rm coh}_{\varepsilon,\infty}\otimes 1.
  $$
  Here, $\omega_{\infty}$ is the fixed generator of the line $\bigwedge^{\dim\mathscr Y}(\lieh/(\liegk\cap\lieh))^*$ (independent of $\varepsilon$) corresponding to our choice of invariant measure on  $H(\RR)^0/\widetilde{K}_\infty\cap H(\RR)^0$, $1\in\CC_{(j)}$ is a generator as before, and
  $$
    W^{{\rm coh},j}_{\varepsilon,\infty}
    \in\mathscr W(\Pi_\infty\widehat{\otimes}\Sigma_\infty,\psi_\infty\otimes\psi_\infty^{-1}).
  $$
  The latter vector is commonly refered to as a {\em cohomological test vector}. We emphasize that this vector possibly {\em does} depend on $j$.

  On the one hand, we obtain for each such $\varepsilon$ a global Whittaker vector\nomenclature[]{$W_\varepsilon^{{\rm coh},j}$}{$j$-th global cohomological Whittaker vector for sign $\varepsilon$}
  $$
    W_\varepsilon^{{\rm coh},j}\;:=\;W^{{\rm coh},j}_{\varepsilon,\infty}\otimes W_p\otimes\left(\otimes_{v\nmid p\infty}W_v^0\right)
    \;\in\;\mathscr W(\Pi\widehat{\otimes}\Sigma,\psi\otimes\psi^{-1}).
  $$
  By \citep{kastenschmidt2013,sunjams} (cf.\ also section 8 in \citep{januszewski2016}), we know
  \begin{equation}
    \Omega_\infty(j,W^{{\rm coh},j}_{\varepsilon,\infty})\;\neq\;0.
    \label{eq:Lnormalization}
  \end{equation}
  Set
  $$
  W^{{\rm coh},j}\;:=\;\sum_{\varepsilon}W^{{\rm coh},j}_{\varepsilon}.
  $$
  Then for every finite order character $\chi$ unramified outside $p\infty$,
  $$
    \Omega^{S(p)}\left(j,\left(W^{{\rm coh},j}\right)_\chi\right)
    \;=\;\Omega_\infty\left(j,W_{\chi_\infty}^{{\rm coh},j}\right)\;\neq\;0,
  $$
  only depends on $j$ and $\chi_\infty$.

  On the other hand, we have cohomology classes
  $$
    \left[\varphi_{\varepsilon,\infty}\otimes W_p\otimes\left(\otimes_{v\nmid p\infty}W_v^0\right)\right]\;\in\;
    H^{\dim\mathscr Y}(\lieg,\liegk; \mathscr W(\Pi\widehat{\otimes}\Sigma,\psi\otimes\psi^{-1})\otimes L_{\lambda,\CC})_\varepsilon,
  $$
  which by inverse Fourier transform give rise to global cohomology classes
  $$
    \phi_\varepsilon\;\in\;H^{\dim\mathscr Y}_{\rm c}(\mathscr X(K_{\alpha,\alpha});\underline{L}_{\lambda,\CC})_\varepsilon.
  $$
  We may normalize the classes $\phi_\varepsilon$ such that they lie $p$-optimally in
  $$
    \phi_\varepsilon\;\in\;H_{\rm c}^{\dim\mathscr Y}(\mathscr X(K_{\alpha,\alpha});\underline{L}_{\lambda,\QQ(\Pi,\Sigma)})_\varepsilon,
  $$
  where $p$-optimality is understood with respect to the embedding $i_p:\, \QQ(\Pi,\Sigma)\to E$ (the latter induces a $p$-adic valuation on $\QQ(\Pi,\Sigma)$).

  By the hypothesis \eqref{eq:ordinaryeigenvalue} on the $U_p$-eigenvalue of $W_p$, we find a preimage\nomenclature[]{$\phi_{\Pi\widehat{\otimes}\Sigma}$}{global $\OO$-integral cohomology class attached to $\Pi\widehat{\otimes}\Sigma$}
  $$
    \phi_{\Pi\widehat{\otimes}\Sigma}:=\sum_{\varepsilon}\phi_\varepsilon\;\in\;H_{\rm c,ord}^{\dim\mathscr Y}(\mathscr X(K_{\alpha,\alpha});\underline{L}_{\lambda,\OO}).
  $$
  The resulting $p$-adic measure will be indepentent of this choice of preimage. Choose any $j$ as above and put 
  $$
    \mu_{\Pi\widehat{\otimes}\Sigma}
    \;:=\;\omega_F^{-j}(-)\langle-\rangle_F^{-j}\mu_{\OO}^{\lambda,j}(\phi_{\Pi\widehat{\otimes}\Sigma}).
  $$
  By Theorem \ref{thm:tatetwists}, $\mu_{\Pi\widehat{\otimes}\Sigma}$ is independent of $j$. Hence, for any $\chi$ and any $j$ as in the statement of the Theorem, if $\beta\geq\alpha$ such that $f_{\chi\vartheta}\mid p^\beta$, then
  \begin{align*}
    &\int\limits_{C_F(p^\infty)}\chi(x)\omega_F^j(x)\langle x\rangle_F^jd\mu_{\Pi\widehat{\otimes}\Sigma}(x)\\
    =\;&
      \int\limits_{C_F(p^\infty)}\chi d\mu_{\OO}^{\lambda,j}(\phi_{\Pi\widehat{\otimes}\Sigma})\\
    =\;&
      \sum_{x\in C(p^\beta)}\chi(x)
      \mathscr P_{\OO,x,\beta}^{\lambda,j}(\phi_{\Pi\widehat{\otimes}\Sigma})\\
    =\;&
      \sum_{x\in C(p^\beta)}\chi(x)
      \int_{\mathscr Y(g_\beta K_{\alpha,\alpha} g_\beta^{-1}\cap H(\Adeles^{(\infty)}))[x]}
      \eta_{j,\OO} i^*\left[\lambda^\vee(t_p^\beta)\cdot t^{\lambda}_{g_{\beta}}\right]
      (U_p^{-\beta}\phi_{\Pi\widehat{\otimes}\Sigma})\\
    =\;&
      \lambda^\vee(t_p^\beta)\cdot
      [H(\widehat{\ZZ}):H(\Adeles^{(\infty)})\cap g_\beta K_{\alpha,\alpha}g_\beta^{-1}]\cdot\\
    &
      \int_{H(\QQ)\backslash H(\Adeles)}
      \phi_{\left(\lambda^\vee(t_p^{-\beta})U_p^{-\beta}W^{{\rm coh},j}\right)_{(-1)^j}}\left(g\cdot ht_{p^\beta}\right)\chi(\det(g))\absnorm{\det(g)}^jdg,
    \end{align*}
    By Proposition \ref{prop:leveladdition},
    \begin{align*}
      &
           {[H(\widehat{\ZZ}):H(\Adeles^{(\infty)})\cap g_\beta K_{\alpha,\alpha}g_\beta^{-1}]}\\
      =\;&
           {[H(\widehat{\ZZ}^{(p)}):K^{(p)}\cap H(\Adeles^{(p\infty)})]
             \cdot [H(\ZZ_p):H(\QQ_p)\cap g_\beta I_{\alpha,\alpha}g_\beta^{-1}]}\\
       =\;&
            {[H(\widehat{\ZZ}^{(p)}):K^{(p)}\cap H(\Adeles^{(p\infty)})]
              \cdot [H(\ZZ_p):\mathfrak I_\beta^n]}\\
       =\;&
            {[H(\widehat{\ZZ}^{(p)}):K^{(p)}\cap H(\Adeles^{(p\infty)})]}\cdot
            \prod_{v\mid p}\prod_{\mu=1}^{n}\left(1-q_v^{-\mu}\right)^{-1}\cdot\\
       &
            \absNorm(p\OO_p)^{\beta\frac{(n+2)(n+1)n+(n+1)n(n-1)}{6}}.
    \end{align*}
    By the global Birch Lemma (cf.\ Theorem \ref{thm:globalbirch}),
    \begin{align*}
      &
        \vartheta(t_p^{-\beta})
        \int_{H(\QQ)\backslash H(\Adeles)}
        \phi_{\left(W^{{\rm coh},j}\right)_{(-1)^j}}\left(g\cdot ht_{p^\beta}
                         \right)\chi(\det(g))|\det(g)|^jdg\\
      =\;&
        \vartheta(t_p^{-\beta})\cdot
        \vartheta(t_p^\beta)\cdot
        \absNorm(p\OO_p)^{-\beta\frac{(n+2)(n+1)n+(n+1)n(n-1)}{6}}\cdot
        \absNorm(\mathfrak{f}_{\chi\vartheta})^{-\frac{(n+1)n(n-1)}{6}}\cdot
        \absnorm{t_{f_{\chi\vartheta}}}^{-j}\cdot\\
     &
        \Omega_\infty(j,W^{{\rm coh},j}_{(-1)^j\chi_\infty})\cdot
        \prod_{\mu=1}^n\prod_{\nu=1}^\mu
        G(\chi\vartheta_{\mu,\nu})\cdot
        L(j,\Pi\widehat{\otimes}(\Sigma\otimes\chi))
  \end{align*}
  Collecting terms concludes the proof of existence of $\mu_{\Pi\widehat{\otimes}\Sigma}$  with
  $$
  \Omega_{(-1)^j\sgn\chi,j}\;:=\;\Omega_\infty(j,W^{{\rm coh},j}_{(-1)^j\chi_\infty})^{-1}.
  $$

  Using Lemma 10.2 in \citep{hida1988a}, it is easy to see that the interpolation property in Theorem \ref{thm:interpolation} at a single critical $s_0=\frac{1}{2}+j$ determines the measure $\mu_{\Pi\widehat{\otimes}\Sigma}$ uniquely, cf.\ Corollary 6.9 in \citep{januszewski2015}.
\end{proof}

\begin{remark}\label{rmk:omegajvariance}
  The dependence of the periods $\Omega_{\pm,j}$ on $j$ originates from the fact that for each $j$ we obtain potentially a {\em different} \lq{}cohomoligcal\rq{} vector inside the archimedean component of $\Pi\widehat{\otimes}\Sigma$. The effect of this dependence on the periods $\Omega_{\pm,j}$ is studied in \citep{januszewski2018,grobnerlin2017,harderraghuram2017}. It is a result of the author that each $\varphi_{\epsilon,\infty}$ lies in a $\QQ(\Pi,\Sigma,\sqrt{(-1)^{\frac{(n+1)n}{2}}})$-rational structure, and so do $W^{\rm coh}_{\varepsilon,\infty}$ and $W_\varepsilon$ (cf.\ \citep{januszewski2017,januszewski2018}). We remark that in the case of $\GL(3)\times\GL(2)$ results in \citep{hiranoishiimiyazaki2016,hiranoishiimiyazaki2018} indicate, that at least in this case it should be possible to make the period relations between $\Omega_{\pm,j}$ explicit integrally. For case $F=\QQ$ and $n=2$ this was carried out by Hara and Namikawa \cite{haranamikawa2021}, who showed that the period $\Omega_{(-1)^j\sgn\chi,j}$, studied here is, as expected, up to powers of $2\pi i$ and binomial coefficients, independent of $j$ and factors canonically as a product of cohomological periods attached to the automorphic representations $\Pi$ and $\Sigma$ on $\GL(3)$ and $\GL(2)$ (cf.\ Theorem 1.1 in loc.\ cit.).
\end{remark}

\subsection{Non-abelian $p$-adic $L$-functions}\label{sec:nonabelianpadicL}

  Assume as before that $F$ is totally real, CM or the validity of Conjecture \ref{conj:galoisrepresentations}. Consider a non-Eisenstein component ${\rm h}^{q_0+l_0}_{\rm ord}(K_{\alpha',\alpha};\lambda,\OO)_{\mathfrak{m}^\vee}$ for a non-Eisenstein maximal ideal $\mathfrak{m}$ in ${\bf h}_{\rm ord}(K_{\infty,\infty};\OO)$. We also assume that $p\nmid(n+1)n$ and $S(K)=S(p)$, i.e.\ $K$ has full level outside $p$ for technical reasons.

  Then we know by Hida's Corollary A.4 in \citep{hida1998} (see also the proof of Prop.\ 6.4 in loc.\ cit.), which is applicable in our situation, that for each dominant weight $\lambda$, inner nearly ordinary cohomology
  \begin{equation}
    H^{q_0+l_0}_{!,\rm ord}(\mathscr X(K_{\alpha',\alpha});\underline{L}_{\lambda^\vee,E})
    \label{eq:finitecohomology}
  \end{equation}
  is a semi-simple ${\rm h}^{q_0+l_0}_{!,\rm ord}(K_{\alpha',\alpha};\lambda^\vee,E)$-module, each simple {\em cuspidal} summand occuring with multiplicity $|\pi_0(H(\RR))|$, provided that $K_{\alpha',\alpha}$ is neat and $E$ is sufficiently large (we refer to sections 2.3 and 4.3 in \citep{januszewski2018} and section 4.3 \citep{harderraghuram2017} for an explanation of the multiplicity in the case at hand). In particular, by Theorem \ref{thm:localizedcohomology},
  $$
    H^{q_0+l_0}_{\rm ord}(\mathscr X(K_{\alpha',\alpha});\underline{L}_{\lambda^\vee,\OO}^\circ)_{\mathfrak{m}^\vee}\;=\;
    E[\pi_0(G(\RR))]\otimes
    {\rm h}^{q_0+l_0}_{\rm ord}(K_{\alpha',\alpha};\lambda^\vee,E)_{\mathfrak{m}^\vee}
  $$
  as $\pi_0(H(\RR))\times{\rm h}^{q_0+l_0}_{\rm ord}(K_{\alpha',\alpha};\lambda^\vee,E)_{\mathfrak{m}^\vee}$-modules. Consequently, the space corresponding to the trivial signature at infinity
  \begin{equation}
    H^{q_0+l_0}_{\rm ord}(\mathscr X(K_{\alpha',\alpha});\underline{L}_{\lambda^\vee,\OO}^\circ)_{\mathfrak{m}^\vee}^{\pi_0(H(\RR))}\;=\;
    {\rm h}^{q_0+l_0}_{\rm ord}(K_{\alpha',\alpha};\lambda^\vee,E)_{\mathfrak{m}^\vee}
    \label{eq:finitefreeness}
  \end{equation}
  is free of rank one. We remark that the hypothesis $p\nmid(n+1)n$ implies that $p$ is odd and therefore ${\pi_0(H(\RR))}$ acts semi-simply on the cohomology spaces under consideration.

  On the one hand, each $E$ valued point $\xi\in\Spec{\rm h}^{q_0}_{\rm ord}(K_{\alpha',\alpha};\lambda,\OO)_{\mathfrak{m}}(E)$ corresponds bijectively to an $E$ valued point $\xi^\vee\in\Spec{\rm h}^{q_0+l_0}_{\rm ord}(K_{\alpha',\alpha};\lambda^\vee,\OO)_{\mathfrak{m}^\vee}(E)$, and on the other hand, $\xi$ corresponds bijectively to a cuspidal automorphic representation $\Pi_\xi\widehat{\otimes}\Sigma_\xi$ of $G(\Adeles)$ defined over $E$ whose contragredient contributes to \eqref{eq:finitecohomology} (cf.\ Theorem \ref{thm:localizedcohomology}), by virtue of the embeddings $i_\infty$ and $i_p$ from the previous section. For notational simplicity, we may fix once and for all an (algebraic) embedding $i_E:E\to\CC$.

  We call such a $\xi$ a {\em classical point} in $\Spec{\rm h}^{q_0}_{\rm ord}(K_{\alpha',\alpha};\lambda,\OO)_{\mathfrak{m}}(E)$. If we do not assume ${\mathfrak{m}}$ to be non-Eisenstein and consider inner cohomology, the representation $\Pi_\xi\widehat{\otimes}\Sigma_\xi$ is either cuspidal or residual. By \citep{lischwermer2004}, regularity of $\lambda$ always implies cuspidality of $\Pi_\xi\widehat{\otimes}\Sigma_\xi$. In the cuspidal case, Theorem \ref{thm:interpolation} provides us with an abelian $p$-adic $L$-function for $\xi$. In the residual case, the proof of Theorem \ref{thm:interpolation} shows that the modular symbol vanishes identically by Corollary 5.8 in \citep{ichinoyamana2015}.

  Back in the non-Eisenstein situation, let
  $$
    \xi\in\Spec{\rm h}^{q_0}_{\rm ord}(K_{\alpha',\alpha};\lambda,\OO)_{\mathfrak{m}}(E)
  $$
  be a classical point of balanced weight $\lambda$ and Nebentypus $\vartheta$. By Corollary \ref{cor:Hordcontrol},
  $$
    {\mathcal H}_{\rm ord}^{q_0+l_0}(K_{\infty,\infty};\OO)_{\mathfrak{m}^\vee}^{\pi_0(H(\RR))}
    \otimes_{\Lambda^\circ} \Lambda^\circ/P_{\lambda^\vee\vartheta^{-1}}^\circ
    \otimes_\OO E\;\cong\;
    H^{q_0+l_0}_{\rm ord}(\mathscr X(K_{\alpha',\alpha});\underline{L}_{\lambda^\vee,E}^\circ)_{\mathfrak{m}^\vee}^{\pi_0(H(\RR))}
    \otimes_{E\otimes_\OO\Lambda^\circ,\vartheta^{-1}}E,
  $$
  and
  \begin{equation}
    H^{q_0+l_0}_{\rm ord}(\mathscr X(K_{\alpha',\alpha});\underline{L}_{\lambda^\vee,\OO}^\circ)^{\pi_0(H(\RR))}_{\mathfrak{m}^\vee}\otimes_{
    {\rm h}^{q_0+l_0}_{\rm ord}(K_{\alpha',\alpha};\lambda^\vee,\OO)_{\mathfrak{m}^\vee},\xi^\vee}E
    \label{eq:specializedcohomologyoverE}
  \end{equation}
  is of rank $1$ over $E$, cf.\ \eqref{eq:finitefreeness}. The image of
  $$
    H^{q_0+l_0}_{\rm ord}(\mathscr X(K_{\alpha',\alpha});\underline{L}_{\lambda^\vee,\OO}^\circ)_{\mathfrak{m}^\vee}^{\pi_0(H(\RR))}\otimes_{
    {\rm h}^{q_0+l_0}_{\rm ord}(K_{\alpha',\alpha};\lambda^\vee,\OO)_{\mathfrak{m}^\vee},\xi^\vee}\OO
  $$
  is a canonical $\OO$-lattice $\OO_\xi$ in \eqref{eq:specializedcohomologyoverE}. Fix an $\OO$-basis $b_\xi\in\OO_\xi$\nomenclature[]{$b_\xi$}{a generator of $\OO_\xi$}. Identifying $L_{p}^{0,0}$ with its image under the canonical projection map
    $$
    {\mathcal H}_{\rm ord}^{q_0+l_0}(K_{\infty,\infty};\OO)_{\mathfrak{m}^\vee}\to
    {\mathcal H}_{\rm ord}^{q_0+l_0}(K_{\infty,\infty};\OO)_{\mathfrak{m}^\vee}^{\pi_0(H(\RR))},
    $$
    Theorems \ref{thm:controlforlocalizedLp} and \ref{thm:interpolation} imply

\begin{theorem}\label{thm:nonabelianinterpolation}
  Assume $p\nmid(n+1)n$, $F$ totally real, CM or Conjecture \ref{conj:galoisrepresentations}. Let $\mathfrak{m}$ denote a non-Eisenstein maximal ideal in $\mathfrak{m}$ in ${\bf h}_{\rm ord}(K_{\infty,\infty};\OO)$. Then the element
  $$
    L_{p}^{0,0}\;\in\;{\mathcal H}_{\rm ord}^{q_0+l_0}(K_{\infty,\infty};\OO)_{\mathfrak{m}^\vee}^{\pi_0(H(\RR))}
  $$
  has the following interpolation property. For every classical point
  $$
    \xi\in\Spec{\rm h}^{q_0}_{\rm ord}(K_{\alpha',\alpha};\lambda,\OO)_{\mathfrak{m}}(\overline{E})
  $$
  of balanced weight $\lambda$ and Nebentypus $\vartheta$, such that $s_0=\frac{1}{2}$ is critical for $L(s,\Pi_\xi\widehat{\otimes}\Sigma_\xi),$ we have
  \begin{align*}
    \Omega_{\xi,p}^{-1}\cdot\xi^\vee\circ \iota_{\lambda^\vee\vartheta^{-1}}(L_{p}^{0,0})
    \;=\;&
    \int_{C_F(p^\infty)}d\mu_{\Pi_\xi\widehat{\otimes}\Sigma_\xi}\cdot b_\xi\\
    =\;&
    \absNorm(\mathfrak{f}_{\vartheta})^{\frac{(n+1)n(n-1)}{6}}\cdot
    \prod_{\mu=1}^n\prod_{\nu=1}^\mu
      G(\vartheta_{\mu,\nu})
    \cdot
    \frac{L(\frac{1}{2},\Pi_\xi\widehat{\otimes}\Sigma_\xi)}{\Omega_{\xi}}\cdot b_\xi,
  \end{align*}
where the second identity is valid whenever $\vartheta$ has fully supported constant conductor.

  Here $\Omega_{\xi,p}^{-1}\in\OO[\xi]^\times$ is a $p$-adic period, $\Omega_{\xi}\in\CC^\times$ is a complex period and both these periods may be normalized in such a way that they are invariant under twists of $\Pi_\xi\widehat{\otimes}\Sigma_\xi$ by finite order Hecke characters $\chi$ unramified outside $p$.
\end{theorem}

By remark \ref{rmk:reconstruction}, $\Omega_{\xi,p}$ invariant under twisting $\xi$ with finite order Hecke characters $\chi$ unramified outside $p$, provided that we choose the test vectors in the construction of the measures for $\Pi_\xi\widehat{\otimes}\Sigma_\xi$ (or equivalently, the complex periods $\Omega_{\pm,0}$) coherently. That this is indeed possible is a consequence of Theorem \ref{thm:interpolation}.

\begin{proof}[{Proof of Theorem \ref{thm:nonabelianinterpolation}}]
    For simplicity of notation we may assume that $\xi$ is defined over $\OO$, i.e.\ $\OO[\xi]=\OO$. Then by the preceeding discussion and the definition of $L_{p}^{0,0}$, we have
  \begin{align*}
    \iota_{\lambda^\vee\vartheta^{-1}}^\circ(L_{p}^{0,0})
    \;=\;&\left[\phi\;\mapsto\;\int_{C_F(p^\infty)}d\mu_{E/\OO}^{\lambda,0}(\phi)\right],
  \end{align*}
  where $\phi\in H_{\rm c, ord}^{q_0}(\mathscr X(K_{\alpha',\alpha});\underline{L}_{\lambda,E/\OO})_{\mathfrak{m}}^{\pi_0(H(\RR))}
$. We emphasize that any possible defect in the control theorem is absorbed by the commutativity of diagram \eqref{eq:specializationdiagram}.

  Now application of $\xi^\vee$ sends $\iota_{\lambda^\vee\vartheta^{-1}}^\circ(L_{p}^{0,0})$ into the rank one $\OO$-lattice $\OO\cdot b_\xi$ in \eqref{eq:specializedcohomologyoverE}. By Poincar\'e-Pontryagin duality this operation corresponds to the precomposition of the map $\iota_{\lambda^\vee\vartheta^{-1}}^\circ(L_{p}^{0,0})(-)$ with the dual map
  $$
  \xi:\quad
   E/\OO\to
   H_{\rm c, ord}^{q_0}(\mathscr X(K_{\alpha',\alpha});\underline{L}_{\lambda,E/\OO})_{\mathfrak{m}}^{\pi_0(H(\RR))}.
  $$
  However, appealing to Poincar\'e duality over $\OO$ instead, we see that $\xi^\vee\circ\iota_{\lambda^\vee\vartheta^{-1}}(L_{p}^{0,0})$ corresponds to the map
  $$
  H_{\rm c, ord}^{q_0}(\mathscr X(K_{\alpha',\alpha});\underline{L}_{\lambda,\OO})_{\mathfrak{m}}^{\pi_0(H(\RR))}\to\OO_{(0)},\quad
  \phi\;\mapsto\;\int_{C_F(p^\infty)}d\mu_{\OO}^{\lambda,0}(\phi),
  $$
  evaluated at the element
  $$
  \xi\;\in\;
   \Hom_\OO(H_{\rm ord}^{q_0+l_0}\left(\mathscr X(K_{\alpha',\alpha});\underline{L}_{\lambda^\vee,\OO}^\circ)_{\mathfrak{m}^\vee}^{\pi_0(H(\RR))},\OO\right)\;=\;
   H_{\rm c, ord}^{q_0}(\mathscr X(K_{\alpha',\alpha});\underline{L}_{\lambda,\OO})_{\mathfrak{m}}^{\pi_0(H(\RR))}/({\rm torsion}).
  $$
  We conclude the proof with the observation that
  $$
  \int_{C_F(p^\infty)}d\mu_{\OO}^{\lambda,0}(\phi)\;=\;
  \widetilde{\Omega}_{\xi,p}\cdot
  \int_{C_F(p^\infty)}d\mu_{\Pi_\xi\widehat{\otimes}\Sigma_\xi}\cdot \eta_{0,\OO}(v_0),
  $$
  with $\widetilde{\Omega}_{\xi,p}\in\OO^\times$, since $\mu_{\Pi_\xi\widehat{\otimes}\Sigma_\xi}$ is only defined up to $p$-integral units, since the same applies to the image of the cohomology class $\phi_{\Pi_\xi\widehat{\otimes}\Sigma_\xi}$ after inverting $p$.
  \end{proof}

To obtain an integral version of Theorem \ref{thm:nonabelianinterpolation} we need to formulate

\begin{conjecture}\label{conj:freeness}
  For any non-Eisenstein maximal ideal $\mathfrak{m}$ in ${\bf h}_{\rm ord}(K_{\infty,\infty};\OO),$
  $$
  \mathcal H_{\rm ord}^{q_0+l_0}(K_{\infty,\infty}; \OO)_{\mathfrak{m}^\vee}^{\pi_0(H(\RR))}
  $$
  is a free module of rank one over ${\bf h}_{\rm ord}^{q_0+l_0}(K_{\infty,\infty};\OO)_{\mathfrak{m}^\vee}$.
\end{conjecture}

\begin{remark}
  In the totally real case with $n=1$, Conjecture \ref{conj:freeness} is known under mild hypotheses on the residual representation $\overline{\rho}_{\mathfrak{m}}$ and $p$, cf.\ Theorem 4.6 in \citep{dimitrov2011}. This result relies on an \lq{}$\mathcal R={\mathbb T}$\rq{} Theorem, which allows allows Dimitrov to interpret his non-abelian $p$-adic $L$-function as an element of a universal deformation ring.
\end{remark}

\begin{remark}
  Hansen-Thorne's Theorem 4.9 in \citep{hansenthorne2017} conditionally implies Conjecture \ref{conj:freeness} up to $\OO$-torsion in the case $F=\QQ$.
\end{remark}

  Following section 5 in \citep{dimitrov2011} and assuming Conjecture \ref{conj:freeness}, we define the universal $p$-adic $L$-function
  $$
    L_{p,\mathfrak{m}}^{\rm univ}\;\in\;{\bf h}_{\rm ord}^{q_0+l_0}(K_{\infty,\infty};\OO)_{\mathfrak{m}^\vee}
  $$
  as the image of $L_p^{0,0}$ under the isomorphism
  $$
    \mathcal H_{\rm ord}^{q_0+l_0}(K_{\infty,\infty}; \OO)_{\mathfrak{m}^\vee}^{\pi_0(H(\RR))}\;\to\;
    {\bf h}_{\rm ord}^{q_0+l_0}(K_{\infty,\infty};\OO)_{\mathfrak{m}^\vee}.
  $$

  Let $\xi^\vee\in\Spec{\rm h}^{q_0+l_0}_{\rm ord}(K_{\alpha',\alpha};\lambda,\OO)_{\mathfrak{m}^\vee}(\OO)$ denote a classical point of weight $\lambda^\vee$ and Nebentypus $\vartheta^{-1}$. Assume that $\eta_0$ is admissible for $\lambda$ (in particular, $\lambda$ is balanced).
    
  By Theorem \ref{thm:controlforlocalizedLp}, the canonical isogeny from Corollary \ref{cor:hordcontrol},
  \begin{align*}
    &{\bf h}_{\rm ord}^{q_0+l_0}(K_{\infty,\infty};\OO)_{\mathfrak{m}^\vee}
    /P_{\lambda^\vee\vartheta^{-1}}{\bf h}_{\rm ord}^{q_0+l_0}(K_{\infty,\infty};\OO)_{\mathfrak{m}^\vee}\\
    &\to\;
    {\rm h}_{\rm ord}^{q_0+l_0}(K_{\alpha',\alpha};\lambda^\vee,\OO)_{\mathfrak{m}^\vee}
    /P_{\vartheta^{-1}}{\rm h}_{\rm ord}^{q_0+l_0}(K_{\alpha',\alpha};\lambda^\vee,\OO)_{\mathfrak{m}^\vee},
  \end{align*}
  composed with $\xi^{-1},$ maps $L_{p,\mathfrak{m}}^{\rm univ}$ onto
  $$
    \xi^{-1}\circ {\lambda^\vee\vartheta^{-1}}(L_{p,\mathfrak{m}}^{\rm univ})\;=\;\Omega_{\xi,p}\cdot
    \int_{C_F(p^\infty)}d\mu_{\Pi_\xi\widehat{\otimes}\Sigma_\xi},
  $$
  where $\Omega_{\xi,p}\in \OO^\times$ is again a $p$-adic period. We obtain

\begin{theorem}\label{thm:nonabelianinterpolation2}
  Assume $p\nmid(n+1)n$, $F$ totally real, CM or Conjecture \ref{conj:galoisrepresentations}. Let $\mathfrak{m}$ denote a non-Eisenstein maximal ideal in  $\mathfrak{m}$ in ${\bf h}_{\rm ord}(K_{\infty,\infty};\OO)$ for which Conjecture \ref{conj:freeness} holds. Then there exists an element\nomenclature[]{$L_{p,\mathfrak{m}}^{\rm univ}$}{$p$-adic $L$-function in ${\bf h}_{\rm ord}^{q_0+l_0}(K_{\infty,\infty};\OO)_{\mathfrak{m}^\vee}$}
    $$L_{p,\mathfrak{m}}^{\rm univ}\;\in\;{\bf h}_{\rm ord}^{q_0+l_0}(K_{\infty,\infty};\OO)_{\mathfrak{m}^\vee}$$
  with the following interpolation property. For every classical point
  $$
    \xi^\vee\in\Spec{\rm h}^{q_0+l_0}_{\rm ord}(K_{\alpha',\alpha};\lambda^\vee,\OO)_{\mathfrak{m}^\vee}(\overline{\OO})
  $$
  of balanced weight $\lambda^\vee$ and Nebentypus $\vartheta^{-1}$, such that $s_0=\frac{1}{2}$ is critical for $L(s,\Pi_\xi\widehat{\otimes}\Sigma_\xi),$ we have
  \begin{align*}
    \Omega_{\xi,p}^{-1}\cdot\xi^\vee\circ {\lambda^\vee\vartheta^{-1}}(L_{p,\mathfrak{m}}^{\rm univ})
    \;=\;&
    \int_{C_F(p^\infty)}d\mu_{\Pi_\xi\widehat{\otimes}\Sigma_\xi}\\
    =\;&
    \absNorm(\mathfrak{f}_{\vartheta})^{\frac{(n+1)n(n-1)}{6}}\cdot
    \prod_{\mu=1}^n\prod_{\nu=1}^\mu
      G(\vartheta_{\mu,\nu})
    \cdot
    \frac{L(\frac{1}{2},\Pi_\xi\widehat{\otimes}\Sigma_\xi)}{\Omega_{\xi}},
  \end{align*}
where the second identity is valid whenever $\vartheta$ has fully supported constant conductor.

  Here $\Omega_{\xi,p}^{-1}\in\OO[\xi]^\times$ is a $p$-adic period, $\Omega_{\xi}\in\CC^\times$ is a complex period and both these periods may be normalized in such a way that they are invariant under twists of $\Pi_\xi\widehat{\otimes}\Sigma_\xi$ by finite order Hecke characters $\chi$ unramified outside $p$.
\end{theorem}

  \begin{remark}\label{rmk:fullLp}
    In general, twisting by finite order characters also twists the residual representation by a character of the finite torsion subgroup $\Delta$ of $C_F(p^\infty)$. However, there are only finitely many such twists and hence only a finite set of corresponding non-Eisenstein maximal ideals $\mathfrak{m}_1,\dots,\mathfrak{m}_{\delta},$ ($1\leq \delta\leq|\Delta[1/p]|$), corresponding to finite order twists with characters unramified outside $p\infty$.
    Let\nomenclature[]{$\mathfrak{n}$}{intersection of maximal ideals $\mathfrak{m}_i$}
    $$
      \mathfrak{n}\;:=\;\bigcap_{i=1}^{\delta}\mathfrak{m}_i
    $$
    Our construction generalizes to Hecke algebras localized at $\mathfrak{n}$, i.e.\ the product of their localizations at each ${\mathfrak{m}}_i$, provided that Conjecture \ref{conj:freeness} holds for every $\mathfrak{m}_i$. By the Chinese Remainder Theorem, we obtain an isomorphism
    $$
      \mathcal H_{\rm ord}^{q_0+l_0}(K_{\infty,\infty}; \OO)_{\mathfrak{n}^{-1}}\;\cong\;
      \OO[\pi_0(H(\RR))]\otimes_\OO{\bf h}_{\rm ord}^{q_0+l_0}(K_{\infty,\infty};\OO)_{\mathfrak{n}^{-1}},
    $$
    which allows us to define a $p$-adic $L$-function
    $$
      L_{p,\mathfrak{n}}^{\rm univ}\;\in\;{\bf h}_{\rm ord}^{q_0+l_0}(K_{\infty,\infty};\OO)_{\mathfrak{n}^{-1}},
    $$
    with the corresponding interpolation property for classical points in
    $$
    \Spec{\rm h}^{q_0+l_0}_{\rm ord}(K_{\alpha',\alpha};\lambda^\vee,\OO)_{\mathfrak{n}^{-1}}(\overline{\OO}).
    $$
      
  \end{remark}
  \begin{remark}
    The interest in considering $L_{p,\mathfrak{n}}^{\rm univ}$ is that it allows us to recover the full abelian $p$-adic $L$-functions from Theorem \ref{thm:interpolation} as follows.

    The universal nearly ordinary Hecke algebra ${\bf h}_{\rm ord}^{q_0+l_0}(K_{\infty,\infty};\OO)_{\mathfrak{n}^{-1}}$ carries a canonical $\OO[[C_F(p^\infty)]]$-algebra structure and any $\OO[[C_F(p^\infty)]]$-algebra homomorphism
    $$
      \Xi:\quad {\bf h}_{\rm ord}^{q_0+l_0}(K_{\infty,\infty};\OO)_{\mathfrak{n}^{-1}}\to\OO[[C_F(p^\infty)]]
    $$
    for which the composition with the trivial character of $C_F(p^\infty)$ yields a classical point $\xi^\vee$ of balanced weight, sends $L_{p,\mathfrak{n}}^{\rm univ}$ to $\mu_{\Pi_\xi\widehat{\otimes}\Sigma_\xi}$.
    Likewise, we can recover $\mu_{\Pi_\xi\widehat{\otimes}\Sigma_\xi}$ from $L_p^{0,0}$ localized at $\mathfrak{n}$.
  \end{remark}

\subsection{Applications to non-vanishing of central $L$-values}

As an application of Theorem \ref{thm:interpolation} we prove

\begin{theorem}\label{thm:abeliannonvanishing}
  Let $F$ be a number field, $\Pi\widehat{\otimes}\Sigma$ be a cuspidal automorphic representation of $G(\Adeles)$ satisfying the hypotheses of Theorem \ref{thm:interpolation}. Assume that $L(s,\Pi\widehat{\otimes}\Sigma)$ admits at least two critical values and that $w_\lambda$ is even. Then the $p$-adic measure $\mu_{\Pi\widehat{\otimes}\Sigma,\lambda,w_\lambda/2}\in \OO[[C_F(p^\infty)]]$ is non-zero. Furthermore, for each $\chi\in\mathscr X_F^0$ the measure
$$
\chi\cdot\mu_{\Pi\widehat{\otimes}\Sigma,\lambda,w_\lambda/2}|_{C_\QQ^{\rm cyc}(p^\infty)}\;\in\;\OO[\chi][[C_F^{\rm cyc}(p^\infty)]]
$$
is non-zero.
\end{theorem}

\begin{proof}
We have an identity
\begin{equation}
\langle\cdot\rangle_F\cdot\omega_F\cdot
\chi\cdot\mu_{\Pi\widehat{\otimes}\Sigma,\lambda,w_\lambda/2}|_{C_F^{\rm cyc}(p^\infty)}
\;=\;
\chi\cdot \mu_{\Pi\widehat{\otimes}\Sigma,\lambda,1+w_\lambda/2}|_{C_F^{\rm cyc}(p^\infty)}
\label{eq:shiftedmeasureidentity}
\end{equation}
of measures on $C_F^{\rm cyc}(p^\infty)$. We know that the measure
$$
\chi\cdot \mu_{\Pi\widehat{\otimes}\Sigma,\lambda,1+w_\lambda/2}|_{C_F^{\rm cyc}(p^\infty)}
$$
is non-zero by Theorem \ref{thm:interpolation}, because the complex $L$-function in the interpolation formula is non-zero whenever the real part of $s$ is larger than $\frac{w_\lambda}{2}+1,$ cf.\ \citep{jacquetshalika1981}. Whence $\chi\cdot\mu_{\Pi\widehat{\otimes}\Sigma,\lambda_\lambda,w_\lambda/2}|_{C_F^{\rm cyc}(p^\infty)}$ is non-zero.
\end{proof}

\begin{corollary}\label{cor:abeliannonvanishing}
For every finite order Hecke character $\chi$ of $F$ we have
$$
L(\frac{1+w_\lambda}{2},\Pi\widehat{\otimes}\Sigma\otimes\chi\chi')\;\neq\;0
$$
for all but finitely many characters $\chi':C_F^{\rm cyc}(p^\infty)\to\CC^\times$.
\end{corollary}

\begin{remark}\label{rmk:nearcentral}
  $\Pi|\cdot|^{-w_{\lambda_{n+1}}}  \widehat{\otimes}\Sigma|\cdot|^{-w_{\lambda_{n}}}$ is unitary and Corollary \ref{cor:abeliannonvanishing} is a non-vanishing statement about the central $L$-value. By \citep{shahidi1981} it is known that near-central critical values are {\em always} non-zero.
\end{remark}

\begin{proof}
Remark that if $\Pi\widehat{\otimes}\Sigma$ satisfies the hypotheses of Theorem \ref{thm:interpolation}, then so does $\Pi\widehat{\otimes}\Sigma\otimes\chi$. Assume without loss of generality that $\chi$ takes values in $E$. By the above theorem,
\begin{align*}
0\neq \mu_{\Pi\widehat{\otimes}\Sigma\otimes\chi,\lambda,w_\lambda/2}|_{C_F^{\rm cyc}(p^\infty)}
\;\in\;&
\OO[[C_F^{\rm cyc}(p^\infty)]]\\
\cong\;&
\OO[[X]].
\end{align*}
By the Weierstrass Peparation Theorem, a non-zero power series in one variable over $\OO$ admits only finitely many zeroes and the claim follows.
\end{proof}

\subsection{Non-vanishing via non-abelian deformations}\label{sec:nonabeliannonvanishing}

As an application of Theorems \ref{thm:interpolation} and \ref{thm:controlforlocalizedLp} we prove (independently of Conjecture \ref{conj:freeness}),

\begin{theorem}\label{thm:nonabeliannonvanishing}
  Let $p\nmid(n+1)n$, $F$ totally real, CM or assume the validity of Conjecture \ref{conj:galoisrepresentations} over $F$. Let $\mathfrak{m}$ denote a non-Eisenstein maximal ideal in  $\mathfrak{m}$ in ${\bf h}_{\rm ord}(K_{\infty,\infty};\OO)$. Assume the existence of a classical point $\xi\in\Spec{\bf h}_{\rm ord}(K_{\infty,\infty};\OO)_{\mathfrak{m}^{-1}}$ of balanced weight such that $L(s,\Pi_{\xi}\widehat{\otimes}\Sigma_\xi)$ admits at least two critical values.

Then the image $L_{p,\mathfrak{m}}$ of $L_p^{0,0}$ in $\mathcal H_{\rm ord}^{q_0+l_0}(K_{\infty,\infty}; \OO)_{\mathfrak{m}^{-1}}^{\pi_0(H(\RR))}$ is non-zero. Moreover, its projection to
\begin{equation}
  \mathcal H_{\rm c,ord}^{q_0+l_0}(K_{\infty,\infty}; \OO)_{\mathfrak{m}^{-1}}^{\pi_0(H(\RR))}
  \otimes_{{\bf h}_{\rm ord}(K_{\infty,\infty};\OO)_{\mathfrak{m}^{-1}},\xi'}\overline{\OO}
  \label{eq:nonzeromodxi}
\end{equation}
is non-zero for some classical point $\xi'\in\Spec{\bf h}_{\rm ord}(K_{\infty,\infty};\OO)_{\mathfrak{m}}$ of balanced weight.
\end{theorem}

\begin{proof}
By Theorem \ref{thm:muindependenceofweight}, we may assume without loss of generality that $\eta_0$ is admissible for the weight $\lambda$ of $\xi$.

We import the notation of Remark \ref{rmk:fullLp}, i.e.\ let $\mathfrak{m}_i$, $1\leq i\leq \delta$, denote the non-Eisenstein maximal ideals corresponding to finite order twists of the residual representation $\overline{\rho}_{\mathfrak{m}}$ by characters unramified outside $p\infty$ and of prime-to-$p$ order.

From the proof of Theorem \ref{thm:abeliannonvanishing} we deduce that the measure $\mu_{\Pi_{\xi}\widehat{\otimes}\Sigma_\xi}$ is non-zero on every translate of $C_F^{\rm cyc}(p^\infty)$ in $C_F(p^\infty)$.

By Theorem \ref{thm:controlforlocalizedLp} and Remark \ref{rmk:reconstruction}, the non-vanishing of $\mu_{\Pi_{\xi}\widehat{\otimes}\Sigma_\xi}$ on every translate of the cyclotomic line implies that the image $L_{p,\mathfrak{m}_i}^{0,0}$ of $L_p^{0,0}$ in $\mathcal H_{\rm c,ord}^{q_0+l_0}(K_{\infty,\infty}; \underline{L}_{\lambda,\OO}^\circ)_{\mathfrak{m}_i^\vee}^{\pi_0(H(\RR))}$ is non-zero for every $1\leq i\leq\delta$ and also non-zero in \eqref{eq:nonzeromodxi} for some cyclotomic twist $\xi'$ of $\xi$ by Theorem \ref{thm:nonabelianinterpolation}.
\end{proof}

\begin{corollary}\label{cor:nonabeliannonvanishing1}
  Under the assumptions of Theorem \ref{thm:nonabeliannonvanishing}, for every dominant weight $\lambda$ and every $\eta_j$ admissible for $\lambda$, the restriction of $\mu_{\alpha}^{\lambda,j}$ to the localization
$$
\mu_{\alpha,\mathfrak{m}}^{\lambda,j}:\quad
H_{\rm c,ord}^{\dim\mathscr Y}({\mathscr X}(K_{\alpha,\alpha}); \underline{L}_{\lambda,p^{-\alpha}\OO/\OO})_{\mathfrak{m}}^{\pi_0(H(\RR))}\to
  (p^{-\alpha}\OO/\OO)_{(j)}[[C(p^\infty)]]
$$
is non-zero for every sufficiently large $\alpha$.
\end{corollary}

\begin{remark}
  Corollary \ref{cor:nonabeliannonvanishing1} implies the existence of non-zero $p$-adic $L$-functions for torsion classes in the absence of classical points.
\end{remark}

\begin{corollary}\label{cor:nonabeliannonvanishing2}
  Under the assumptions of Theorem \ref{thm:nonabeliannonvanishing}, let $\lambda$ denote a balanced cohomological weight for which $\eta_0$ is the (unique) admissible character. Let furthermore
  $$
  \mathcal X\subseteq\Spec{\bf h}_{\rm ord}(K_{\infty,\infty};\OO)_{\mathfrak{m}}(\overline{E})
  $$
  denote an irreducible component containing the classical point $\xi$ from Theorem \ref{thm:nonabeliannonvanishing} with the property that its subset $\mathcal X_{\rm cl}^\lambda\subseteq\mathcal X$ of classical points of balanced weight $\lambda$ is Zariski-dense in $\mathcal X$.

  Then there exists $\xi\in\mathcal X_{\rm cl}^\lambda$ such that $\Pi_\xi\widehat{\otimes}\Sigma_\xi$ is cuspidal and a Hecke character $\chi$ of finite order and unramified outside $p\infty$, satisfying
  $$
    L(\frac{1}{2},\Pi_{\xi}\widehat{\otimes}\Sigma_\xi\otimes\chi)\neq 0.
  $$
  Moreover, 
  $$
    L(\frac{1}{2},\Pi_{\xi}\widehat{\otimes}\Sigma_\xi\otimes\chi\chi')\neq 0
  $$
  for all but finitely many Hecke characters $\chi'\in\mathscr X_F^{\rm cyc}$.
\end{corollary}

\begin{proof}
  Assume to the contrary that 
  $$
    L(\frac{1}{2},\Pi_{\xi}\widehat{\otimes}\Sigma_\xi\otimes\chi)=0
  $$
  for all $\xi\in\mathcal X_{\rm cl}^\lambda$ and all $\chi$ unramified outside $p\infty$. By Theorem \ref{thm:interpolation} this implies
  $$
    \mu_{\Pi_{\xi}\widehat{\otimes}\Sigma_\xi}= 0
  $$
    for all $\xi\in\mathcal X_{\rm cl}^\lambda$. Therefore, the specialization of $L_{p}^{0,0}$ vanishes at all classical $\xi$ of weight $\lambda$ in $\mathcal X$ (cf.\ Theorem \ref{thm:nonabelianinterpolation}).

  By the Zariski-density of $\mathcal X_{\rm cl}^\lambda,$ this implies that the evaluation of $L_{p}^{0,0}$ vanishes at all points in $\mathcal X$. This contradicts Theorem \ref{thm:nonabeliannonvanishing}. The generic non-vanishing statement on the cyclotomic line follows as in the proof of Theorem \ref{thm:abeliannonvanishing}.
\end{proof}

\begin{small}
\printnomenclature
\end{small}

\end{document}